\documentclass[oneside,UTF12]{amsart}
\usepackage[T1]{fontenc}
\usepackage{geometry}
\usepackage{amstext}
\usepackage{amsthm}
\usepackage{amssymb}
\usepackage{esint}
\usepackage{latexsym,amsmath,amssymb,amsthm,amsfonts,arydshln,enumerate,graphicx}
\usepackage{amssymb, amsfonts,arydshln,color}
\usepackage{amsmath}
\usepackage{latexsym}
\usepackage{mathrsfs}
\usepackage{graphicx}
\usepackage{cite}
\usepackage[unicode=true,pdfusetitle,
 bookmarks=true,bookmarksnumbered=false,bookmarksopen=false,
 breaklinks=false,pdfborder={0 0 1},backref=false,colorlinks=false]
 {hyperref}

\makeatletter

\newcommand{\eproof}{\hfill$\square$}

\numberwithin{equation}{section}
\numberwithin{figure}{section}
\theoremstyle{plain}
\newtheorem{thm}{\protect\theoremname}[section]
  \theoremstyle{definition}
  
  \theoremstyle{remark}
  
  \theoremstyle{plain}
  \newtheorem{lem}[thm]{\protect\lemmaname}
  \theoremstyle{plain}
  \newtheorem{prop}[thm]{\protect\propositionname}
  \theoremstyle{plain}
  \newtheorem{cor}[thm]{\protect\corollaryname}

\usepackage{indentfirst}
\usepackage{amsmath}

\makeatother

  \providecommand{\definitionname}{Definition}
  \providecommand{\lemmaname}{Lemma}
  \providecommand{\propositionname}{Proposition}
  \providecommand{\remarkname}{\rm\textbf{Remark}}
\providecommand{\theoremname}{Theorem}
\providecommand{\corollaryname}{Corollary}
\providecommand{\examplename}{Example}
\begin{document}
\title
{Frame set for Gabor systems with Haar window}
\author{Xin-Rong Dai}

\address{[Xin-Rong Dai] School of Mathematics, Sun Yat-sen University, Guangzhou, 510275,
P.R. China }

\email{daixr@mail.sysu.edu.cn}

\author{Meng Zhu}

\address{[Meng Zhu] School of Mathematics, Sun Yat-sen University, Guangzhou, 510275,
P.R. China }

\email{zhum53@mail2.sysu.edu.cn}

\begin{abstract}

We show the full structure of the frame set for the  Gabor system $\mathcal{G}(g;\alpha,\beta):=\{e^{-2\pi i m\beta\cdot}g(\cdot-n\alpha):m,n\in\Bbb Z\}$ with the window being the Haar function $g=-\chi_{[-1/2,0)}+\chi_{[0,1/2)}$. The  strategy of this paper is to introduce the piecewise linear transformation $\mathcal{M}$ on the unit circle, and to  provide a complete characterization of  structures for its (symmetric) maximal invariant sets. This transformation is related to the famous three gap  theorem of Steinhaus  which may be of independent interest.  Furthermore, a classical criterion on Gabor frames is improved, which  allows us to establish  {a} necessary and sufficient condition for the Gabor system $\mathcal{G}(g;\alpha,\beta)$ to be a frame, i.e.,  the symmetric invariant set of the transformation $\mathcal{M}$ is empty.

Compared with the  previous studies, the present paper provides a self-contained environment to study Gabor frames by a new perspective, which includes that the techniques developed here are new  and all the proofs could be understood thoroughly by the readers without reference to the known results in the previous literature.
 \end{abstract}



\maketitle
\tableofcontents{}

\section{Introduction}

\vspace{0.15cm}
 Let $g\in L^2:=L^2(\Bbb R)$ and $\alpha,\beta>0$. The \emph{Gabor system} $\mathcal{G}(g;\alpha,\beta)$ with the \emph{window function} (or \emph{window} for short)  $g$ and \emph{time--frequency parameters} $\alpha$ and $\beta$, is defined by
\begin{equation}\label{eq.Gabor-System}
\mathcal{G}(g;\alpha,\beta):=\{e^{-2\pi im\beta x}g(x-n\alpha): m,n\in\Bbb Z\},
\end{equation}
  and  we call that the Gabor system $\mathcal{G}(g;\alpha,\beta)$ a \emph{Gabor frame}
if it forms a frame for $L^2$, i.e., there exist positive constants $A$ and $B$ such that
\begin{equation}\nonumber
A\| f\|_2^2\le  \sum_{m,n\in \Bbb Z}\mid\langle f,e^{-2\pi im\beta\cdot}g(\cdot-n\alpha)\rangle\mid^2      \le B\| f\|_2^2\ \ \mbox{for all } f\in L^2.
\end{equation}

 By $\Bbb N$, $\Bbb Q_+$ and $\Bbb R_+$ we denote the set of all natural numbers, positive rational numbers and positive real numbers.  For $x,y\in\Bbb Q_+$, we use $\gcd(x,y)$ to denote the greatest positive rational number that divides both $x$ and $y$.  The characteristic function on a set $E$ is denoted by $\chi_E$. In this paper, we give the following full description of the frame set  for the Haar function.

 \begin{thm} \label{main1.01}
 Let $(\alpha,\beta)\in {\mathbb R}_+^2$ and
 $\mathcal{G}(H_{1/2};\alpha,\beta)$ be the Gabor system \eqref{eq.Gabor-System} generated by the Haar function $H_{1/2}=-\chi_{[-\frac12,0)}+\chi_{[0,\frac12)}$. Then the following statements hold.
\begin{enumerate}
\item [\rm(i)]
If $\mathcal{G}(H_{1/2};\alpha,\beta)$ is a  frame for $L^2$, then
\begin{equation} \label{main1.01-1}
\alpha\beta\le1\ \mbox{ and }\ \alpha\le 1.
\end{equation}
\item[\rm(ii)] If  $(\alpha,\beta)$ satisfies \eqref{main1.01-1} and $\alpha\beta\not\in\Bbb Q$, then
$\mathcal{G}(H_{1/2};\alpha,\beta)$ is a frame for $L^2$ if and only if $\beta\not\in 2\Bbb N$.
\item[\rm(iii)]  If $(\alpha,\beta)$ satisfies \eqref{main1.01-1} and $\alpha\beta\in\Bbb Q$,
then $\mathcal{G}(H_{1/2};\alpha,\beta)$ is a frame for $L^2$ if and only if
$$
|\beta-2n| \ge {\gcd(n, \alpha\beta)} 
$$
 and
$$
|1/\alpha-2n|\ge {\gcd(n, 1/\alpha\beta)}-(1/\alpha\beta-1)
$$
for all $n\in \Bbb N$.
\end{enumerate}
\end{thm}


\subsection {Background and   reformulation of Theorem \ref{main1.01}}
The conception of the Gabor frame was originated with two historical problems  on the Gabor system with Gaussian window discussed by von Neumann\cite{neumannbook} in 1932 and Gabor\cite{gabor46} in 1946, and the frame theory introduced by Duffin and Schaeffer \cite{duffintams52} in 1952.
The Gabor system can usually be viewed as a Schr\"{o}dinger representation on the Weyl-Heisenberg group equipped with the Haar measure, and therefore Gabor frame is also known as Weyl-Heisenberg frame, see \cite{dgm86}, \cite{grochenigboo} and \cite{C-heil} and references therein.
A landmark paper\cite{dgm86} was published in 1986, where Daubechies, Grossmann and Meyer proved that  for each pair $(\alpha,\beta)\in {\mathbb R}_+^2$ with $\alpha\beta<1$, there exists  a compactly supported smooth function $g$ such that  $\mathcal{G}(g;\alpha,\beta)$ is a Gabor frame. Since then, Gabor theory has been received great attentions.
The readers may refer to the surveys  by Casazza\cite{casazzasurvey}, Gr\"{o}chenig\cite{grochenigboo} and Heil\cite{heilsurvey} for historical remarks and its interaction to a wide range of mathematical fields, such as operator algebra and complex analysis, and many engineering applications involving time-dependent frequency content.

\vspace{0.2cm}

A fundamental problem in Gabor analysis is to determine the \emph{frame set}
$$
\mathcal{F}(g)=\{(\alpha,\beta)\in\Bbb R^2_+:\ \mathcal{G}(g;\alpha,\beta) \ \ \mbox{is a Gabor frame} \ \}
$$
for a given window function $g\in L^2$.  Well-known results on the frame set include its openness for window functions in the Feichtinger algebra \cite{feichtinger83,feichtinger04}, and the following density requirement
\begin{equation}\nonumber
\mathcal{F}(g)\subset \{(\alpha,\beta)\in\Bbb R^2_+: \alpha\beta\le1\}
\end{equation}
for arbitrary window function $g\in L^2$
\cite{heilsurvey,grochenigbook,daubechiesbook}.
 The first breakthrough on the frame set is that $\mathcal{F}(g)=\{(\alpha,\beta)\in\Bbb R^2_+: \alpha\beta< 1\}$ for the Gaussian window $g=e^{-x^2}$, which was conjectured by Daubechies and Grossmann\cite{dg88} and
confirmed by Lyubarskii \cite{lyubarskii92} and Seip and Wallst\'en\cite{seip92,seipwallsten92} independently. Later  on, Janssen and Strohmer \cite{janssenstrohmer,janssen96,janssenjfa03} proved that the window functions $(e^x+e^{-x})^{-1}, (1+x^2)^{-1}$ and $e^{-|x|}$  share the same frame set with the Gaussian window, and the frame sets for $e^{-x}\chi_{\Bbb R_+}(x)$ and $(1+ix)^{-1}$ are $\{(\alpha,\beta)\in\Bbb R^2_+: \alpha\beta\le1\}$.
A recent significant progress was made by Gr\"{o}chenig, Romero and St\"{o}ckler  \cite{grochenigstockler,grochenigrs}, where they proved that the frame sets for a class of totally positive functions are $\{(\alpha,\beta)\in\Bbb R^2_+: \alpha\beta<1\}$  and indicated that all the window functions mentioned above,  with the exception of the hyperbolic secant window $(e^x+e^{-x})^{-1}$, are either  totally positive functions or their slight modifications. In 2021, Belov, Kulikov and Lyubarskii \cite{bkl} obtained some deep results on the frame sets for Herglotz functions and rational functions.

\vspace{0.2cm}

Gabor frames with compact windows (i.e., window functions are compactly supported) occupy  important positions in Gabor theory due to their natural time-frequency localizations\cite{daubechiesIEEE}, and they are widely studied in recent years, see \cite{chris-kim,chris-kim-1,lemvig-N,grochenigbook} and references therein.
 Among them the known necessary condition for the Gabor system $\mathcal{G}(g;\alpha,\beta)$  with a compact window $g$  being  a frame for $L^2$ is that the support of its $\alpha-$shifts should cover the whole line, and hence 
\begin{equation}\label{frame-set-2}
\mathcal{F}(g)\subset \{(\alpha,\beta)\in\Bbb R^2_+: \alpha\beta\le1 \mbox{ and } \alpha\le |\mbox{supp } g|\},
\end{equation}
where $\mbox{supp } g$ denotes the closed interval with minimal length that contains the support of the window function $g$.
The only compact window that the frame set was totally determined is the characteristic function on an interval.
It was first considered by Janssen\cite{janssen03}, and then, after the contribution of Gu and Han \cite{hangu}, Sun and the first author\cite{D-S} provided a complete characterization to the frame set by 14 subfamilies. 


\vspace{0.2cm}

 As stated in our main result Theorem \ref{main1.01}, the main purpose of this paper is to explore the frame set of the Haar function $g(x)=H_{1/2}(x)$. Throughout this paper, we always denote  $H(x)=-\chi_{[-1,0)}(x)+\chi_{[0,1)}(x)$ and
\begin{equation}\label{haar-function}
H_c(x)=-\chi_{[-c,0)}(x)+\chi_{[0,c)}(x)\ \ \mbox{for all }\ c>0.
\end{equation}
 It is well known that the Haar function  plays special roles in scientific fields, and the Haar wavelet basis generated by the  Haar function has been widely used in many engineering fields related to information science.  Therefore, we naturally expect that the Gabor system generated by the Haar window function should have considerable impact on basic and applied sciences.


\vspace{0.2cm}

 For $g\in L^2$ and scale $c>0$, define
\begin{equation}\nonumber
g_c(x):=g(x/ c),
\end{equation}
and the Gabor system with normalized shift $\beta=1$ by
\begin{equation}\label{new-gabor-sys}
\mathcal{G}(g;\alpha):=\mathcal{G}(g;\alpha,1)=\{e^{-2\pi imx}g(x-n\alpha): m,n\in\Bbb Z\}.
\end{equation}
Observe that for any $c>0$, the Gabor system $\mathcal{G}(g;\alpha,\beta)$ with $(\alpha,\beta)\in\Bbb R_+^2$ is a frame for $L^2$ if and only if
$\mathcal{G}(g_{c};\alpha c,\beta/c)$ is a frame for $L^2$.  Thus,
 by taking $c=\beta/2$ and $a=\alpha\beta$,
the full description of the frame set of the Haar window function $H_{1/2}(x)$  is reduced to finding parameter pairs $(a,c)\in\Bbb R^2_+$ such that the normalized Gabor system $\mathcal{G}(H_c;a)$ is a frame for $L^2$.
Furthermore, according to \eqref{frame-set-2},  the assertion (i) of Theorem \ref{main1.01} holds, and hence it suffices  to consider the parameter pair  $(a,c)$ satisfying
\begin{equation}\label{frame-set-3}
0<a\le \min(1 ,2 c).
\end{equation}
 In view of the facts mentioned above, one can formulate the following statement that is equivalent to the assertions (ii) and (iii) of  Theorem \ref{main1.01}.

\vspace{0.2cm}

For simplicity, we denote a non-negative rational number $x$ by {simple fraction} $p/q$, which means that $p,q$ are co-prime positive integers if $x\not\in\Bbb Z$, and $p=x$ and $q=1$ if $x\in\Bbb Z$.
%
%

\begin{thm} \label{main1}
Let $a,c$ satisfy \eqref{frame-set-3},
 $H_c$ be as in \eqref{haar-function} and $\mathcal{G}(H_c;a)$ be the Gabor system in \eqref{new-gabor-sys}. Then the following statements hold.
\begin{enumerate}
\item[\rm(i)] If 
$a\not\in\Bbb Q$, then
$\mathcal{G}(H_c;a)$ is not a  frame for $L^2$ if and only if $c\in \Bbb N$.

\item[\rm(ii)] If 
$a=p/q$ is a simple fraction,
then $\mathcal{G}(H_c;a)$ is not a  frame for $L^2$ if and only if either
\begin{equation}
\label{main1.eq1}
|c-n| < \frac{\gcd(n, p)}{2q}\ \mbox{ for some }\ n\in\Bbb N
\end{equation}
 or
\begin{equation}\label{main1.eq2}
|c-\frac{np}q|< \frac{\gcd(n, q)}{2q}-\frac{q-p}{2q}\ \mbox{ for some }\ n\in\Bbb N\setminus q\Bbb N.
\end{equation}

\end{enumerate}
\end{thm}

\vspace{0.2cm}

It is not difficult to verify that     \eqref{main1.eq1} and \eqref{main1.eq2} have no intersection since we have got rid of $n\in q\Bbb N$ in \eqref{main1.eq2} which coincides to \eqref{main1.eq1} with $n\in p\Bbb N$.

\vspace{0.2cm}

 Roughly speaking, with the exception of some special cases (see Theorem \ref{main0}), the approach to Theorem \ref{main1} is based on establishing a link between the frame properties for the Gabor system $\mathcal{G}(H_c;a)$ and the symmetric maximal invariant set $\mathcal{E}$ for the piecewise linear transformation $\mathcal{M}$ (see Theorem \ref{thm-gabor-eq-e-set}),   which involves a detailed analysis of the  dynamical  asymptotic behavior of $\mathcal{M}$ and the structure of maximal invariant set $\mathcal{S}$, e.g., see Theorems \ref{thm-s-1}, \ref{thm-mod-1} and Proposition \ref{thm-exist-e++}.
\vspace{0.2cm}

\subsection{Piecewise linear  transformation on the unit circle }\label{subsection 1.2}
  Let $\Bbb T:=\Bbb R/\Bbb Z$ be the unit circle. The famous ``Steinhaus conjecture'' on the  mod 1 distribution said that, for given $\alpha>0$ and  $N\in\Bbb N$, there are at most three arcs of the various lengths on the circle $\Bbb T$  partitioned  by $\{\langle n\alpha\rangle:n=0,1,\ldots, N\}$, where $\langle x\rangle$ denotes the fractional part of number $x$. This conjecture was proved independently by P. Erd\"{o}s, G. Haj\'{o}s, V. T. S\'{o}s, J. Sur\'{a}nyi, S. \'{S}wierczkowski and  P. Sz\"{u}sz, and was well known as the three gap theorem, see \cite{Vos, Suranyi} for details, and also see \cite{HM-norm} and references therein for its recent progresses and extensions.
\vspace{0.2cm}

 Throughout this paper, we will use $[0, 1)$ and also  $(0, 1]$ to represent the unit circle $\Bbb T$, and denote by $\Bbb T_{[0,1)}$ and $\Bbb T_{(0,1]}$ the topological space retaining the topology of $\Bbb T$, in the usual manner.
For $0\le\alpha<1$ and $0\le x_1<x_2\le 1$,  we define the \emph{piecewise linear transformation}  $\mathcal{M}:=\mathcal{M}_{\alpha, x_1,x_2}$ from $\Bbb T_{[0,1)}$ to itself as follows,
\begin{equation}\label{map.def}
\mathcal{M}_{\alpha, x_1,x_2}(t)=\left\{\begin{array} {lll}
\langle t+ \alpha  +x_2-x_1\rangle &
{\rm if} \ t\in [0,x_1)=:U\\
t & {\rm if} \ t\in  [x_1,x_2)=:\textbf{H}\\
\langle t+\alpha\rangle  &
{\rm if} \ t\in [x_2,1)=:V
\end{array}\right.;
\end{equation}
and the \emph{maximal invariant set}  of the map $\mathcal{M}$ is defined to be the maximal set $\mathcal{S}$ for which
\begin{equation}\label{s-set.def}
\mathcal{M}(\mathcal{S})=\mathcal{S}\mbox{ and }\mathcal{S}\subset \Bbb T_{[0,1)}\setminus\textbf{H}.
\end{equation}

%
%
%
%
%
%
%
%
%

\vspace{0.2cm}

Obviously, 
$\mathcal{S}$ consists of the points $t\in \Bbb T_{[0,1)}$ such that the orbit $\{\mathcal{M}^n(t)\}_{n\in \Bbb N}$ will not be absorbed in $\textbf{H}$ when $n$ tends to infinity.
We shall point out that the asymptotic behavior of the points in  $\mathcal{S}$ under the map $\mathcal{M}$ behave like that in the Steinhaus' problem. For instance, if $\mathcal{S}\subset U$ (or $\mathcal{S}\subset V$),  it is clear that $ \mathcal{M}^n(t)=\langle t+n(\alpha+x_2-x_1)\rangle $ (or $ \mathcal{M}^n(t)=\langle t+n\alpha\rangle$) for all $t\in\mathcal{S}$, and we will further show that if 
\begin{equation}\label{SUV.def}
{\mathcal S}\cap U\ne\emptyset\ \mbox{ and }\ {\mathcal S}\cap V\ne\emptyset,
\end{equation}
there will have a bijective map $Y$ from $\mathcal{S}$ to $\Bbb T_{[0,1)}$ such that $Y(\mathcal{M}^n(t))=\langle Y(t)+nY(\alpha)\rangle$ holds for all $t\in\mathcal{S}$, see Theorem \ref{thm-mod-1} below.


\vspace{0.2cm}

An extended topic for the Steinhaus' problem
is the interval exchange transformations, for which considering the transformations on ${[0,1)}$ by exchanging the order of the partition $I_j, j=1,2,\ldots, n$, of ${[0,1)}$ via a substitution  over $\{1,2,\cdots,n\}$.
This topic links to vast intrinsic problems in many subjects of pure mathematics such as number theory, dynamical system, 
harmonic analysis and so on, see \cite{Keane,ferenczi,chaika,Veech,MMY} and references therein for details. 
It is worth noting that the map $\mathcal{M}$ defined in \eqref{map.def} is a 4--intervals exchange transformation 
by redefining $\mathcal{M}(\textbf{H})=\textbf{H}+\alpha-x_1$ and the appropriate division on $\Bbb T_{[0,1)}$. 
 In the topic of the interval exchange transformation, the studies mainly focus on the transformations with the property  ``unique ergodicity'', which leads to $\mathcal{S}=\emptyset$.
 However, in this paper we shall characterize the condition for $\mathcal{S}\ne\emptyset$, as well as its structure, and the property of its {symmetric maximal} invariant subset, and so on.
%

\vspace{0.2cm}

 The following theorem gives us a characterization on the structure of the maximal invariant set $\mathcal{S}$ of the transformation $\mathcal{M}$.

\begin{thm}\label{thm-s-1}
Let $\mathcal{M}$ be the map defined in \eqref{map.def}, and ${\mathcal S}$ be the maximal invariant set satisfying \eqref{s-set.def}. Then there exists {a} non-negative integer $N$ such that
\begin{equation}\label{eq-thm-s-1}
{\mathcal S}=\mathcal{M}^N(\Bbb T_{[0,1)})\setminus \textbf{H}.
\end{equation}
\end{thm}
\vspace{0.2cm}

We might see at once from Theorem \ref{thm-s-1} that ${\mathcal S}$ is composed of no more than $N+1$ left-closed and right-open intervals when ${\mathcal S}\ne\emptyset$.  By squeezing out all the intervals in $\Bbb T_{[0,1)}\setminus {\mathcal S}$ and reconnecting their endpoints from left to right, and then by rescaling,  it makes sense to define a map from $\Bbb T_{[0,1)}$ to itself,
\begin{equation}\label{biject-y.def}
 Y(t)=\big\langle  {|{\mathcal S}\cap [0,t)|}/{|\mathcal{S}|}\big\rangle, \ \ t\in [0,1),
\end{equation}
where $|E|$ means the Lebesgue measure  of set $E\subset[0,1)$.  
The dynamical behavior of $Y$ acting on $\mathcal{S}$ is described as follows.
\vspace{0.2cm}

\begin{thm}\label{thm-mod-1}
Let $\mathcal{M}$ be the map defined in \eqref{map.def} and ${\mathcal S}$ be its maximal invariant set.
If ${\mathcal S}\ne\emptyset$, then the map $Y$ defined by \eqref{biject-y.def} is a bijection from ${\mathcal S}$ to $\Bbb T_{[0,1)}$, and for each $t\in {\mathcal S}$,
\begin{equation}\label{eq-thm-mod-01}
Y(\mathcal{M}(t))=\langle Y(t)+Y(\alpha)\rangle.
\end{equation}
Moreover, $Y(\alpha)$ is irrational only if $\mathcal{S}$ satisfies ${\mathcal S}\cap U\ne\emptyset$ and ${\mathcal S}\cap V\ne\emptyset$.
\end{thm}



We will later know that if $\mathcal{S}$ satisfies \eqref{SUV.def}, then Theorem \ref{thm-mod-1} provides a powerful tool to determine the location and size of the intervals constituting $\mathcal{S}$ by using the squeezing map $Y$, see Subsection \ref{subsection-4-1}. Another powerful tool to characterize the structure of such maximal invariant set $\mathcal{S}$ is the disturbance for the map $\mathcal{M}$, more precisely, for small $\delta$, we call $\mathcal{M}_\delta=\mathcal{M}_{\alpha, x_1+\delta,x_2+\delta}$ the $\delta$-disturbance of $\mathcal{M}.$ Moreover, $\mathcal{M}_\delta$ inherits most of the important properties of $\mathcal{M}$.
Roughly speaking,
if we denote the maximal invariant set and the squeezing map of $\mathcal{M}_\delta$ by $\mathcal{S}_\delta$ and $Y_\delta$, then $Y_\delta(\alpha)=Y(\alpha)$ and $\mathcal{M}_\delta(t)=\mathcal{M}(t)$ for all $t\in \mathcal{S}\cap\mathcal{S}_\delta$; in particular, we have either $\mathcal{S}=\mathcal{S}_\delta+[-\delta, 0]$ or $\mathcal{S}_\delta=\mathcal{S}+[0,\delta]$ when $\delta>0$; see Theorem \ref{thm-4-struct} and Proposition \ref{prop-delta-tur}.

\vspace{0.2cm}

The next objective  is to consider whether there exists a non-empty symmetric invariant subset of $\mathcal{S}$. Let $ {\mathcal{M}}= {\mathcal{M}}_{\alpha, x_1,x_2}$ be the map defined in \eqref{map.def}. We define $\widetilde{\mathcal{M}}:=\widetilde{\mathcal{M}}_{\alpha, x_1,x_2}$ be the \emph{coherent map} of $ {\mathcal{M}}$, that is the piecewise linear map from $\Bbb T_{(0,1]}$ to itself
\begin{equation}\label{map-conv.def}
\widetilde{\mathcal{M}}_{\alpha, x_1,x_2}(t):=\left\{\begin{array} {lll}
\langle t+ \alpha  +x_2-x_1\rangle^\ast &
{\rm if} \ t\in (0,x_1]  =:\widetilde{U},\\
t &
{\rm if} \ t\in  (x_1,x_2]=:\widetilde{\textbf{H}},\\
\langle t+\alpha\rangle^\ast  &
{\rm if} \ t\in (x_2,1]=:\widetilde{V},
\end{array}\right.
\end{equation}
here $\langle x\rangle^\ast:=1-\langle 1-x\rangle$.   The \emph{symmetric maximal invariant set} of $\mathcal{M}$ is defined to be the maximal set $\mathcal{E}$ satisfying
\begin{equation}\label{e-set.def}
\mathcal{E}\subset \Bbb T_{[0,1)}\setminus \big(\textbf{H}\cup (1-\widetilde{\textbf{H}})\big),\ \ \mathcal{M} (\mathcal{E})= \mathcal{E} \ \mbox{ and }\   \widetilde{\mathcal{M}} (1-\mathcal{E})= 1-\mathcal{E}.
\end{equation}

By virtue of the structure of $\mathcal{S}$ and the disturbance property of $\mathcal{M}$, we have the following proposition, which provides a full description on the symmetric maximal invariant set $\mathcal{E}\ne \emptyset$ and its structure. This proposition, together with Theorem \ref{thm-gabor-eq-e-set} in the next subsection, will play the core role in the proof of Theorem \ref{main1}.

\begin{prop}\label{thm-exist-e++}
Let $\mathcal{M}$ be the map defined in \eqref{map.def}, and $\mathcal{E}$ be the symmetric maximal invariant set of $\mathcal{M}$. Assume that $\mathcal{E}\ne\emptyset $. Then $Y(\alpha)\in\Bbb Q$. Moreover, if we write $Y(\alpha) =\frac NM$ as the simple fraction, then one of the following statements is valid:
\begin{enumerate}
   \item [\rm(i)]$x_2<\frac 1{2M}$ and $ \alpha=\frac NM$.
   \item [\rm(ii)] $x_1>1-\frac 1{2M}$ and $\alpha=\langle\frac NM-(x_2-x_1)\rangle$. 
   \item [{\rm(iii)}] $M>1$ 
   and there exist
   $0<\Delta<\frac 1{M-1}$ and $|\delta|<\frac 1{2M}-\frac {M-1}{2M}\Delta$, such that
   \begin{equation}\label{alpha-E-U-V}
   x_1=\frac {M-N}M-\frac NM\Delta+\delta,\ x_2=\frac {M-N}M(1+\Delta)+\delta\ \mbox{ and }\ \alpha=\frac NM-\frac {M-N}M\Delta.
   \end{equation}
\end{enumerate}

Conversely, if
$\alpha$, $x_1$ and $x_2$ satisfy either one of the assertions \rm{(i), (ii) and (iii)} for some integers $0\le N< M$ such that $\frac NM$ is a simple fraction, then $\mathcal{E}\ne\emptyset$ and $Y(\alpha) =\frac NM$. Furthermore,
\begin{equation}\label{E-structu}
\mathcal{E}=\cup_{k=0}^{M-1} I_k,
\end{equation}
where $I_k=\frac kM+[x_2,\frac 1M-x_2)$, $I_k=\frac kM+[1-x_1,x_1-\frac {M-1}M)$ and $I_k=\frac kM(1+\Delta)+[|\delta|,\frac {1-(M-1)\Delta}M-|\delta|)$, $k=0,1,\ldots, M-1$, corresponding to  {\rm{(i), (ii) and (iii),}} respectively.
\end{prop}

The statements  (i), (ii) and (iii)  above are corresponding to the cases $\mathcal{S}\subset V$, $\mathcal{S}\subset U$ and $\mathcal{S}$ satisfying \eqref{SUV.def}, respectively. 
Especially, for $\mathcal{S}$ satisfying \eqref{SUV.def}, there are two parameters $\Delta=x_2-x_1$ and $\delta$, 
 by the disturbance of $\mathcal{M}$, we may focus our attention on the assumption that $\delta=0$. In this time, if $\mathcal{E}\ne \emptyset$, then
 $x_2=1-\alpha$ and  $|\mathcal{S}|=1-(M-1)\Delta$. Furthermore, 
 $\mathcal{E}=\mathcal{S}$ is composed of $M$ intervals $I_k,k=0,1,\ldots, M-1,$ with length $|I_k|=\frac {1-(M-1)\Delta}M$, and $\Bbb T_{[0,1)}\setminus\mathcal{E}$ is composed of $M-1$ intervals (gaps) with length $\Delta$, i.e.,
 $\Bbb T_{[0,1)}$ is partitioned alternately by the $M$ intervals and the $M-1$ gaps.
See Lemma \ref{lem-suv} and Proposition \ref{prop-tur-E}.

\vspace{0.2cm}

From \eqref{E-structu} in Proposition \ref{thm-exist-e++}, it is easy to see that $\mathcal{M}(I_k)=I_{k+N}$ for all $k\ge 0$, where we denote $I_k=I_p$ if $k\equiv p$ (mod $M$). Therefore, we have the following corollary, which indicates that $\mathcal{E}$ is "symmetric", i.e., $\pi(\mathcal{E})=1-\mathcal{E}$, and $1-\widetilde{\mathcal{M}}(1-\cdot)$ is an inverse map of $\mathcal{M}$ on $\mathcal{E}$.
Here and hereafter,  if $J=[c,d)\subset [0,1)$, we define
 \begin{equation}\label{eq.pi}
   \pi(J):=(c,d]\subset(0,1];
\end{equation}
and if $E=\cup_{k=1}^n J_k$ is a union of finitely many left-closed and right-open intervals $J_k\subset [0,1)$, we define $\pi(E)=\cup_{k=1}^n \pi(J_k)$.

\begin{cor}\label{thm-exist-e-2}
Let $\mathcal{M}$ be the map defined in \eqref{map.def}. If ${\mathcal E}\ne\emptyset$,
 then
\begin{equation}\label{eq-sym-e-set}
\mathcal{E}=\mathcal{S}\cap (1-\pi(\mathcal{S})),
\end{equation}
and for each $t\in \mathcal{E}$,
\begin{equation}\label{eq-map-invers}
\mathcal{M}(1-\widetilde{\mathcal{M}}(1-t))=1-\widetilde{\mathcal{M}}(1-\mathcal{M}(t))=t.
\end{equation}
\end{cor}

\vspace{0.2cm}

We remark that the condition $\mathcal{E}\ne \emptyset$ in Corollary \ref{thm-exist-e-2} is very strict for the map $\mathcal{M}$ as ${\mathcal E}=\emptyset$ occurs in many cases that the right hand side of \eqref{eq-sym-e-set} is not empty, especially for all the maps with $Y(\alpha)\not\in\Bbb Q$.


\vspace{0.2cm}

\subsection{Connections between transformation $\mathcal{M}_{\alpha,x_1,x_2}$ and Gabor system $\mathcal{G}(H_c;a)$}\label{subsection 1.3}
 A surprising fact is that the frame properties for $\mathcal{G}(H_c;a)$ are intimately  connected to the dynamical properties of the  transformation  $\mathcal{M}_{\alpha, x_1,x_2}$. With the help of {Theorem \ref{main0}}, excluding some special cases, we focus on the parameter pair $(a,c)$ satisfying that
\begin{equation}\label{a-c-condition-1.3}
0<a<1<c \mbox{ and  } c\not\in\Bbb N,
\end{equation}
and let $x_1,x_2$ and $\alpha$ be as follows,
\begin{equation}\label{alpha-xx}
x_1=\max\Big(\frac {a+\langle c\rangle-1}{a}, 0\Big), \
\ x_2=\min\Big(\frac {\langle c\rangle}{a}, 1\Big)
\end{equation}
and
\begin{equation}\label{alpha-xx-0}
\alpha=\big\langle\frac {\lfloor c\rfloor}{a}\big\rangle \mbox{ if } x_2\ne 1  \mbox{  and } \alpha=\big\langle
\frac c{a}\big\rangle \mbox{ if } x_2= 1,
\end{equation}
where $\lfloor x\rfloor:=x-\langle x\rangle$ is the  greatest integer less than or equal to $x$.
Their connections are stated as the following Theorem \ref{thm-gabor-eq-e-set}, whose proof involves an improvement of the Ron-Shen's criterion (Theorem \ref{bound-sequence-0}).

\vspace{0.2cm}

\begin{thm}\label{thm-gabor-eq-e-set}
Let $a,c$ be as in \eqref{a-c-condition-1.3}, and $\mathcal{M}=\mathcal{M}_{\alpha, x_1,x_2}$  be the map defined as \eqref{map.def} with  $x_1,x_2$ and $\alpha$ defined in \eqref{alpha-xx} and \eqref{alpha-xx-0}. Suppose that $\mathcal{E}$ is the symmetric maximal invariant set of $\mathcal{M}$ defined as \eqref{e-set.def}. Then the Gabor system  $\mathcal{G}(H_c;a)$ is a frame for $L^2$ if and only if $\mathcal{E}=\emptyset$.
\end{thm}

\vspace{0.2cm}

We remark that the relationship between $\mathcal{M}_{\alpha, x_1,x_2}$ and $\mathcal{G}(H_c;a)$ in Theorem \ref{thm-gabor-eq-e-set} 
may be extended to more general situations, after redefining $x_1,x_2$ and $\alpha$ as appropriate. We refer the readers to see Theorem \ref{main0} for the rest cases and its proof in the appendix.
At last, as a  complement of Theorem \ref{thm-gabor-eq-e-set}, we have the following lemma on the connection between $a$ and $Y(\alpha)$.

\vspace{0.2cm}

\begin{lem}\label{lem-Y-map}
Let $a,c$ and $\mathcal{M}$ be as in Theorem \ref{thm-gabor-eq-e-set}.
Suppose $\mathcal{S}\ne \emptyset$, then $Y(\alpha)\in\Bbb Q$ if and only if $a\in\Bbb Q$.
\end{lem}

\vspace{0.2cm}

\subsection{Organization of the rest of the paper}
\ \

In Section \ref{prof-main1} we prove Theorem \ref{main1} by using Proposition \ref{thm-exist-e++}, Theorem \ref{thm-gabor-eq-e-set}, Lemma \ref{lem-Y-map}, and Theorem \ref{main0} for some special cases.

In Section \ref{section-2} we are  devoted to  describe the pointwise linear dependence of the $a-$shift of the window function, and to provide  a criterion that  can be used to determine whether a Gabor system with compact window is frame or not, see Theorem \ref{bound-sequence-0}.

In Section 4 we investigate the structure of the maximal invariant sets. We first prove Theorems \ref{thm-s-1} and \ref{thm-mod-1}, and then provide the structures of $\mathcal{S}$ for different situations (Proposition \ref{prop-s-rational} and Theorem \ref{thm-4-struct}) and the relationship between them (Proposition \ref{prop-delta-tur}).

In Section \ref{section-4+1} we focus on the symmetric maximal invariant sets. After characterizing the conditions for  $\mathcal{E}\ne\emptyset$ (Propositions \ref{thm-exist-e-3-1}, \ref{prop-tur-E} and Lemma \ref{lem-suv}), we prove Proposition \ref{thm-exist-e++} and Corollary \ref{thm-exist-e-2}.

In Section \ref{section-5} we give the proofs of Lemma \ref{lem-Y-map} and Theorem \ref{thm-gabor-eq-e-set}.

In Section \ref{section-7}, we prove Theorem \ref{main0}.

\vspace{0.2cm}

\section{Proof of the main result}\label{prof-main1}  

 In this section, we prove Theorem \ref{main1}, the main result of this paper.
Firstly in the following theorem, we discuss some special cases of pairs $(a, c)$
whether $\mathcal{G}(H_c;a)$ is or is not a Gabor frame, which helps us to focus on the challenging pair of parameters $a,c$  with
 \begin{equation}\label{a-c-condition-006}
0<a<1<c  \mbox{ and }\ c\not\in\Bbb N.
\end{equation}

\begin{thm} \label{main0}
Let $(a,c)\in {\mathbb R}_+^2$ and
$\mathcal{G}(H_c;a)$ be the Gabor system \eqref{new-gabor-sys} generated by the Haar function $H_c$ in \eqref{haar-function}. Then the following statements hold.
\begin{enumerate}
   \item [{\rm(i)}]  If $c\in\Bbb N$, then $\mathcal{G}(H_c;a)$ is not a Gabor frame.
   \item [{\rm(ii)}] If $a=1$ and $c\ge\frac 12$, then $\mathcal{G}(H_c;a)$ is a Gabor frame if and only if $ c=\frac {2n-1}2$ for some $n\in\Bbb N$.
   \item [{\rm(iii)}] If $a<1$, $a\not\in\Bbb Q$ and $\frac a2\le c< 1$, then $\mathcal{G}(H_c;a)$ is
   a Gabor frame.
   \item [{\rm(iv)}] If $a<1$, $a=\frac pq$ with co-prime $p,q\in\Bbb N$ and $\frac a2\le c< 1$, then $\mathcal{G}(H_c;a)$ is a Gabor frame if and only if $c\le 1-\frac 1{2q}$.
  \end{enumerate}
\end{thm}



 Theorem \ref{main0} follows from Theorem \ref{bound-sequence-0}, the improvement of the Ron-Shen's criterion,
 and its Corollary \ref{bound-sequence-cor}. As it is independent on the core part of this paper, we postpone the proof Theorem \ref{main0}
 to the appendix in Section \ref{section-7}. The readers may go directly to the appendix without understanding the arguments of the next few sections.

\vspace{0.2cm}
Now we assume that
Proposition \ref{thm-exist-e++}, Theorem \ref{thm-gabor-eq-e-set}, Lemma \ref{lem-Y-map}  and Theorem \ref{main0} hold, and  start to verify the conclusions in  Theorem \ref{main1} case by case.

\vspace{0.2cm}

\noindent{\bf Proof of Theorem \ref{main1}.} 
%
%
(i). 
The sufficiency is an immediate consequence of Theorem \ref{main0} (i). For the necessity, we assume on the contrary that $c\not\in\Bbb N$. By Theorem \ref{main0} (iii), it is enough to consider $a$ and $c$ satisfying  \eqref{a-c-condition-006}. Observing that $\mathcal{S}=\emptyset$ clearly means that $\mathcal{E}=\emptyset$. Also it follows from Lemma \ref{lem-Y-map} that $\mathcal{S}\ne\emptyset$ implies that $Y(\alpha)\not\in \Bbb Q$, and hence $\mathcal{E}=\emptyset$ by Proposition \ref{thm-exist-e++}. Therefore,  using Theorem \ref{thm-gabor-eq-e-set}, $\mathcal{G}(H_c;a)$ is a Gabor frame. This contradiction implies the desired result.
\vspace{0.2cm}

\vspace{0.2cm}

(ii).
 In terms of the assertions  (i), (ii) and (iv) of Theorem \ref{main0},  the last task left for us is to prove the assertion   (ii) of  Theorem \ref{main1}  for the parameter pair $(a,c)$ satisfying \eqref{a-c-condition-006}. Let $x_1$, $x_2$ and $\alpha$ be as in \eqref{alpha-xx} and \eqref{alpha-xx-0}.


\vspace{0.2cm}

\emph{Necessity.} By Theorem \ref{thm-gabor-eq-e-set}, the symmetric maximal invariant set $\mathcal{E}$ of the map $\mathcal{M}=\mathcal{M}_{\alpha, x_1,x_2}$ is not empty.
Then  one of the statements (i), (ii) and (iii) of Proposition \ref{thm-exist-e++} holds with $Y(\alpha)=\frac NM$ being the simple fraction for some integers $0\le N< M$.

\vspace{0.2cm}

If Proposition \ref{thm-exist-e++} (i) holds, then 
$$\alpha=\langle\frac {\lfloor c\rfloor}{a}\rangle=\langle\frac {q\lfloor c\rfloor}{p}\rangle=\frac NM  \text{ and }
 x_2=\frac {\langle c\rangle}{a}=\frac {q\langle c\rangle}{p}<\frac 1{2M}.$$ Therefore, $p=M\gcd(\lfloor c\rfloor,p)$ and $\langle c\rangle<\frac {\gcd(\lfloor c\rfloor,p)}{2q}$. Hence, \eqref{main1.eq1} is valid by taking $n=\lfloor c\rfloor$.

\vspace{0.2cm}

 If Proposition \ref{thm-exist-e++} (ii) holds, then
$$\langle\alpha+x_2-x_1\rangle=\langle\frac {\lfloor c+1\rfloor}{a}\rangle=\langle\frac {q\lfloor c+1\rfloor}{p}\rangle=\frac NM \text{ and }
x_1=1-\frac {1-\langle c\rangle}{a}=1-\frac {q(1-\langle c\rangle)}{p}>1-\frac 1{2M}.$$ Therefore, $p=M\gcd(\lfloor c+1\rfloor,p)$  and  $1-\langle c\rangle<\frac {\gcd(\lfloor c+1\rfloor,p)}{2q}$. Hence, \eqref{main1.eq1} holds for $n=\lfloor c+1\rfloor$.

\vspace{0.2cm}

  If Proposition \ref{thm-exist-e++} (iii) holds, then it is obvious that
  $N<M$ are co-prime positive integers and $0<x_1<x_2<1$, and hence it follows from \eqref{alpha-xx} and \eqref{alpha-xx-0} that
  \begin{equation}\label{x-E-U-V}
 x_2=\frac {\langle c\rangle}{a},\ \Delta=x_2-x_1=\frac{1-a}a\ \mbox{ and }\ \alpha=\big\langle\frac {\lfloor c\rfloor}{a}\big\rangle.
 \end{equation}
Substituting above $\Delta$ and $\alpha$ into \eqref{alpha-E-U-V}, we obtain
 \begin{equation}\nonumber
  \big\langle\frac {\lfloor c\rfloor}{a}\big\rangle= \alpha=\frac NM-\frac {M-N}M\Delta
  =1-\frac {M-N}{aM},
 \end{equation}
which means $\lfloor c\rfloor+\frac{M-N}M= an=\frac {n p}q$ for some $n\in\Bbb N$, and thus $M=q/\gcd(n, q)$ and $n\not\in q\Bbb N$. 
 Recall that
$
  \big\langle\frac {\lfloor c\rfloor}{a}\big\rangle+\frac {\langle c\rangle}{a}=\alpha+x_2
  =1+\delta
$ by \eqref{alpha-E-U-V} and \eqref{x-E-U-V}.
Therefore
$$
|\frac ca-n|=|\frac {\lfloor c\rfloor}{a}+\frac {\langle c\rangle}{a}-n|=|\big\langle\frac {\lfloor c\rfloor}{a}\big\rangle+\frac {\langle c\rangle}{a}-1|=|\delta|<\frac 1{2M}-\frac {M-1}{2M}\Delta=\frac{\gcd(n, q)}{2qa}-\frac{1-a}{2a}.
$$
Hence, \eqref{main1.eq2}  holds. The necessity is proved.

\vspace{0.2cm}

\emph{Sufficiency}. First, we assume \eqref{main1.eq1} is {satisfied.}
Denote $M=p/\gcd(n,p)$ and $N=M\langle nq/p\rangle$. 
Then $N/M$ is a simple fraction.
 If $c\ge n$, then $n=\lfloor c\rfloor$ and $\langle c\rangle=c-n<\frac{\gcd(n, p)}{2q}$. Therefore, by \eqref{alpha-xx} and \eqref{alpha-xx-0},
 $$x_2=\frac {\langle c\rangle}{a}<\frac{\gcd(n, p)}{2p} =\frac 1{2M}$$
 and
$$
\alpha=\big\langle\frac {\lfloor c\rfloor}{a}\big\rangle=\langle\frac {nq}p\rangle=\frac NM.
$$
Hence, the requirement (i) of Proposition \ref{thm-exist-e++} is satisfied for $x_1, x_2$ and $\alpha$, and thus $\mathcal{E}\ne\emptyset$. Consequently, by Theorem \ref{thm-gabor-eq-e-set}, $\mathcal{G}(H_c;a)$ is not a Gabor frame.


\vspace{0.2cm}

If $c< n$, then $n=\lfloor c\rfloor+1$ and $1-\langle c\rangle=|c-n|<\frac{\gcd(n, p)}{2q}$.
Thus, by \eqref{alpha-xx} and \eqref{alpha-xx-0},
 \begin{equation}\nonumber
 x_1=\frac {a+\langle c\rangle-1}{a}>1-\frac{\gcd(n, p)}{2p}=1-\frac 1{2M}
 \end{equation}
and
 \begin{equation}\nonumber
\langle\alpha+x_2-x_1\rangle=\big\langle\frac {\lfloor c\rfloor+1}{a}\big\rangle=\frac NM. 
 \end{equation}
This means $x_1, x_2$ and $\alpha$ satisfy the requirement (ii) of Proposition \ref{thm-exist-e++}, and therefore $\mathcal{E}\ne\emptyset$ and $\mathcal{G}(H_c;a)$ is not a Gabor frame by Theorem \ref{thm-gabor-eq-e-set}.

\vspace{0.2cm}

Next, we assume \eqref{main1.eq2} is satisfied.
Set $M=q/\gcd(n,q)$.
{Because} 
$n\not\in q\Bbb N$, we have $M>1$ and $\langle na\rangle=\frac {M-N}M$ for some $N\in\{1,\cdots,M-1\}$ co-prime to $M$.
Observe that the right hand side of \eqref{main1.eq2} satisfies
$$
\frac{\gcd(n, q)}{2q}-\frac{(q-p)}{2q}=\frac 1{2M}-\frac{1-a}2<\frac1{2M}.
$$
Therefore, 
we have
$\lfloor c\rfloor=\lfloor na\rfloor$
and
\begin{equation}\label{add-601988}
|\langle c\rangle-\frac {M-N}M|=|c-na|<\frac 1{2M}-\frac{1-a}2.
\end{equation}
This means $a>\frac{M-1}M$ and $1-a<\langle c\rangle<a$,
and therefore {\eqref{alpha-xx} and \eqref{alpha-xx-0} yield that}
\begin{equation}\label{add-6011001}
x_2=\frac {\langle c\rangle}{a},\  x_1=\frac {a+\langle c\rangle-1}{a} \mbox{ and }\ \alpha=\big\langle\frac {\lfloor c\rfloor}{a}\big\rangle.
\end{equation}

Set $\Delta=\frac {1-a}a$ and $\delta=\frac{\langle c\rangle}a-\frac {M-N}{aM}$. { It is obvious that $0<\Delta<\frac 1{M-1}$} and $|\delta|<\frac 1{2M}-\frac {M-1}{2M}\Delta$ by \eqref{add-601988}, and thus \eqref{add-6011001} implies that
$$x_2=\frac {\langle c\rangle}{a}=\frac {M-N}{aM}+\delta=\frac {M-N}M(1+\Delta)+\delta,$$
$$x_1=
x_2-\Delta=\frac {M-N}M-\frac NM\Delta+\delta$$
and
$$
\alpha=\big\langle\frac {\lfloor c\rfloor}{a}\big\rangle=\big\langle\frac {\lfloor na\rfloor}{a}\big\rangle=\big\langle\frac {na-\langle na\rangle}{a}\big\rangle
=\langle 1- \frac{M-N}{aM} \rangle
=\frac NM-\frac {M-N}M\Delta.$$
Therefore the requirement (iii) of Proposition \ref{thm-exist-e++} is valid, and hence $\mathcal{E}\ne\emptyset$. Consequently, $\mathcal{G}(H_c;a)$ is not a Gabor frame by Theorem \ref{thm-gabor-eq-e-set}. The sufficiency of (ii) is proved, and the proof of Theorem \ref{main1} is complete.
\eproof


\vspace{0.2cm}

\section {The Ron-Shen's characterization and its improvements}\label{section-2}
Our journey to the preparatory works for full description of the frame set for Haar windows in Theorem \ref{main1}  starts from
the Ron-Shen's characterization in \cite{RS97}, which can be recast
as follows in our setting.
\begin{prop}\label{discretization-theorem-0}
Let $g\in L^2$ and $a>0$. Then $\mathcal{G}(g;a)$ is a Gabor frame  if and only if
   there exist $A,B>0$ such that for each $\{c_j\}_{j\in\Bbb Z} \in \ell^2$
   \begin{equation}\label{eq-discretization-0}
A \|\{c_j\}_{j\in\Bbb Z}\|^2_{\ell^2}\le \sum_{n\in\Bbb Z}\Big|\sum_{j\in\Bbb Z} c_j g(x+j- na)\Big|^2\le B \|\{c_j\}_{j\in\Bbb Z}\|^2_{\ell^2}\  \mbox{ a.e. }\ x\in\Bbb R.
\end{equation}
\end{prop}

Let $W(\mathcal{C},\ell^1):=\{g\in \mathcal{C}:\sum_{k\in\Bbb Z}\sup_{x\in[0,1)}|g(x+k)|<\infty\}$ be the Wiener amalgam space. Using a technique combining
 the Ron-Shen's characterization,
Gr\"{o}chenig, Romero and St\"{o}ckler \cite{grochenigrs} combined some tricks in functional analysis, operator theory and complex analysis, and they
 presented a very strong criterion that the Gabor system $\mathcal{G}(g;a)$ generated by a window function $g$ in the Wiener amalgam space $W(\mathcal{C},\ell^1)$ is not a Gabor frame if and only if there exist $t_0\in\Bbb R$ and a nonzero bounded sequence $\textbf{0} \ne\{d_j\}_{j\in\Bbb Z}\in\ell^\infty$ such that
\begin{equation}\label{eq-rs-1}
\sum_{j\in\Bbb Z} d_j g(t_0+j-na)=0 \mbox{ for all } n\in\Bbb Z.
\end{equation}
Unfortunately, the Haar function $H$ does not belong to the Wiener amalgam space $W(\mathcal{C},\ell^1)$ and then the above bounded linear dependence does not apply directly. Observe that the Haar function is piecewise continuous. To describe the frame set of Haar function, we make a slight modification to the criterion \eqref{eq-rs-1}, through replacing the window function $g$ by its left/right limits $g^\pm(t):=\lim_{\tilde{t}\rightarrow t^\pm} g(\tilde{t})$ in the following theorem. 

\vspace{0.2cm}

\begin{thm} \label{bound-sequence-0}
Let $a>0$, and $g$ be a compactly supported function such that its left/right limits $g^\pm$ are well defined.  Then $\mathcal{G}(g;a)$ is not a Gabor frame if and only if
 there exist $t_0\in\Bbb R$ and a nonzero bounded sequence $\{q_j\}_{j\in\Bbb Z}\in \ell^\infty$ such that
\begin{equation}\label{eq-linear-4.2}
\sum_{j\in\Bbb Z} q_j h(t_0+j-na)=0\  \mbox{ for all }\ n\in\Bbb Z,
\end{equation}
for either $h=g^+$ or $h=g^-$.
\end{thm}


As the requirement \eqref{eq-linear-4.2} is $a$-shift invariance, we may replace $t_0\in\Bbb R$ by $t_0\in [0,a)$ in Theorem \ref{bound-sequence-0}. Observe that $\lim_{x\rightarrow x_0^+}H(x)=H(x_0)$ and $\lim_{x\rightarrow x_0^-}H(x)=-H(-x_0)$. Therefore we have the following corollary. 


\begin{cor} \label{bound-sequence-cor}
$\mathcal{G}(H_c;a)$ is not a frame for $L^2$, if and only if
 there exist $t_0\in [0,a)$ and a nonzero bounded sequence $\{q_j\}_{j\in\Bbb Z}\in \ell^\infty$  
 such that
\begin{equation}\label{eq-linear-cor}
\sum_{j\in\Bbb Z} q_j H_c(t_0+j-na)=0\  \mbox{ for all }\ n\in\Bbb Z.
\end{equation}
\end{cor}

To prove Theorem \ref{bound-sequence-0}, we need some notations and the following technical lemma.
Given a sequence $Q=\{q_n\}_{n\in\Bbb Z}$ and $E\subset\Bbb R$, we denote $Q_E=\{q_n,n\in E\}$ and also $\{q_n\chi_E(n)\}_{n\in\Bbb Z}$ if there is no confusion, and define
\begin{equation}\label{eq-Lambda}
\Lambda_{N}=\Big\{n\in\Bbb Z: |q_n|\ge \frac {\|Q_{[n-N,n+N]}\|_\infty}2\Big\}, \ \ N\in\Bbb N.
\end{equation}

\vspace{0.2cm}

\begin{lem} \label{lem-ell2-contral}
Let $Q=\{q_n\}_{n\in\Bbb Z}\in\ell^2$ and $\Lambda_{N}, N\ge 1$ be as in \eqref {eq-Lambda}. Then 
\begin{equation}\label{ell2-contral}
\|Q_{\Lambda_{N}}\|_{\ell^2}^2 \ge \frac 3{32N} \|Q \|_{\ell^2}^2.
\end{equation}
\end{lem}

\begin{proof}
For each $k\in\Bbb Z$, let $L(k)=\sup\{j\in \Lambda_{N}\cup\{-\infty\}: j\le k\}$ and
$R(k)=\inf\{j\in \Lambda_{N}\cup\{+\infty\}: j\ge k\}$ be the nearest elements in $\Lambda_{N}\cup\{\pm\infty\}$ next to the left and right of $k$ respectively. We make the following claim:
\begin{equation}\label{eq-distence2}
|q_k|\le \max\{2^{-\ell(k)+1}|q_{L(k)}|,2^{-r(k)+1}|q_{R(k)}|\}, k\in\Bbb Z,
\end{equation}
where $ \ell(k)=\lfloor({k-L(k)})/N\rfloor$ and $r(k)=\lfloor( {R(k)-k})/N\rfloor$.

\vspace{0.2cm}


As the conclusion in \eqref{eq-distence2} is trivial for $k\in\Lambda_N$, we may assume that $k\not\in \Lambda_N$. Therefore, without loss of generality, there exists  $k_1$ such that
\begin{equation}\nonumber
|q_k|<\frac{\|Q_{[k-N,k+N]}\|_\infty}2=\frac{|q_{k_1}|}2 \  \mbox{ and }\ k<k_1\le k+N.
\end{equation}
For the case that $k_1\not\in\Lambda_N$, we can find an integer $k_2$ such that
\begin{equation}\nonumber
|q_{k_1}|<\frac{\|Q_{[k_1-N,k_1+N]}\|_\infty}2=\frac{\|Q_{[k_1,k_1+N]}\|_\infty}2=\frac{|q_{k_2}|}2 \  \mbox{ and }\ k_1<k_2\le k_1+N.
\end{equation}
Repeating the above procedure and using $\lim_{n\rightarrow\pm\infty}q_n=0$, we can find integers
$k_0(=k),k_1,\ldots,k_{j_0}\not\in\Lambda_{N}$ and $k_{j_0+1}\in\Lambda_{N}$ such that
\begin{equation}\label{eq-lem-21-3}
|q_{k_j}|<\frac{\|Q_{[k_j-N,k_j+N]}\|_\infty}2=\frac{\|Q_{[k_j,k_j+N]}\|_\infty}2=\frac{|q_{k_{j+1}}|}2 \  \mbox{ and }\ k_j<k_{j+1}\le k_j+N
\end{equation}
for all $ 0\le j\le j_0$. Therefore there exists $0\le j_1\le j_0$ such that
\begin{equation}\nonumber
{k_{j_1}}<R(k)\le k_{j_1+1}.
\end{equation}
This together with 
\eqref{eq-lem-21-3} implies that
\begin{equation}\label{eq-lem-21-5}
r(k)=\big\lfloor \frac{R(k)-k}N\big\rfloor \le \big\lfloor \frac{k_{j_1+1}-k}N\big\rfloor
 \le j_1+1,
\end{equation}
and
\begin{equation}\label{eq-lem-21-6}
|q_k|<
\frac{|q_{k_{j_1+1}}|}{2^{j_1+1}}\le \frac{\|Q_{[R(k)-N,R(k)+N]}\|_\infty}{2^{j_1+1}}\le \frac{|q_{R(k)}|}{2^{j_1}}.
\end{equation}
Combining \eqref{eq-lem-21-5} and \eqref{eq-lem-21-6} completes the proof of the Claim \eqref{eq-distence2}.


Now we continue our proof of the conclusion \eqref{ell2-contral}. By Claim \eqref{eq-distence2}, we have
\begin{eqnarray*}
\sum_{k\in\Bbb Z}q_k^2 & \hskip-0.08in \le  & \hskip-0.08in  \sum_{k\in \Bbb Z} \Big(2^{-2\ell(k)+2}|q_{L(k)}|^2+2^{-2r(k)+2}|q_{R(k)}|^2\Big)\nonumber\\
 & \hskip-0.08in = & \hskip-0.08in \sum_{\lambda\in \Lambda_N}
\sum_{j=0}^\infty  |q_\lambda|^2\   2^{-2j+2} \times \Big(\sum_{L(k)=\lambda, \ell(k)=j} 1
+\sum_{R(k)=\lambda, r(k)=j} 1\Big)
 \le
\frac{32N}{3} \sum_{\lambda\in \Lambda_N}
  |q_\lambda|^2,
\end{eqnarray*}
where
the last inequality holds as 
the sets $\{k\in\Bbb Z: L(k)=\lambda, \ell(k)=j\}$ and $\{k\in\Bbb Z: R(k)=\lambda, r(k)=j\}$ are contained in
the intervals $[\lambda+jN, \lambda+jN+N-1]$ and $[\lambda-jN-N+1, \lambda-jN]$ respectively,
and hence have their cardinality bounded by $N$.
This completes the proof.
\end{proof}


Now we are ready to prove Theorem \ref{bound-sequence-0}. And the core of the proof is the part of the necessity related to the lower bound of Ron-Shen's characterization, 
i.e., after \eqref{bound-sequence-0.thmpf.eq21} in the proof.

\vspace{0.2cm}

\noindent\textbf{Proof of Theorem \ref{bound-sequence-0}.}
As $g$ is bounded and compactly supported, we can find a positive constant $M_0>a$ such that
\begin{equation}\label{eq-thm-21-1}
\|g\|_\infty\le M_0 \  \mbox{ and }\  g(t)=0  \  \mbox{for all }\ t\not\in(-M_0,M_0).
\end{equation}

\emph{Sufficiency}.
By the $a$-shift invariance of the requirement \eqref{eq-linear-4.2}, without loss of generality, we assume there exist $t_0\in [0,a)$ and a nonzero bounded sequence $Q=\{q_j\}_{j\in\Bbb Z}$ such that \eqref{eq-linear-4.2} holds for $h=g^+$. Take an arbitrary positive constant $A>0$. Then there exists sufficient large
$N_0\in\Bbb N$ such that
\begin{equation}\label{eq-discretization-01}
\|\{q_j\}_{N_0-2M_0\le |j|\le N_0}\|_{\ell^2}^2<\frac {aA}{2M_0^{3}( {2M_0}+a)} \|\{q_j\}_{|j|\le N_0} \|_{\ell^2}^2,
\end{equation}
where such an integer $N_0$ exists because, as $N_0\rightarrow\infty$, the left hand side of \eqref{eq-discretization-01} is bounded and the right hand side of \eqref{eq-discretization-01} tends to $\infty$ if $\|Q\|_{\ell^2}=\infty$, and left/right hand side of \eqref{eq-discretization-01} tends to 0 and a scaling of $\|Q\|_{\ell^2}^2$ respectively if $\|Q\|_{\ell^2}<\infty$.


Set $\{b_j\}_{j\in\Bbb Z}=Q_{[-N_0,N_0]}$,
then by \eqref{eq-linear-4.2} and \eqref{eq-thm-21-1},
\begin{equation}\label{eq-thm-21-2}
\sum_{j\in\Bbb Z} b_jg^+(t_0+j-na)=\sum_{j=-N_0}^{N_0} b_jg^+(t_0+j-na)=-\sum_{j\not\in [-N_0,N_0]} b_jg^+(t_0+j-na)=0
\end{equation}
if $t_0-na\not\in (-N_0-M_0,-N_0+M_0)\cup(N_0-M_0,N_0+M_0)$. Combining \eqref{eq-thm-21-1}, \eqref{eq-discretization-01} and \eqref{eq-thm-21-2} with the definition of the sequence $\{b_j\}_{j\in\Bbb Z}$, we obtain

\begin{eqnarray}\nonumber
\sum_{n\in\Bbb Z}\Big|\sum_{j\in\Bbb Z} b_j g^+(t_0+j- na)\Big|^2
&  =  &  \sum_{t_0-na\in N_0+(-M_0,M_0)}\Big|\sum_{|j+t_0-na|<M_0 } b_j g^+(t_0+j- na)\Big|^2 \\\nonumber
&  +  &  \sum_{t_0-na\in -N_0+(-M_0,M_0)}\Big|\sum_{|j+t_0-na|<M_0 } b_j g^+(t_0+j- na)\Big|^2 \\\nonumber
& \le & \|g^+\|^2_\infty \sum_{t_0-na\in N_0+(-M_0,M_0)}\Big(\sum_{|j-N_0|<2M_0 } |b_j| \Big)^2 \\\nonumber
&  +  & \|g^+\|^2_\infty  \sum_{t_0-na\in -N_0+(-M_0,M_0)}\Big(\sum_{|j+N_0|<2M_0 } |b_j| \Big)^2 \\\nonumber
& \le & 2\big(\frac {2M_0}a+1\big)M_0^3\sum_{N_0-2M_0 < |j|\le N_0} |b_j |^2\\\label{bound-sequence-0.thmpf.eq4}
& < &  A \|\{b_j\}_{j\in\Bbb Z}\|_{\ell^2}^2 .
\end{eqnarray}
Observe that
\begin{eqnarray*} \lim_{t\to t_0^+}
\sum_{n\in\Bbb Z}\Big|\sum_{j\in\Bbb Z} b_j g(t+j- na)\Big |^2 & = &
\lim_{t\to t_0^+}
\sum_{|n|< (N_0+M_0+|t_0|)/a}\Big|\sum_{|j|\le N_0} b_j g(t+j- na)\Big|^2
\\
 & = &
\sum_{n\in\Bbb Z}\Big|\sum_{j\in\Bbb Z} b_j g^+(t_0+j- na)\Big|^2.
\end{eqnarray*}
This together with \eqref{bound-sequence-0.thmpf.eq4}
implies the existence of $\varepsilon>0$ such that
$$
\sum_{n\in\Bbb Z}\Big|\sum_{j\in\Bbb Z} b_j g(t+j- na)\Big|^2  <  A \|\{b_j\}_{j\in\Bbb Z}\|_{\ell^2}^2<\infty
$$
for all $t\in(t_0,t_0+\varepsilon)$, which contradicts to the frame characterization \eqref{eq-discretization-0}
in Proposition  \ref{discretization-theorem-0}.  
This completes the proof of the sufficiency.

\vspace{0.2cm}

\emph{Necessity}. Suppose that $\mathcal{G}(g;a)$ is not a Gabor frame. Observe that for sequence $\{c_j\}_{j\in\Bbb Z} \in \ell^2$,
\begin{equation}\nonumber 
\sum_{n\in\Bbb Z}\big|\sum_{j\in\Bbb Z} c_j g(t+j- na)\big|^2\le 2M_0\|g\|_\infty^2 \sum_{n\in\Bbb Z} \sum_{|t+j-na|<M_0}|c_j|^2
\le \frac{2M_0(2M_0+a)}a \|g\|_\infty^2\|\{c_j\}_{j\in\Bbb Z}\|^2_{\ell^2}
\end{equation}
for almost all $t\in\Bbb R$. Therefore for all $k\ge M_0+2a+1$,  by the  non-frame assumption for the Gabor system ${\mathcal G}(g; a)$ and Proposition  \ref{discretization-theorem-0},
 we can find  a nonzero sequence $Q_k=\{q_{k,j}\}_{j\in\Bbb Z}\in\ell^2$  and a set $E_k\subset [0, a)$ with positive Lebesgue measure
such that
\begin{equation}\label{bound-sequence-0.thmpf.eq21} 
\sum_{n\in {\mathbb Z}} \Big|\sum_{j\in {\mathbb Z}} q_{k,j}  g(t+j-na)\Big|^2\le  \big(64 k^3 (2k+a+1)/3\big)^{-1} \|Q_k\|_{\ell^2}^2,\ t\in E_k.
\end{equation}
Define
$T_{t}(Q_k)=\{(T_{t}(Q_k))_{n}\}_{n\in\Bbb Z}$
by
$(T_{t}(Q)_k)_{n}=\sum_{j\in\Bbb Z}q_{k,j}g(t+j-na), n\in \Bbb Z$.
Then by \eqref{bound-sequence-0.thmpf.eq21} and Lemma \ref{lem-ell2-contral},
we have
\begin{eqnarray}\nonumber
\|\{\|(T_{t}(Q_k))_{(\lfloor {n}/a\rfloor-k/a, \lfloor{n }/a\rfloor +k/a)}\|_{\ell^2}\}_{n\in\Bbb Z}\|_{\ell^2}^2
&  \le  & (2k+a+1)\|T_{t}(Q_k)\|_{\ell^2}^2 \\\nonumber
&  \le   &
(64 k^3/3)^{-1} \|Q_k\|_{\ell^2}^2
\le  k^{-2} \|(Q_k)_{\Lambda_{2k}}\|_{\ell^2}^2, \ t\in E_k.
\end{eqnarray}
Thus there exist $n_0\in\Lambda_{2k}$ and a measurable set $E_{k, n_0}\subset [0, a)$ with positive Lebesgue measure such that
\begin{equation} \label{bound-sequence-0.thmpf.eq22}
\|(T_{t}(Q_k))_{( \lfloor{n_0}/a\rfloor -k/a,  \lfloor{n_0}/a\rfloor +k/a)}\|_{\ell^2} \le  k^{-1} \ |q_{k, n_0}|, \ \ t\in E_{k, n_0}.
\end{equation}

Set $t'=t+n_0-a\lfloor\frac {n_0}a\rfloor$ and define $P_k=\{p_{k,j}\}_{j\in \Bbb Z}$ by
$$p_{k,j}=\frac{q_{k, n_0+j}}{q_{k, n_0}}\chi_{[-2k, 2k]}(j),\ \ j\in \Bbb Z.$$
Then
\begin{equation}  \label{bound-sequence-0.thmpf.eq23} p_{k,0}=1
\end{equation}
by the definition of the sequence $P_k$,
\begin{equation}  \label{bound-sequence-0.thmpf.eq24}
\|P_{k}\|_\infty\le \frac{\|(Q_k)_{[n_0-2k, n_0+2k]}\|_\infty}{|q_{k, n_0}|} \le 2, \ \ j\in \Bbb Z,
\end{equation}
by the assumption  $n_0\in \Lambda_{2k}$ for the sequence $Q_k$,
and
\begin{eqnarray}  \label{bound-sequence-0.thmpf.eq25}
  &  & \sum_{|n|<k/a} \Big | \sum_{|j|\le 2k} p_{k, j} g(t'+j-na)\Big|^2
  \nonumber\\
 &  = &
|q_{k, n_0}|^{-2}
\sum_{|n|<k/a} \Big | \sum_{j\in \Bbb Z} q_{k, n_0+j} g(t+n_0+j-(n+ \lfloor {n_0}/a\rfloor)a)\Big|^2\nonumber\\
& = & |q_{k, n_0}|^{-2} \|(T_{t}(Q_k))_{( \lfloor{n_0}/a\rfloor -k/a,  \lfloor{n_0}/a\rfloor +k/a)}\|_{\ell^2}^2
\le k^{-2},  \  t'\in E_{k, n_0}+ n_0-a\lfloor {n_0}/a\rfloor,
\end{eqnarray}
where the first equality follows from \eqref{eq-thm-21-1}, the definition of the sequence $P_k$ and the observation that
$|t'+j-na|>2k-2a-k=k-2a\ge M_0$ for all $j\in \Bbb Z$ with $|j|> 2k$,
and the last inequality holds by \eqref{bound-sequence-0.thmpf.eq22}.

By  \eqref{bound-sequence-0.thmpf.eq23} and \eqref{bound-sequence-0.thmpf.eq24},
without loss of generality, we assume that the bounded sequence
$P_k, k\ge M+2a+1$, converge in weak star topology in the sense that
\begin{equation} \label{bound-sequence-0.thmpf.eq26}
\lim_{k\to \infty} p_{k,j}= r_j, j\in \Bbb Z,
\end{equation}
otherwise,  replacing by a subsequence.
Set $R=\{r_j\}_{j\in \Bbb Z}$.
By \eqref{bound-sequence-0.thmpf.eq23}, \eqref{bound-sequence-0.thmpf.eq24} and
  \eqref{bound-sequence-0.thmpf.eq26}, we have that
  \begin{equation} \label{bound-sequence-0.thmpf.eq27}  R \ \ {\rm is\ a \ nonzero \ sequence\ with}\ \|R\|_\infty\le 2 .
  \end{equation}

Next from the observation that  $E_{k, n_0}+ n_0-a\lfloor {n_0}/a\rfloor\subset [0, 2a)$, we can find a
monotonic sequence $t_k\in E_{k, n_0}+ n_0-a\lfloor {n_0}/a\rfloor$ such that
\begin{equation} \nonumber
\lim_{k\to \infty} t_{k}= t_\infty\in [0, 2a],
\end{equation}
otherwise replaced by a subsequence.  Without loss of generality, we assume that $t_k, k\ge M+2a+1$, are a decreasing sequence, which implies that
\begin{equation} \nonumber
\lim_{k\to \infty}
g(t_k+j-na)= g^+(t_\infty+j-na) \ {\rm for \ all}\   j, n\in {\Bbb Z}.
\end{equation}
Therefore for any $n\in \Bbb Z$, $\{g(t_k+j-na)\}_{j\in\Bbb Z}$ tends to $\{g^{+}(t_\infty+j-na)\}_{j\in\Bbb Z}$ in norm of $\ell^{1}$. Hence,
\begin{equation}\label{bound-sequence-0.thmpf.eq30}
0=\lim_{k\rightarrow\infty}\sum_{|j|\le 2k} p_{k,j} g(t_k+j-na)=\lim_{k\rightarrow\infty}\sum_{j\in\Bbb Z} p_{k,j} g(t_k+j-na)= \sum_{j\in\Bbb Z} r_{j} g^{+}(t_{\infty}+j-na),
\end{equation}
where the first equality follows from \eqref{bound-sequence-0.thmpf.eq25}, and the last equality is due to the fact that $\{r_j\}_{j\in\Bbb Z}$ is the weak star limit of $\{p_{k,j}\}_{j\in\Bbb Z}$. Combining \eqref{bound-sequence-0.thmpf.eq27} and \eqref{bound-sequence-0.thmpf.eq30} proves
\eqref{eq-linear-4.2} and hence  completes the proof.
\eproof

\vspace{0.2cm}

\vspace{0.2cm}

\vspace{0.2cm}

\section{Piecewise linear  {transformation}  I:  maximal invariant sets}\label{section-4}


This section {will be devoted to investigating} the structure of the maximal invariant set $\mathcal{S}$ of the map $\mathcal{M}$ defined as in \eqref{map.def}, which plays an important role in the next section for analyzing the symmetric maximal invariant set.
In Subsection \ref{subsection-4-3}, we  provide some descriptions on $\mathcal{S}$ in Proposition \ref{prop-4-point}, and then we use it to prove {Theorems} \ref{thm-s-1} and \ref{thm-mod-1}.
The structures of $\mathcal{S}$ for the cases $\mathcal{S}\subset U$ and  $\mathcal{S}\subset V$ are presented in Subsection \ref{subsection-4-2}.
In Subsection \ref{subsection-4-1}, we consider the case that ${\mathcal S}\cap U\ne\emptyset$ and ${\mathcal S}\cap V\ne\emptyset$, and show in Theorem \ref{thm-4-struct} that there are three types of the structures for such $\mathcal{S}$.  {At last}, we build the relationship between the different types of the structures via the disturbance of the map $\mathcal{M}$ 
in Proposition \ref{prop-delta-tur}.

\vspace{0.2cm}

{At first,} we indicate that both the maximal invariant set $\mathcal{S}$ and the symmetric maximal invariant set $\mathcal{E}$ are well defined. This is because we may verify that
\begin{align*}
   \mathcal{S} & =\{t\in\Bbb T_{[0,1)}:\mathcal{M}^n(t)\not\in\textbf{H}\ \mbox{ for all }\ n\ge 0\}
\end{align*}
and
\begin{align*}
   \mathcal{E} &= \{t\in\Bbb T_{[0,1)}:\mathcal{M}^m(\widetilde{\mathcal{M}}^n(1-t))\not\in\textbf{H}\ \mbox{ and }\ \widetilde{\mathcal M}^m(1-{\mathcal{M}}^n(t))\not\in\widetilde{\textbf{H}} \ \mbox{ for all }\ m,n\ge 0 \}.
\end{align*}

\vspace{0.2cm}

 We shall introduce some notations which will be used in the sequel. Two sets $E_1, E_2$ in the unit circle $\Bbb T_{[0,1)}$ are called  \emph{separated} if they have positive distance. Given an interval $I\subset \Bbb R$, we use the symbol $\langle I\rangle$ to denote an arc in $\Bbb T_{[0,1)}$, that is, $\langle I\rangle:=\{\langle x\rangle: x\in I\}$. In particular, if an arc $\langle I\rangle$ in $\Bbb T_{[0,1)}$   contains $0$, we might regard it as the union of two intervals in $[0,1)$,  such as $(a,1)\cup [0,b)$ for some $0<b<a<1$. Moreover, an arc $J$ is said to be a \emph{gap} (resp. \emph{open gap}) of $\mathcal{M}$ if $J$ is the maximal arc (resp. open arc) such that $J\subset \Bbb T_{[0,1)}\setminus\mathcal{S}$. We will see, from the proof of Theorem \ref{thm-s-1}, that all gaps are left-closed and right-open intervals in $\Bbb T_{[0,1)}$,  therefore they are pairwise separated.
\vspace{0.2cm}

\subsection{Proofs of {Theorems} \ref{thm-s-1} and \ref{thm-mod-1}}\label{subsection-4-3} $\empty$

\begin{prop}\label{prop-4-point}
 For $0\le\alpha<1$ and $0\leq x_1<x_2\le 1$, we let $\mathcal{M}=\mathcal{M}_{\alpha, x_1,x_2}$ be the transformation defined as  in \eqref{map.def}, and ${\mathcal S}$ be its maximal invariant set defined in \eqref{s-set.def}. If ${\mathcal S}\ne\emptyset$, then for each $0<\varepsilon<1$,
\begin{enumerate}
   \item [{\rm(i)}] ${\mathcal S}\cap [0,\varepsilon)\ne \emptyset $ or ${\mathcal S}\cap \langle[x_2,x_2+\varepsilon)\rangle\ne \emptyset$;
   \item [{\rm(ii)}]  ${\mathcal S}\cap [1-\varepsilon,1)\ne \emptyset $ or ${\mathcal S}\cap \langle[x_1-\varepsilon,x_1)\rangle\ne \emptyset$;
   \item [{\rm(iii)}] ${\mathcal S}\cap [0,\varepsilon)\ne \emptyset $ or ${\mathcal S}\cap [1-\varepsilon,1)\ne \emptyset $;
   \item [{\rm(iv)}]  ${\mathcal S}\cap \langle[x_1-\varepsilon,x_1)\rangle\ne \emptyset$ or  ${\mathcal S}\cap \langle[x_2,x_2+\varepsilon)\rangle\ne \emptyset$.
\end{enumerate}
\end{prop}

\proof
 Let $U,V$ {and $\mathcal{S}$} be as in \eqref{map.def} {and \eqref{s-set.def}}.

(i). Suppose on the contrary that there exists an $\varepsilon\in(0,1)$ such that
${\mathcal S}\cap [0,\varepsilon)={\mathcal S}\cap \langle[x_2,x_2+\varepsilon)\rangle=\emptyset$.
Set ${\mathcal S}'={\mathcal S}-\varepsilon$.   Then  ${\mathcal S}'\cap U={\mathcal S}\cap U-\varepsilon$ and ${\mathcal S}'\cap V={\mathcal S}\cap V-\varepsilon$. The linearity of $\mathcal{M}$ on $U$ and $V$ will yield that
 ${\mathcal M}({\mathcal S}'\cap U)=\langle{\mathcal M}({\mathcal S}\cap U)-\varepsilon\rangle$ and ${\mathcal M}({\mathcal S}'\cap V)=\langle{\mathcal M}({\mathcal S}\cap V)-\varepsilon\rangle$. Consequently,
$$
{\mathcal M}({\mathcal S}')
=\langle{\mathcal M}({\mathcal S})-\varepsilon\rangle=\langle{\mathcal S}-\varepsilon\rangle={\mathcal S}'.
$$
This implies ${\mathcal M}({\mathcal S}\cup{\mathcal S}')={\mathcal S}\cup{\mathcal S}'$, contradicting to the maximality of $\mathcal S$,  we obtain (i).
\vspace{0.2cm}

(ii). {The idea of the proof is the same to that of (i). We omit the details.}

\vspace{0.2cm}

(iii). On the contrary, we may let $\varepsilon_1,\varepsilon_2\in(0,1)$ be the largest numbers such that
 \begin{equation}\label{i-ii-gap-0}
 \mathcal{S}\cap [0,\varepsilon_1)={\mathcal S}\cap [1-\varepsilon_2,1)=\emptyset.
 \end{equation}
Now it follows from  the  dichotomy in (i) and (ii)  that $0<x_1<x_2<1$ and
\begin{equation}\label{i-ii-gap}
{\mathcal S}\cap [x_1-\varepsilon,x_1)\ne \emptyset \ \mbox{ and }\ {\mathcal S}\cap [x_2,x_2+\varepsilon)\ne \emptyset
\end{equation}
hold for all $0<\varepsilon<\min{(x_1, 1-x_2)}$, and therefore $[0,\varepsilon_1)\subset U$ and $[1-\varepsilon_2,1)\subset V$. This means $\varepsilon_1<x_1$ and $\varepsilon_2<1-x_2$.
 Combining with \eqref{i-ii-gap-0} and $\mathcal{M}(\mathcal{S})=\mathcal{S}$, one has that
$$\mathcal{S}\subset \mathcal{M}([\varepsilon_1,x_1))\cup \mathcal{M}([x_2,1-\varepsilon_2))\subset\langle[\alpha+x_2-x_1+ \varepsilon_1,1+\alpha-\varepsilon_2)\rangle.$$ Let $K=\langle[\alpha-\varepsilon_2,\alpha+x_2-x_1+ \varepsilon_1)\rangle$. Then $K$ is an arc of the length $x_2-x_1+\varepsilon_1+ \varepsilon_2$ and $K\cap \mathcal{S}=\emptyset.$ This, together with \eqref{i-ii-gap-0} and \eqref{i-ii-gap},  yields that $K\cap\textbf{H}=\emptyset$ and $K\cap\big( [1-\varepsilon_2,1)\cup[0,\varepsilon_1)\big)=\emptyset$.  Therefore, $K\subset U$ or $K\subset V$, and hence $|{\mathcal M} (K)|=|K|$ and
$$
\mathcal{S}\cap\mathcal{M}(K)\subset\mathcal{M}\Big(\mathcal{S}\cap\big(K\cup\textbf{H}\big)\Big)=\emptyset.
$$
Applying the above procedure to the set $\mathcal{M}(K)$ instead of $K$, and then by induction, we will get a sequence of arcs $\mathcal{M}^n(K), n\geq 0$, with the same length $|{\mathcal M}^n(K)|=|K|$,  such that $\mathcal{M}^n(K)\subset\Bbb T_{[0,1)}\setminus \mathcal{S}$ satisfying that  $\mathcal{M}^n(K)\cap\textbf{H}=\emptyset$ and $\mathcal{M}^n(K)\cap\big([1-\varepsilon_2,1)\cup[0,\varepsilon_1)\big)=\emptyset$.
Thus, for any two integers  $1\le n_1<n_2$,
$$
\mathcal{M}^{n_1}(K)\cap \mathcal{M}^{n_2}(K)=\mathcal{M}^{n_1}(K\cap\mathcal{M}^{n_2-n_1}(K))
=\mathcal{M}^{n_1+1}\Big(\big([1-\varepsilon_2,1)\cup[0,\varepsilon_1)\big)\cap\mathcal{M}^{n_2-n_1-1}(K)\Big)=\emptyset.
$$
This means ${\mathcal M}^n(K), n\in\Bbb N$,   is a sequence of pairwise disjoint sets with the same length in $\Bbb T_{[0,1)}$, this contradiction means that (iii) holds.

\vspace{0.2cm}

(iv). The proof is similar to that of (iii).  In fact, since $\mathcal{S}\cap $\textbf{H}$=\emptyset,$ one can suppose on the contrary that there exists
  an $\varepsilon_0>0$ such that {${\mathcal S}\cap \langle[x_1-\varepsilon_0,x_2+\varepsilon_0)\rangle= \emptyset$}.
This together with (i) and (ii) implies that
\begin{equation}\nonumber
{\mathcal S}\cap [1-\varepsilon,1)\ne \emptyset \ \mbox{ and }\ {\mathcal S}\cap [0,\varepsilon)\ne \emptyset
\end{equation}
for all $0<\varepsilon<1$,
    which yields, from the definition of $\mathcal{M}$, that
\begin{equation}\nonumber
{\mathcal S}\cap \langle[\alpha-\varepsilon,\alpha)\rangle\ne \emptyset \ \mbox{ and }\ {\mathcal S}\cap \langle[\alpha+x_2-x_1,\alpha+x_2-x_1+\varepsilon)\rangle\ne \emptyset.
\end{equation}
 Let $K=\langle(\alpha,\alpha+x_2-x_1)\rangle$. Then the above conditions mean that $K$ is an open gap, and therefore {$K\cap \textbf{H}=\emptyset$ and $0\not \in K$}. Thus,  it follows that $K\subset U$ or $K\subset V$, and hence ${\mathcal M}(K)$ is an open gap with $|{\mathcal M}(K)|=|K|$. By induction, we will get a sequence of open gaps ${\mathcal M}^n(K), n\in\Bbb N$, with the same length $|{\mathcal M}^n(K)|=|K|$, and that,   for any $1\le n_1<n_2$,
\begin{equation}\nonumber
{\mathcal M}^{n_1}(K)\cap {\mathcal M}^{n_2}(K)={\mathcal M}^{n_1}(K\cap{\mathcal M}^{n_2-n_1}(K))=\emptyset.
\end{equation}
It is impossible, (iv) follows. The  proof of Proposition \ref{prop-4-point}  is  complete.
\eproof

\vspace{0.2cm}

\begin{cor}\label{cor-4-point}
Let $\mathcal{M}=\mathcal{M}_{\alpha, x_1,x_2}$ be the map defined in \eqref{map.def}, and ${\mathcal S}$ be its maximal invariant set.
If ${\mathcal S}\ne\emptyset$, then $0\le \alpha<\alpha+x_2-x_1\le 1$.
\end{cor}
\proof
 Since $0\le \alpha<1$ and $x_2>x_1$, it suffice to prove $\alpha+x_2-x_1\le 1$. Assume on the contrary that $\alpha+x_2-x_1>1$. Then $0<x_1+(\alpha+x_2-x_1-1)=\alpha+x_2-1\le\alpha$. Thus $\mathcal{M}(U)=\langle U+\alpha+x_2-x_1\rangle=[\alpha+x_2-x_1-1,\alpha+ x_2-1)$ and $\mathcal{M}(V)=\langle V+\alpha\rangle=[\alpha +x_2-1,\alpha)$. Therefore $\mathcal{S}\cap\big((\alpha,1)\cup[0,\alpha+x_2-x_1-1)\big)=\emptyset$  contradicting to Proposition \ref{prop-4-point} (iii).
\eproof

\vspace{0.2cm}
\vspace{0.2cm}

\noindent {\bf Proof of Theorem \ref{thm-s-1}}.
If $\alpha+x_2-x_1> 1$, then $\mathcal{S}=\emptyset$ by Corollary \ref{cor-4-point}. Moreover, since $\mathcal{M}(t)$ is equal $t+\alpha+x_2-x_1- 1$ if $t\in U$; and $t+\alpha- 1$ if $t\in V$, it follows that $\mathcal{M}^n(t)\in \textbf{H}$ for all $n>\max (\frac{x_1}{\alpha+x_2-x_1- 1}, \frac{1-x_2}{1-\alpha})$, and hence $\mathcal{M}^n(\mathbb{T}_{[0,1)}) \subset \textbf{H}$, \eqref{eq-thm-s-1} follows.
\vspace{0.2cm}

If $\alpha+x_2-x_1\leq 1$, we will conclude that \eqref{eq-thm-s-1}  by considering  the dynamical behavior of the sequence $T_n={\mathcal M}^n(T)\setminus \textbf{H}$ for all $ n\ge 0$, where $T=[\alpha, \alpha+x_2-x_1)$.
 As ${\mathcal M}$ is a bijection from $\Bbb T_{[0,1)}\setminus \textbf{H}$ to $\Bbb T_{[0,1)} \setminus T$, we have ${\mathcal M}(\Bbb T_{[0,1)}) \cap T_0=\emptyset$, which yields that
\begin{equation}\label{eq-finite-in-1}
T_n\cap \mathcal{S}\subset\mathcal{M}^{n}(T_0\cap\mathcal{S})=\emptyset
\end{equation}
for all $n\ge 0$, and
\begin{equation}\label{eq-finite-infinite-1}
T_{n_1}\cap T_{n_2}\subset\mathcal{M}^{n_1}(T_{0}\cap T_{n_2-n_1})=\emptyset
\end{equation}
 for all $0\le n_1<n_2$.

 \vspace{0.2cm}
When $T_0=\emptyset$, i.e., $\alpha=x_1$, it is clear that $\mathcal{S}=\Bbb T_{[0,1)}\setminus \textbf{H}$, so \eqref{eq-thm-s-1} holds for all $N\ge 0$.  When $T_0\ne\emptyset$, it is not hard to see  from \eqref{eq-finite-in-1} and Proposition \ref{prop-4-point} (iii) that $T_0\subset U$ or $T_0\subset V$, and thus $T_{1}=\mathcal{M}(T_0)\setminus\textbf{H}$ is a left-closed and right-open arc whenever $T_1\ne\emptyset$. By induction, one will get a sequence  $T_n, n\geq 0$  satisfying  that  either $T_n\subset U$ or $T_n\subset V$, and $T_n$ is also a left-closed and right-open arc whenever $T_n\ne\emptyset$. Notice that $T_{n+1}=\mathcal{M}(T_n)\setminus\textbf{H}$  and $|\mathcal{M}(T_n)|=|T_n|$ for all $n\ge0$. Then
\begin{equation}\label{eq-t-t-01}
|{\mathcal M}(T_{n})\cap \textbf{H}|=|T_{n}|-|T_{n+1}|\ \mbox{for all}\ n\ge 0.
\end{equation}

 \vspace{0.2cm}

 We now claim that
\begin{equation}\label{claim-101}
T_N=\emptyset\  \mbox{ for some }\ N\in\Bbb N.
\end{equation}

 Indeed, if \eqref{claim-101} is not valid, then the argument above shows that the sets $T_n,n\ge 0$, are all left-closed and right-open arcs.  Thus, by \eqref{eq-finite-infinite-1}, there are at most  two of the sets $T_n$ such that $\mathcal{M}(T_n)\cap\textbf{H}\ne\emptyset$, which means there exists $N_0\in\Bbb N$ such that ${\mathcal M}(T_n)\cap \textbf{H}= \emptyset$  for all $n\ge N_0$, and therefore $|T_n|=|T_{N_0}|>0$ for all $n\ge N_0$. This is a contradiction since $\sum_{n=0}^\infty |T_n|<1$ by \eqref{eq-finite-infinite-1}. Hence, \eqref{claim-101}  is valid.

\vspace{0.2cm}

Let's  continue to our proof. Recall that $T\cap {\mathcal M}(T_n)=\emptyset$ for all  $n\ge0$.
Then by \eqref{eq-finite-infinite-1}, \eqref{eq-t-t-01}  and \eqref{claim-101},
\begin{equation}\nonumber
\Big|\cup_{k=0}^{N-1}({\mathcal M}(T_{k})\cap \textbf{H})\cup(T\cap  \textbf{H})\Big|=\sum_{k=0}^{N-1}\big|{\mathcal M}(T_{k})\cap \textbf{H}\big|+\big|T\cap  \textbf{H}\big|=|T|-|T_{N}|=|T|=|\textbf{H}|.
\end{equation}
Since the set of the left hand side of above equality is the  union of finitely many left-closed and right-open intervals, it follows that
\begin{equation}\label{eq-finite-infinite-7}
\textbf{H}\subset \cup_{k=0}^{N-1}{\mathcal M}(T_{k})\cup T.
\end{equation}

Set $E=\Bbb T_{[0,1)}\setminus \big(\cup_{k=0}^{N-1} T_k\cup \textbf{H}\big)$.
Then $|{\mathcal M}(E)|=|E|$, and ${\mathcal M}(E)\subset E$ by \eqref{eq-finite-infinite-7} and the observation that ${\mathcal M}$ is a bijection from $\Bbb T_{[0,1)}\setminus \textbf{H}$ to $\Bbb T_{[0,1)} \setminus T$. Therefore ${\mathcal M}(E)= E$ since  they are both the union of finitely many left-closed and right-open intervals. Recall that ${\mathcal S}\subset E$ by \eqref{eq-finite-in-1}. Then $E= {\mathcal S}$ by the maximality of ${\mathcal S}$. This together with \eqref{claim-101} leads to ${\mathcal M}^N(t)\in \textbf{H}$ for all $t\in \Bbb T_{[0,1)}\setminus \mathcal{S}$.
We complete the proof of Theorem \ref{thm-s-1}.
\eproof

\vspace{0.2cm}

\vspace{0.2cm}

\noindent {\bf Proof of Theorem \ref{thm-mod-1}.}
First, by Theorem \ref{thm-s-1}, ${\mathcal S}$  can be  written as the union of finitely many pairwise disjoint  left-closed and right-open intervals in [0,1).
  It is obvious that  the map $Y$ is bijective from ${\mathcal S}$ to $\Bbb T_{[0,1)}$.

 \vspace{0.2cm}

 If $0\le t_1<t_2<1$ are two points in one interval of $\mathcal{S}$, then
 $$
 \langle Y({\mathcal M}(t_2))-Y(\mathcal{M}(t_1))\rangle=\frac {|{\mathcal S}\cap \mathcal{M}([t_1,t_2))|}{|\mathcal{S}|}=\frac {|{\mathcal S}\cap [t_1,t_2)|}{|\mathcal{S}|}=\langle Y(t_2)-Y(t_1)\rangle,
 $$
which implies  $\langle Y({\mathcal M}(t_2))-Y(t_2)\rangle=\langle Y({\mathcal M}(t_1))-Y(t_1)\rangle$.

If $0\le t_1<t_2<1$ are two points in two adjacent intervals of ${\mathcal S}$, assume that $t_1\in [a,b)\subset{\mathcal S}$, $t_2\in [c,d)\subset{\mathcal S}$ and $[b,c)\subset\Bbb T_{[0,1)}\setminus{\mathcal S}$, then
$$
 \langle Y({\mathcal M}(t_2))-Y(t_2)\rangle=\langle Y({\mathcal M}(c))-Y(c)\rangle=\lim_{t\rightarrow b^+}\langle Y({\mathcal M}(t))-Y(t)\rangle=\langle Y({\mathcal M}(t_1))-Y(t_1)\rangle.
 $$
 So,  by recursion, $\langle Y({\mathcal M}(t))-Y(t)\rangle$ is a constant for all
 $t\in {\mathcal S}$, and therefore $$\langle Y({\mathcal M}(t))-Y(t)\rangle=\langle Y({\mathcal M}(0))-Y(0)\rangle=\langle Y(\alpha+x_2-x_1)\rangle=\langle Y(\alpha)\rangle,$$
where the first equality follows from Proposition \ref{prop-4-point} (iii) and take $t\to 0$ or $t\to 1$ for desired situations.  We obtain \eqref{eq-thm-mod-01}.

 \vspace{0.2cm}

 The other conclusion in Theorem \ref{thm-mod-1} is an immediate consequence of Proposition \ref{prop-s-rational} below.
\eproof

\subsection{Maximal invariant set, Part I}
\label{subsection-4-2}

\vspace{0.2cm}

In this subsection, we study the structure of the maximal invariant set ${\mathcal S}$ when ${\mathcal S}\subset U$ and $ {\mathcal S}\subset V$ respectively.


 \begin{prop}\label{prop-s-rational}
Let $\mathcal{M}=\mathcal{M}_{\alpha, x_1,x_2}$ be the map defined in \eqref{map.def}. Then:
\begin{enumerate}
\item [{\rm(i)}] $\emptyset \ne {\mathcal S}\subset U$ if and only if $\langle\alpha+x_2-x_1\rangle=\frac pq$ and $x_1>\frac {q-1}q$, where $\frac pq$ is a simple fraction for some integers $0\le p< q$. In this case, $Y(\alpha)=\frac pq$ and
\begin{equation}\label{struct-s-rational-1}
{\mathcal S}= \bigcup_{k=0}^{q-1}\big[\frac kq,\frac {k-q+1}q+ x_1\big).
\end{equation}
\item [{\rm(ii)}] $\emptyset \ne {\mathcal S}\subset V$ if and only if $\alpha=\frac pq$ and
    $x_2<\frac {1}q$,  where $\frac pq$ is a simple fraction for some integers $0\le p< q$. In this case, $Y(\alpha)=\frac pq$ and
\begin{equation}\label{struct-s-rational-2}
{\mathcal S}= \bigcup_{k=0}^{q-1}\big[x_2+\frac kq, \frac {k+1}q\big).
\end{equation}
\end{enumerate}
\end{prop}

\proof
(i). \emph{Necessity}.  Note that $\emptyset \ne {\mathcal S}\subset U$ implies that $\alpha+x_2-x_1\in (0,1]\cap\Bbb Q$ by Corollary \ref{cor-4-point} and the observation that
\begin{equation}\label{S-u-v-1-121}
  {\mathcal S}=\mathcal{M}(\mathcal{S})=\langle\mathcal{S}+\alpha+x_2-x_1\rangle.
\end{equation}

  \vspace{0.2cm}

 If $\alpha+x_2-x_1=1$, then $\mathcal{S}=U$, which obviously leads to that $Y(\alpha)=Y(\alpha+x_2-x_1)=0$ and \eqref{struct-s-rational-1} holds. Write $$\alpha+x_2-x_1=\frac pq$$ for some co-prime positive integers $p<q$. Then it follows from  \eqref{S-u-v-1-121}  that
$$\langle{\mathcal S}+\frac1q\rangle={\mathcal S} \quad \hbox{and} \quad   \emptyset\ne{\mathcal S}\cap [\frac {q-1}q,1)\subset [0,x_1)\cap[\frac {q-1}q,1),$$ which implies that $x_1>\frac {q-1}q$ and
\begin{equation}\label{s-rational}
{\mathcal S}\subset \bigcup_{k=0}^{q-1}\big[\frac kq,\frac {k-q+1}q+ x_1\big).
\end{equation}
  Observing  that the right hand side of \eqref{s-rational} is invariant under the map ${\mathcal M}$. Thus,  \eqref{struct-s-rational-1} is valid  and $$Y(\alpha)=Y(\alpha+x_2-x_1)=\frac {|\mathcal{S}\cap[0,\alpha+x_2-x_1)|}{|\mathcal{S}|}=\frac {|\mathcal{S}\cap[0,p/q)|}{|\mathcal{S}|}=\frac pq,$$  as desired.

 \vspace{0.2cm}

        \emph{Sufficiency}. Notice that the set ${\mathcal S_U}= \bigcup_{k=0}^{q-1}\big[\frac kq,\frac {k-q+1}q+ x_1\big)\subset U$ is an invariant set of ${\mathcal M}$, i.e., ${\mathcal S}_U \subset \mathcal{S}$.  Therefore  it suffices to show that $\mathcal{S}\cap V=\emptyset$.

    \vspace{0.2cm}

        In fact, if $\mathcal{S}\cap V\ne \emptyset$, it is reasonable to define $$t=\min \{x: x\in\mathcal{S}\cap V\}$$ since $\mathcal{S}$ can be rewritten as the union of finitely many pairwise disjoint left-closed and right open intervals in $[0,1)$ by Theorem \ref{thm-s-1}. Thus, $t-(x_2-x_1)\not\in\mathcal{S}$.

    \vspace{0.2cm}

  However, since each $t'\in\mathcal{S}\cap [0,\frac {q-1}q)$  corresponds to a $k'\in\{1,2,\ldots,q-1\}$ such that $\mathcal{M}^{k'}(t')=\langle t'+\frac{k'p}q\rangle \in [\frac {q-1}q, 1)$, it follows that
  $$\mathcal{M}^{k+1}(t)=\mathcal{M}^{k}(\langle t+\alpha\rangle)=\langle t+\alpha+\frac{kp}q\rangle \in [\frac {q-1}q, 1)$$
   for some  $k\in\{0, 1,\ldots,q-1\}$.   Notice that $t-(x_2-x_1)\in [\frac {q-1}q, 1)$ and $(t-(x_2-x_1))-\langle t+\alpha+\frac{kp}q\rangle\in\Bbb Z/q$. Then $t-(x_2-x_1)=\mathcal{M}^{k+1}(t)\in \mathcal{S}$. It is a contradiction, and thus $\mathcal{S}\cap V=\emptyset$, as desired.

\vspace{0.2cm}

(ii). The proof is similar to that of (i), we will give the detailed proof for the sake of completeness.

    \emph{Necessity}. If ${\mathcal S}\subset V$, it follows from  $ {\mathcal S}=\langle\mathcal{S}+\alpha\rangle$ that $\alpha\in \Bbb Q$. Thus, for the case that $\alpha=0$, one has $p=0$ and $q=1$, and $\mathcal{S}=V$ by the assumption that  $\emptyset \ne {\mathcal S}\subset V$. Therefore $Y(\alpha)=0$ and \eqref{struct-s-rational-2} is valid. For the case that $\alpha\not=0$, we write $\alpha=\frac pq$ for some co-prime positive integers $p<q$, then ${\mathcal S}$ is a $\frac 1q-$periodic set, and  $$\emptyset\ne{\mathcal S}\cap [0,\frac {1}q)\subset [x_2,1)\cap[0,\frac {1}q).$$ Therefore, we have $x_2<\frac {1}q$, ${\mathcal S}= \cup_{k=0}^{q-1}\big[x_2+\frac kq,\frac {k+1}q\big)$, and obviously $Y(\alpha)=\frac pq$.

\vspace{0.2cm}

    \emph{Sufficiency}. Since the set ${\mathcal S}_V= \bigcup_{k=0}^{q-1}\big[x_2+\frac kq,\frac {k+1}q\big)\subset V$ is an invariant set of ${\mathcal M}$, so it suffices to prove  $\mathcal{S}\cap U=\emptyset$.

     \vspace{0.2cm}

   Assume on the contrary that $\mathcal{S}\cap U\ne\emptyset$. One  can choose $t\in\mathcal{S}\cap U$ such that $$t>\sup\{x: x\in\mathcal{S}\cap U\}-(x_2-x_1).$$
   This means that  $t+x_2-x_1\not\in\mathcal{S}$. We will show that it is impossible in the following:

    \vspace{0.2cm}

    (a) If $\mathcal{M}(t)=\langle t+\alpha+x_2-x_1\rangle\in [0,\frac {1}q)$, then $\mathcal{M}(t)=t+x_2-x_1$ by the observation that $t+x_2-x_1\in [0,\frac {1}q)$ and $\langle t+\alpha+x_2-x_1\rangle-(t+x_2-x_1)\in\Bbb Z/q$.  Therefore $t+x_2-x_1\in \mathcal{S}$.

   (b) If $\mathcal{M}(t)=\langle t+\alpha+x_2-x_1\rangle\not\in [0,\frac {1}q)$, then there exists $k\in\{1,2,\ldots,q-1\}$ such that $$\mathcal{M}^{k+1}(t)=\mathcal{M}^{k}(\langle t+\alpha+x_2-x_1\rangle)=\langle t+\alpha+x_2-x_1+\frac{kp}q\rangle \in [0,\frac {1}q).$$  Therefore, $\mathcal{M}^{k+1}(t)=t+x_2-x_1$, and $t+x_2-x_1\in \mathcal{S}$.

   \vspace{0.2cm}

    Thus $\mathcal{S}\cap U=\emptyset$. This finishes the proof of Proposition \ref{prop-s-rational}.
 \eproof





 \vspace{0.2cm}


\subsection{Maximal invariant set, Part II}
\label{subsection-4-1}


In this subsection, we consider the topological structure of the maximal invariant set ${\mathcal S}$ when $\mathcal{S}$ satisfies
\begin{equation}\label{S-u-v}
{\mathcal S}\cap U\ne\emptyset\ \mbox{and }\ {\mathcal S}\cap V\ne\emptyset.
\end{equation}

\vspace{0.2cm}

In this case, we have  $0<x_1<x_2<1$. 
Recall that $\alpha=0$ and $\alpha+x_2-x_1=1$ will lead to $\mathcal{S}=V$ and $\mathcal{S}=U$ respectively by Proposition \ref{prop-s-rational}. Therefore, by Corollary \ref{cor-4-point}, we have
\begin{equation}\nonumber
0<\alpha<\alpha+x_2-x_1<1.
\end{equation}

The fine structure of the maximal invariant set $\mathcal{S}$ as in \eqref{S-u-v} is characterized by the following Theorem \ref{thm-4-struct}, in which   we denote these three types of structures for $\mathcal{S}$ by \textbf{Type I}, \textbf{II} and \textbf{III}.

\vspace{0.2cm}

\begin{thm}\label{thm-4-struct}
Let $\mathcal{M}=\mathcal{M}_{\alpha, x_1,x_2}$ be the map defined in \eqref{map.def} and ${\mathcal S}$ be its maximal invariant set defined in \eqref{s-set.def}. Assume that ${\mathcal S}$ satisfies \eqref{S-u-v}. Then one of the following statements holds.
\begin{enumerate}
\item [{ \rm(i)}] (\textbf{Type I}) There exists $N\in\Bbb N$ such that ${\mathcal M}^{N-1}([\alpha,\alpha+x_2-x_1))=[x_1,x_2)$ and 
\begin{equation}\label{struct-s-1}
\Bbb T_{[0,1)}\setminus{\mathcal S}= \bigcup_{k=1}^N{\mathcal M}^{k-1}([\alpha,\alpha+x_2-x_1)).
\end{equation}
\item [{\rm (ii)}] (\textbf{Type II}) There exist {$0<\delta<\min (x_1,1-\alpha-x_2+x_1)$} and integers $0<N<M$ such that
$${\mathcal M}^{N-1}([\alpha,\alpha+x_2-x_1+\delta))=[x_1-\delta,x_2), \quad {\mathcal M}^{M-N}([x_1-\delta,x_1))=[0,\delta)$$
and
\begin{equation}\label{struct-s-2}
\Bbb T_{[0,1)}\setminus{\mathcal S}= \bigcup_{k=1}^N{\mathcal M}^{k-1}([\alpha,\alpha+x_2-x_1+\delta)) \cup \bigcup_{j=1}^{M-N}{\mathcal M}^{j-1}([x_1-\delta,x_1)).
\end{equation}
\item [{\rm (iii)}] (\textbf{Type III}) There exist {$0<\delta<\min(1-x_2,\alpha)$} and integers $0<N<M$ such that
$${\mathcal M}^{N-1}([\alpha-\delta,\alpha+x_2-x_1))=[x_1,x_2+\delta), \quad {\mathcal M}^{M-N}([x_2,x_2+\delta))=[1-\delta,1)$$
 and
\begin{equation}\label{struct-s-3}
\Bbb T_{[0,1)}\setminus{\mathcal S}= \bigcup_{k=1}^N{\mathcal M}^{k-1}([\alpha-\delta,\alpha+x_2-x_1)) \cup \bigcup_{j=1}^{M-N}{\mathcal M}^{j-1}([x_2,x_2+\delta)).
\end{equation}
\end{enumerate}
Furthermore, the unions of \eqref{struct-s-1}, \eqref{struct-s-2} and \eqref{struct-s-3} are  separated respectively.
\end{thm}

 \proof  According to Proposition \ref{prop-4-point} (iii), we  can divide the whole proof into the following three cases corresponding to (i), (ii) and (iii), respectively.

 \vspace{0.2cm}
\vspace{0.2cm}

 \emph{Case 1:} \emph{${\mathcal S}\cap [0,\varepsilon)\ne \emptyset$ and ${\mathcal S}\cap [1-\varepsilon,1)\ne \emptyset$ for all $\varepsilon\in(0,1)$.}

     \vspace{0.2cm}

     In this case, it follows from ${\mathcal M}({\mathcal S})={\mathcal S}$   that
         \begin{equation}\nonumber
           {\mathcal S}\cap [\alpha-\varepsilon,\alpha)\ne\emptyset \quad \hbox{and} \quad
         {\mathcal S}\cap [\alpha+x_2-x_1,\alpha+x_2-x_1+\varepsilon)\ne\emptyset
    \end{equation}
    for all $0<\varepsilon<\min (\alpha, 1-\alpha-x_2+x_1)$. Then the set $T=[\alpha,\alpha+x_2-x_1)$ is a gap of $\mathcal{M}$ and it is separated from $0$.

    By Theorem \ref{thm-s-1}, there is the smallest number $N\in \Bbb N$  such that $\mathcal{M}^{N-1}(T)\cap \textbf{H}\ne \emptyset$. Moreover, to each $0\le k<N-1$, if $\mathcal{M}^k(T)$ is a gap of $\mathcal{M}$,
    then  either $\mathcal{M}^k(T)\subset U$ or $\mathcal{M}^k(T)\subset V$, and thus it follows from $\mathcal{M}(\mathcal{S})=\mathcal{S}$ that $\mathcal{M}^{k+1}(T)$ is also a gap with length $|\mathcal{M}^{k+1}(T)|=|\mathcal{M}^{k}(T)|$. Therefore, $\mathcal{M}^{N-1}(T)$ is a gap with $|\mathcal{M}^{N-1}(T)|=|T|=|\textbf{H}|$, which implies $\mathcal{M}^{N-1}(T)=\textbf{H}$.

       \vspace{0.2cm}

  We next claim that $\mathcal{M}^k(T)$, $0\le k<N$, are pairwise disjoint. In fact, if there exist $0\le k_1<k_2<N$ such that $\mathcal{M}^{k_1}(T)\cap\mathcal{M}^{k_2}(T)\ne \emptyset$, then $\mathcal{M}^{k_1}(T)=\mathcal{M}^{k_2}(T)$ since both of them are gaps, and therefore $\mathcal{M}^{N-1-k_2+k_1}(T)=\mathcal{M}^{N-1}(T)=\textbf{H}$, which contradicts to the assumption of $N$.

  \vspace{0.2cm}

 Let $S^\ast=\Bbb T_{[0,1)}\setminus \cup_{k=0}^{N-1}{\mathcal M}^k(T)$. One can check that  ${\mathcal S}\subset S^\ast$ and ${\mathcal M}(S^\ast)=S^\ast$.  So we have ${\mathcal S}= S^\ast$ by the maximality of $\mathcal{S}$, (i) follows.





 \vspace{0.2cm}
 \vspace{0.2cm}

\emph{Case 2:  ${\mathcal S}\cap [1-\varepsilon,1)\ne\emptyset$ for all $\varepsilon>0$, and  ${\mathcal S}\cap [0,\delta)=\emptyset$ for some $\delta>0$.}

 \vspace{0.2cm}
Without loss of generality assume $\delta<x_1$ is the largest number such that ${\mathcal S}\cap [0,\delta)=\emptyset$. {Obviously, $\delta<1-\alpha-x_2+x_1$, the sets $[0,\delta)$} and  $T_{\delta}=[\alpha,\alpha+x_2-x_1+\delta)$ are separated gaps of $\mathcal{M}$, and $[0,\delta)$ is   separated from $\textbf{H}$.

By Theorem \ref{thm-s-1}, there is the smallest number $ N\in\Bbb N$ such that
\begin{equation}\label{cap-H-1}
{\mathcal M}^{N-1}(T_{\delta}) \cap \textbf{H}\ne \emptyset.
  \end{equation}
Thus, by the same argument as in the {\emph{Case 1}},   ${\mathcal M}^k(T_{\delta}), 0\le k\le N-1$, are pairwise disjoint gaps with the same length, and they are separated from $[0,\delta)$. Notice that Proposition \ref{prop-4-point} (i) implies that ${\mathcal S}\cap [x_2,x_2+\varepsilon)\ne \emptyset$ for all small $\varepsilon>0$. This, together with \eqref{cap-H-1}, yields that
$$
{\mathcal M}^{N-1}(T_{\delta})=[x_1-\delta,x_2),
$$
and hence ${\mathcal M}([x_1-\delta,x_1))$ is a gap with length $\delta$.

 \vspace{0.2cm}

By Theorem \ref{thm-s-1} again,  there is the smallest number  $K\in \Bbb N$   such that ${\mathcal M}^K([x_1-\delta,x_1))$ intersects $\bigcup_{k=0}^{N-1}{\mathcal M}^{k}(T_{\delta})\cup\bigcup_{j=1}^{K-1}{\mathcal M}^{j}([x_1-\delta,x_1))\cup [0,\delta)$,
  then ${\mathcal M}^j([x_1-\delta,x_1))$, $1\le j\le K$, are pairwise separated gaps with the same length $\delta$, and
${\mathcal M}^K([x_1-\delta,x_1))=[0,\delta)$.

 \vspace{0.2cm}

Let $\mathcal{S}^\ast$ be the complementary set of $\bigcup_{k=0}^{N-1}{\mathcal M}^{k}(T_{\delta})\cup\bigcup_{j=1}^{K-1}{\mathcal M}^{j}([x_1-\delta,x_1))\cup [0,\delta)$. Then ${\mathcal S}\subset S^\ast$ and ${\mathcal M}(S^\ast)=S^\ast$,  (ii) follows.

 \vspace{0.2cm}
 \vspace{0.2cm}

 \emph{Case 3: ${\mathcal S}\cap [0,\varepsilon)\ne\emptyset$ for all $\varepsilon>0$, and  ${\mathcal S}\cap [1-\delta,1)=\emptyset$ for some $\delta>0$.}

\vspace{0.2cm}

 As in \emph{Case 2}, we may assume $\delta<1-x_2$ is the largest number such that ${\mathcal S}\cap [1-\delta,1)=\emptyset$. One can {check that $\delta<\alpha$, $[1-\delta,1)$} and $T_{\delta}=[\alpha-\delta,\alpha+x_2-x_1)$ are separated gaps, and $[1-\delta,1)$ is separated from $\textbf{H}$. Applying the arguments of \emph{Case 2} to $[x_2,x_2+\delta)$ and $[1-\delta,1)$ instead of $[x_1-\delta,x_1)$ and $[0,\delta)$, respectively, we will get the desired result of (iii). This finishes the proof of Theorem \ref{thm-4-struct}.
 \eproof

 \vspace{0.2cm}

It is easy to see that $M$ is the smallest positive integer such that $\mathcal{M}^M(\delta)=\delta$ and $\mathcal{M}^M(0)=0$ in Theorem \ref{thm-4-struct} (ii) and (iii), respectively.
Therefore we have $\langle MY(\alpha)\rangle=0$ by Theorem \ref{thm-mod-1}, which means $Y(\alpha)=\frac pM$ for some integer $0\le p<M$.
 By \eqref{struct-s-2} (\emph{resp.} \eqref{struct-s-3}), the right end points of the gaps  are $\mathcal{M}^k(\delta)$ (\emph{resp.} $\mathcal{M}^k(0)$), $1\le k\le M$. Then $\langle kY(\alpha)\rangle$, $1\le k\le M$, are different to each other. This means $p$ is co-prime to $M$. Hence, we have the following corollary.


\vspace{0.2cm}
\begin{cor}\label{cor-Y}
Let $\mathcal{M}$ and ${\mathcal S}$ be as in Theorem \ref{thm-4-struct}.  Assume that Theorem \ref{thm-4-struct} {\rm (ii)} or {\rm(iii)} is valid, then $Y(\alpha)
=\frac pM$ for some $1\le p<M$ prime to $M$. 
 \end{cor}

 We could easily conclude from Corollary \ref{cor-Y} that $\mathcal{S}$ must be of \textbf{Type I} if $Y(\alpha)\not\in\Bbb Q$, but $\mathcal{S}$ might be any one of \textbf{Types I, II} and \textbf{III} if $Y(\alpha)\in\Bbb Q$. The following Proposition \ref{prop-delta-tur} will focus on the case that $Y(\alpha)\in\Bbb Q$, and we will establish certain transference  relationship between these three types via disturbances of $\mathcal{M}$.


\vspace{0.2cm}
 Let ${\mathcal M}={\mathcal M}_{\alpha,x_1,x_2}$ be the map defined as \eqref{map.def} with $0<x_1<x_2<1$.  For each $\delta\in (-x_1,1-x_2)$, we call the transformation ${\mathcal M}_\delta:={\mathcal M}_{\alpha,x_1+\delta,x_2+\delta}$  the \emph{$\delta-$disturbance} of ${\mathcal M}$. By analogous to ${\mathcal M}$,  we use ${\mathcal S}_\delta$ to denote  the maximal invariant set for ${\mathcal M}_\delta$ in $\Bbb T_{[0,1)}\setminus [x_1+\delta,x_2+\delta)$, and  $Y_\delta$ to denote its squeezing map if ${\mathcal S}_\delta\ne \emptyset$.

 \vspace{0.2cm}

\begin{prop}\label{prop-delta-tur}
Let $\mathcal{M}$ and ${\mathcal S}$ be as in Theorem \ref{thm-4-struct}.
 Assume that Theorem \ref{thm-4-struct} {{\rm(i)}} is valid for $\mathcal{M}$ with $N\in\Bbb N$ and $Y(\alpha)=\frac pM$ for some co-prime positive integers $p<M$. Then
for each $|\delta'|<\frac {|\mathcal{S}|}M$,
\begin{enumerate}
\item [{\rm (i)}]  
${\mathcal M}_{\delta'}$  is well defined, and  Theorem \ref{thm-4-struct} {{\rm(ii)}} {(resp. {\rm(iii)})} is valid for  ${\mathcal M}_{\delta'}$  with $N$, $M$ and $\delta=|\delta'|$ when $\delta'>0$ {(resp. $\delta'<0$)}.
\item [{\rm (ii)}] 
${\mathcal S}_{\delta'}=\cup_{k=0}^{M-1} J_k$,
where 
\begin{equation}\label{delta-tur-1-0991}
J_k=\left\{\begin{array} {lll}
\ {\mathcal M}^k(0)+[\delta',\frac {|\mathcal{S}|}M)  &
{\rm if} \ \delta'\in [0,\frac {|\mathcal{S}|}M) \\
\
{\mathcal M}^k(0)+[0,\frac {|\mathcal{S}|}M+\delta') &
{\rm if} \ \delta'\in (-\frac {|\mathcal{S}|}M,0)
\end{array}\right.,\  k=0,1,\ldots,M-1.
\end{equation}
\item [{\rm (iii)}] $Y_{\delta'}(\alpha)=Y(\alpha)$, and ${\mathcal M}_{\delta'}(t)={\mathcal M}(t)$ for all $t\in {\mathcal S}_{\delta'}$.
\end{enumerate}

\vspace{0.2cm}

{Conversely,} if Theorem \ref{thm-4-struct} {{\rm (ii)}} {(resp. {\rm (iii)})} is valid for ${\mathcal M}$, 
then
${\mathcal M}_{-\delta}$ {(resp. ${\mathcal M}_\delta$)} is well defined, and
 Theorem \ref{thm-4-struct} {{\rm (i)}} is valid for ${\mathcal M}_{-\delta}$
 {(resp. ${\mathcal M}_{\delta}$)} with the same $N$.
\end{prop}

\proof

 Suppose Theorem \ref{thm-4-struct} (i) holds. It follows from \eqref{struct-s-1} that all the gaps
${\mathcal M}^{\ell-1}([\alpha,\alpha+x_2-x_1)), 1\le \ell\le N,$ are separated from $0$. Therefore, by Theorem \ref{thm-mod-1},
$$
 \big\langle \frac {\ell p}M\big\rangle =\langle Y({\mathcal M}^{\ell}(0))\rangle=\langle Y({\mathcal M}^{\ell-1}(\alpha+x_2-x_1))\rangle\ne 0
 $$
for all $1\le \ell\le N$, and $\langle Y({\mathcal M}^M(0))\rangle=\langle MY(\alpha)\rangle=0$. This leads to $ N<M$ and ${\mathcal M}^M(0)=0$.

\vspace{0.2cm}
 Since $Y$ is a bijection from ${\mathcal S}$ to $\Bbb T_{[0,1)}$ by Theorem \ref{thm-mod-1}, it is reasonable to define its inverse map $Y^{-1}$ from $\Bbb T_{[0,1)}$ to ${\mathcal S}$.
 Observe that $$Y(t)\in\{Y({\mathcal M}^{\ell}(0))=\langle\frac {\ell p}M\rangle: \ell=1,\ldots,N\}\subset\{\frac kM: k=0,1,\ldots,M-1\}$$ for all $t\in\Bbb T\setminus\mathcal{S}$.
Then $$Y^{-1}\big(Y({\mathcal M}^{k}(0))+[0,\frac1M)\big)=[{\mathcal M}^k(0),{\mathcal M}^{k}(0)+\frac {|{\mathcal S}|}M), \quad k=0,1,\ldots,M-1,$$
and therefore
\begin{equation}\label{delta-tur-add91}
{\mathcal S}=\cup_{k=0}^{M-1} [{\mathcal M}^k(0),{\mathcal M}^{k}(0)+\frac {|{\mathcal S}|}M).
\end{equation}

 (i).  The definition of ${\mathcal M}_{\delta'}$ is well defined for each $|\delta'|\in (0,\frac {|\mathcal{S}|}M)$, since it is easily seen that  $\frac {|\mathcal{S}|}M<x_1<x_2<1-\frac {|\mathcal{S}|}M$ by \eqref{delta-tur-add91} and \eqref{S-u-v}.
Set
\begin{equation}\label{S-star}
{\mathcal S}_{\delta'}^\ast=\cup_{k=0}^{M-1} J_k 
 \end{equation}
with $J_k$ in \eqref{delta-tur-1-0991}. Then ${\mathcal S}_{\delta'}^\ast\cap (\textbf{H}\cup\textbf{H}_{\delta'})=\emptyset$,
$$
{\mathcal S}_{\delta'}^\ast \cap U_{\delta'}={\mathcal S}_{\delta'}^\ast\cap U=\cup_{\mathcal{M}^k(0)\in U} J_k 
$$
and
$$
{\mathcal S}_{\delta'}^\ast \cap V_{\delta'}={\mathcal S}_{\delta'}^\ast\cap V=\cup_{\mathcal{M}^k(0)\in V} J_k, 
$$
where $\textbf{H}_{\delta'}=[x_1+\delta', x_2+\delta')$, $U_{\delta'}=[0,x_1+{\delta'})$ and $V_{\delta'}=[x_2+{\delta'},1)$.
So
\begin{equation}\label{M-delta}
{\mathcal M}_{\delta'}(t)={\mathcal M}(t)\ \ \mbox{for all }\ t\in {\mathcal S}_{\delta'}^\ast,
\end{equation}
which together with \eqref{S-star} implies ${\mathcal M}_{\delta'}({\mathcal S}_{\delta'}^\ast)={\mathcal S}_{\delta'}^\ast$,
 and thus ${\mathcal S}_{\delta'}^\ast \subset {\mathcal S}_{\delta'}$. Therefore, 
\begin{equation}\label{SUV-delta'}
{\mathcal S}_{\delta'} \cap U_{\delta'}\ne \emptyset\  \mbox{ and }\ {\mathcal S}_{\delta'} \cap V_{\delta'}\ne \emptyset.
\end{equation}

 \vspace{0.2cm}

We want to prove (i) holds, it is best to split the argument into the following two cases.

 \vspace{0.2cm}

\emph{Case 1.}  $\delta'>0$. Let $k_0\in\{0,1,\cdots, M-1\}$ be the number such that $\langle\frac {k_0 p}M\rangle=\frac {M-1}M$. Then  $J_{k_0}=[1-\frac {|\mathcal{S}|}M+\delta',1)$, and thus   \eqref{SUV-delta'} means that ${\mathcal S}_{\delta'}$ is either \textbf{Type I} or \textbf{Type II}. We may therefore let $\delta''\in [0,\delta']$ be the largest number such that $[0,\delta'')\cap {\mathcal S}_{\delta'} =\emptyset$, and let $N_1>0$ be the smallest integer
such that
\begin{equation}\label{SUV-delta'+delta''}
{\mathcal M}_{\delta'}^{N_1}(\delta'')=x_2+\delta'. 
\end{equation}

 Since $x_2+\delta'=\mathcal{M}^N(0)+\delta'\in {\mathcal S}_{\delta'}^\ast$ and ${\mathcal S}_{\delta'}^\ast$ is invariant under ${\mathcal M}_{\delta'}$,   it follows that $\delta''\in {\mathcal S}_{\delta'}^\ast$, and hence $\delta''=\delta'$ by the fact that $\delta'$ is the smallest number in ${\mathcal S}_{\delta'}^\ast$.
This means $ {\mathcal S}_{\delta'}$ is \textbf{Type II}. Moreover, combining
\eqref{M-delta}, \eqref{SUV-delta'+delta''} and the fact that $N\in\{1,\cdots,M-1\}$ is the only number satisfying ${\mathcal M}^{N}(\delta')=x_2+\delta'$, we get that $N_1=N$. 
 Thus, we may assume Theorem \ref{thm-4-struct} (ii) is valid for ${\mathcal M}_{\delta'}$ with parameters $\delta'$, $N$ and $M_1$ for some integer $M_1>N$.

  \vspace{0.2cm}


According to Corollary \ref{cor-Y}, we can write $Y_{\delta'}(\alpha)=\frac {p_1}{M_1}$ for some positive integer $p_1<M_1$ prime to $M_1$. Then
\begin{equation}\label{struyt-add1}
0=Y_{\delta'}(\delta')=Y_{\delta'}\big(\mathcal{M}_{\delta'}^{M}(\delta')\big)=\langle M Y_{\delta'}(\alpha)+ Y_{\delta'}(\delta')\rangle=\big\langle \frac {M p_1}{M_1}\big\rangle.
\end{equation}
Note that $Y_{\delta'}(t)=0$ if and only if $t\in[0,\delta']$. Then for each $1\le k< M$, $\big\langle \frac {k p_1}{M_1}\big\rangle =Y_{\delta'}\big(\mathcal{M}_{\delta'}^{k}(\delta')\big) \ne 0 $ by the observation  that $\mathcal{M}_{\delta'}^{k}(\delta')\in J_k$. This together with \eqref{struyt-add1} implies $M_1=M$. Consequently, in this case, (i) follows. 

\vspace{0.2cm}
\vspace{0.2cm}

\emph{Case 2.} $\delta'<0$. By \eqref{S-star}, $0\in{\mathcal S}_{\delta'}^\ast$, which 
 means $ {\mathcal S}_{\delta'}$ is either \textbf{Type I} or \textbf{Type III}.

 \vspace{0.2cm}

 If $ {\mathcal S}_{\delta'}$ is \textbf{Type I}, by  Theorem \ref{thm-4-struct} (i) and \eqref{M-delta}, there exists $N_2\in\Bbb N$ such that $\mathcal{M}^{N_2}(0)=\mathcal{M}_{\delta'}^{N_2}(0)=x_2+\delta'$.
 This is a contradiction due to the fact that $\mathcal{M}^{N}(0)=x_2$ and 
  \begin{equation}\label{M-delta-2022-1}
 |\mathcal{M}^{k}(0)-\mathcal{M}^{N}(0)|\ge \frac {|{\mathcal S}|}M>|\delta'|\ \mbox{for all }\ k\not\equiv N\ (\mbox{mod} M)
 \end{equation}
 by \eqref{delta-tur-add91}. Therefore, $ {\mathcal S}_{\delta'}$ is \textbf{Type III}.


\vspace{0.2cm}

Let $\sigma\in (0,-\delta']$ be the largest number such that $[1-\sigma,1)\cap {\mathcal S}_{\delta'} =\emptyset$, and $0<N_3<M_3$ be the integers such that Theorem \ref{thm-4-struct} (iii) is valid for ${\mathcal M}_{\delta'}$. Then
\begin{equation}\label{M-delta+type111+1}
\mathcal{M}_{\delta'}^{N_3}(0)=\mathcal{M}_{\delta'}^{N_3-1}(\alpha+x_2-x_1)=x_2+\delta'+\sigma
\end{equation}
 and
 \begin{equation}\label{M-delta+type111+2}
 \mathcal{M}_{\delta'}^{M_3-N_3}(x_2+\delta'+\sigma)=0,
 \end{equation}
and $N_3$ and $M_3$ are the smallest positive integers such that \eqref{M-delta+type111+1} and  \eqref{M-delta+type111+2} are valid respectively.
Combining \eqref{M-delta} and \eqref{M-delta+type111+1}, $|\mathcal{M}^{N}(0)-\mathcal{M}^{N_3}(0)|=|\delta'+\sigma|< \frac {|{\mathcal S}|}M$. This together with \eqref{M-delta-2022-1} leads to $\delta'+\sigma=0$ and $N_3=N$. Thus, by \eqref{M-delta} and \eqref{M-delta+type111+2}, $M_3=M$.   We get (i).

\vspace{0.2cm}

(ii). By Theorem \ref{thm-4-struct} and \eqref{S-star}, we have $|{\mathcal S}_{\delta'}^\ast|=|{\mathcal S}_{\delta'}|$.
Recall that ${\mathcal S}_{\delta'}^\ast\subset {\mathcal S}_{\delta'}$ and both of them {are unions of
finitely many} left-closed and right-open intervals. Then
${\mathcal S}_{\delta'}^\ast={\mathcal S}_{\delta'}$.
This proves (ii).
\vspace{0.2cm}

(iii). Observe that ${\mathcal S}_{\delta'}\cap [\alpha,\alpha+x_2-x_1)={\mathcal S}\cap [\alpha,\alpha+x_2-x_1)=\emptyset$ and
$$
\{k: [{\mathcal M}^k(0),{\mathcal M}^{k}(0)+\frac {|{\mathcal S}|}M)\subset [0,\alpha), 0\le k\le M-1\}=\{k: J_k \subset [0,\alpha), 0\le k\le M-1\}
$$
by (ii) and \eqref{delta-tur-add91}. Then $Y_{\delta'}(\alpha)=Y(\alpha)$. This together with \eqref {M-delta} 
implies (iii).

\vspace{0.2cm}

\vspace{0.2cm}

 Now, we prove the converse part. Suppose Theorem \ref{thm-4-struct} (ii) holds. Notice that $0<\delta< x_1$. Then
${\mathcal M}_{-\delta}$ is well defined and
\begin{equation}\label{M-delta-9}
 \mathcal{M}(t)={\mathcal M}_{-\delta}(t)\  \mbox{ for all }\ \ t\in\Bbb T_{[0,1)}\setminus [x_1-\delta,x_2).
\end{equation}
Therefore ${\mathcal M}_{-\delta}({\mathcal S})={\mathcal S}$.
Set
\begin{equation}\label{M-delta-10}
{\mathcal S}_{-\delta}'={\mathcal S}+[-\delta, 0].
\end{equation}
Observe that $\mathcal{S}\cap[x_1-\delta,x_2)=\mathcal{S}\cap[0,\delta)=\emptyset$. Then ${\mathcal S}_{-\delta}'\cap [x_1-\delta,x_2-\delta)=\emptyset$ and ${\mathcal M}_{-\delta}(t+\varepsilon)={\mathcal M}_{-\delta}(t)+\varepsilon$ for each $t\in {\mathcal S}$ and $\varepsilon\in [-\delta, 0]$. Thus
\begin{equation}\nonumber
{\mathcal M}_{-\delta}({\mathcal S}_{-\delta}')={\mathcal M}_{-\delta}({\mathcal S})+[-\delta, 0]
={\mathcal S}+[-\delta, 0]={\mathcal S}_{-\delta}',
\end{equation}
and therefore, we have
\begin{equation}\label{M-delta-12}
{\mathcal S}_{-\delta}'\subset {\mathcal S}_{-\delta}.
\end{equation}

\vspace{0.2cm}

Because $[0,\delta)$ is a gap of ${\mathcal M}$, \eqref{M-delta-10} and \eqref{M-delta-12} yield that  $[-\varepsilon,\delta]\subset{\mathcal S}_{-\delta}$ for some $\varepsilon>0$. Thus ${\mathcal S}_{-\delta}$ is \textbf{Type I}. Let $N_4\in\Bbb N$ be the smallest integer such that Theorem \ref{thm-4-struct} (i) is valid for  ${\mathcal M}_{-\delta}$, i.e., ${\mathcal M}_{-\delta}^{N_4}([\alpha,\alpha+x_2-x_1))=[x_1-\delta,x_2-\delta)$. Then $N_4=N$ as ${\mathcal M}_{-\delta}^{k}([\alpha,\alpha+x_2-x_1))={\mathcal M}^{k}([\alpha,\alpha+x_2-x_1))$ for all $0\le k\le N-1$ by \eqref{M-delta-9} and 
\eqref{struct-s-2}. 

\vspace{0.2cm}

 If Theorem \ref{thm-4-struct} (iii) holds, the proof is similar to that Theorem \ref{thm-4-struct} (ii) holds. We omit the details.
\eproof

\vspace{0.2cm}

\section{Piecewise linear {transformation} II: symmetric maximal invariant sets} \label{section-4+1}


This section is to investigate  
the symmetric maximal invariant set $\mathcal{E}$ of the transformation $\mathcal{M}$ defined in \eqref{e-set.def}. Based on the results 
in the previous section, we provide a characterization on the symmetric maximal invariant set via the squeezing map $Y$ {(see \eqref{biject-y.def})} in Subsection \ref{subsection-4-4}, see Theorem \ref{thm-exist-e-1}, {Propositions}  \ref{thm-exist-e-3-1} and \ref{prop-tur-E} for different situations. Especially,  we present a strict condition on $\mathcal{E}\ne\emptyset$ when $\mathcal{S}$ satisfies $\mathcal{S}\cap U\ne\emptyset$ and $\mathcal{S}\cap V\ne\emptyset$, see Lemma \ref{lem-suv}. And in Subsection \ref{subse-5-2}, we prove Proposition \ref{thm-exist-e++} and Corollary \ref{thm-exist-e-2}.

\vspace{0.2cm}

Remembering that $\widetilde{\mathcal{M}}=\widetilde{\mathcal{M}}_{\alpha, x_1,x_2}: \Bbb T_{(0,1]}\rightarrow \Bbb T_{(0,1]}$ defined in \eqref{map-conv.def} is  the  coherent map of $ {\mathcal{M}}= {\mathcal{M}}_{\alpha, x_1,x_2}$, and $\widetilde{\mathcal{S}}$ is the maximal invariant set of $\widetilde{\mathcal{M}}$ in $\Bbb T_{(0,1]}\setminus \widetilde{\textbf{H}}$, i.e. the maximal set satisfying
\begin{equation}\nonumber
 \widetilde{{\mathcal S}}\subset \Bbb T_{(0,1]}\setminus \widetilde{\textbf{H}}\  \mbox{ and }\  \widetilde{\mathcal{M}} (\widetilde{{\mathcal S}})= \widetilde{{\mathcal S}};
\end{equation}
where $\widetilde{U}=(0,x_1]$,  $\widetilde{V}=(x_2,1]$ and $\widetilde{\textbf{H}}=(x_1, x_2]$. The following Lemma \ref{lem-endpoint-change} tells us that the transformation ${\mathcal{M}}$ is \emph{equivalent} to $\widetilde{\mathcal{M}}$ via the map $\pi$ defined as in \eqref{eq.pi}. This yields that all properties of ${\mathcal{M}}$ (\emph{resp.}
 ${\mathcal{S}}$) in Subsections \ref{subsection-4-2} and \ref{subsection-4-1} can be transferred to that of $\widetilde{\mathcal{M}}$ (\emph{resp.} $\widetilde{\mathcal{S}}$), with a minimum of effort by changing all the left-closed and right-open intervals into the left-open and right-closed ones.
\vspace{0.2cm}

\begin{lem}\label{lem-endpoint-change}
 Let ${\mathcal{M}}$, $\widetilde{\mathcal{M}}$ and $\pi$ be the maps defined in \eqref{map.def}, \eqref{map-conv.def} and \eqref{eq.pi}, respectively.  Then for every $E\subset \Bbb T_{[0,1)}$ being the finite union of some left-closed and right-open intervals,
\begin{equation}\label{eq-end-exch}
\widetilde{\mathcal{M}}(\pi(E))=\pi({\mathcal{M}}(E)).
\end{equation}
Especially, we have
\begin{equation}\nonumber
\widetilde{\mathcal{S}}=\pi(\mathcal{S}).
\end{equation}
\end{lem}

\proof
%
%
%
It is easy to see that the commutative diagram \eqref{eq-end-exch} is valid when $E$ is a left-closed and right-open interval only contained in $U$ (\emph{resp.} $V$ or $\textbf{H}$), and then \eqref{eq-end-exch} also holds for general $E$ since it can be decomposed into a disjoint union of left-closed and right-open intervals coming from $U$, $V$ and $\textbf{H}$.

 By Theorem \ref{thm-s-1} and \eqref{eq-end-exch}, we can  easily check that $\widetilde{\mathcal{M}}(\pi(\mathcal{S}))
=\pi({\mathcal{M}}(S))=\pi(\mathcal{S})$, so $\pi(\mathcal{S})\subset \widetilde{\mathcal{S}}$. Moreover, if $\pi(\mathcal{S})\ne \widetilde{\mathcal{S}}$, there is an interval $I=[t_1,t_2)\subset \Bbb T_{[0,1)}$ such that $I\cap \mathcal{S}= \emptyset$ and $\pi(I)\cap \widetilde{\mathcal{S}}\ne \emptyset$. Nevertheless, by Theorem \ref{thm-s-1}, $\mathcal{M}^N(I)\subset \textbf{H}$ for some $N\in\Bbb N$, now \eqref{eq-end-exch} yields that     $\widetilde{\mathcal{M}}^N(\pi(I))=\pi({\mathcal{M}}^N(I))\subset \widetilde{\textbf{H}}$, which implies $\pi(I)\cap \widetilde{\mathcal{S}}= \emptyset$, a contradiction. This forces $\widetilde{\mathcal{S}}=\pi(\mathcal{S}).$
 \eproof

\vspace{0.2cm}

By virtue of Lemma \ref{lem-endpoint-change}, hereafter without confusion, we may denote $\widetilde{E}=\pi(E)$ for $E$ being the union of finitely many left-closed and right-open intervals. 


\subsection{Structure for symmetric maximal invariant set $\mathcal{E}$}\label{subsection-4-4} $\empty$


\begin{thm}\label{thm-exist-e-1}
Let $\mathcal{M}$ be the map defined in \eqref{map.def} with the maximal invariant set ${\mathcal S}\ne\emptyset$.  If $Y(\alpha) \not\in\Bbb Q$, then  $\mathcal{E}=\emptyset$. Furthermore, there exists $K\in \Bbb N$ such that for each $t\in\mathcal{S}$, $\widetilde{\mathcal{M}}^m(1-\mathcal{M}^n(t))\in \widetilde{\textbf{H}}$ for some $m,n\le K$.
\end{thm}
\proof
If $Y(\alpha) \not\in\Bbb Q$, then it follows from Proposition \ref{prop-s-rational} and Corollary \ref{cor-Y}  that $\mathcal{S}$ must be of \textbf{Type I}, whence we can choose the number $N\in\Bbb N $ as in Theorem \ref{thm-4-struct} (i),
and let $t_0=Y^{-1}(1-Y(\alpha))$, where $Y^{-1}$ is the inverse map of $Y$ restricted to $\mathcal{S}$. Note that $1-Y(\alpha)\not\in\{\langle nY(\alpha)\rangle: n\in\Bbb N\}$ by $Y(\alpha) \not\in\Bbb Q$. Then
\begin{equation}\nonumber
t_0\not\in\{\mathcal{M}^k(0):k=1,2, \ldots,N\},
\end{equation}
and therefore, $(t_0-\varepsilon,t_0+\varepsilon)\subset \mathcal{S}$ for some sufficient small $0<\varepsilon<x_2-x_1$ by Theorem \ref{thm-4-struct} (i).

\vspace{0.2cm}

Observe that $Y(\mathcal{M}(t_0))=\langle Y(t_0)+Y(\alpha)\rangle=0$ by Theorem \ref{thm-mod-1}. Then $\mathcal{M}(t_0)=0$, and therefore we have either $t_0=1-\alpha$ or $t_0=1-(\alpha+x_2-x_1)$, which means $J=(t_0-\varepsilon,t_0+\varepsilon)\cap(1-\alpha-x_2+x_1,1-\alpha)$ is an interval with length $|J|=\varepsilon$, and $Y(J)$ is an interval such that $|Y(J)|= \varepsilon/{|\mathcal{S}|}$.

\vspace{0.2cm}

Since $\{\langle nY(\alpha)\rangle: n\in\Bbb N\}$ is dense in $\Bbb T_{[0,1)}$,
there exists $K> N$  such that for each $t\in \mathcal{S}$,
  $$Y(\mathcal{M}^n(t))=\langle Y(t)+nY(\alpha)\rangle\in Y(J)$$
  for some $n\le K$.
This leads to $\mathcal{M}^n(t)\in J$, and $1-\mathcal{M}^n(t)\in (\alpha,\alpha+x_2-x_1)$, in other  words,   
$t\not\in {\mathcal{E}}$   by the definition of ${\mathcal{E}}$. Hence, 
$\mathcal{E}=\emptyset$ by the arbitrariness of $t$.

Moreover, by Lemma \ref{lem-endpoint-change} and
 Theorem \ref{thm-4-struct} (i), 
 $$\widetilde{\mathcal{M}}^{N-1}(1-\mathcal{M}^n(t))\in \widetilde{\mathcal{M}}^{N-1}((\alpha,\alpha+x_2-x_1])=\pi({\mathcal{M}}^{N-1}([\alpha,\alpha+x_2-x_1)))=\widetilde{\textbf{H}}.$$
We complete the proof of Theorem \ref{thm-exist-e-1}.
\eproof

\vspace{0.2cm}

By Theorem \ref{thm-exist-e-1},  it remains  to characterize the symmetric maximal invariant set $\mathcal{E}$ {when} $Y(\alpha)\in\Bbb Q$. First, we investigate the cases that ${\mathcal S}\subset U$ and ${\mathcal S}\subset V$. 


\begin{prop}\label{thm-exist-e-3-1}
Let $\mathcal{M}$ be the map defined in \eqref{map.def} with the maximal invariant set ${\mathcal S}\ne\emptyset$ and $Y(\alpha) \in \Bbb Q$. Denote $Y(\alpha)=\frac pM$ be a simple fraction. Then:
\begin{enumerate}
   \item [{\rm(i)}] If $\mathcal{S}\subset V$, then $\mathcal{E}\ne\emptyset$ if and only if $x_2<\frac 1{2M}$. 
   In this case,
    \begin{equation}\nonumber
   \mathcal{E}=\bigcup_{k=0}^{M-1}\Big[\frac kM+x_2,\frac {k+1}M-x_2\Big).
   \end{equation}
   \item [{\rm(ii)}]  If $\mathcal{S}\subset U$, then $\mathcal{E}\ne\emptyset$ if and only if $x_1>1-\frac 1{2M}$.
   In this case,
    \begin{equation}\nonumber
   \mathcal{E}=\bigcup_{k=0}^{M-1}\Big[\frac kM+1-x_1,\frac {k+1}M-1+x_1\Big).
   \end{equation}
\end{enumerate}
\end{prop}


 \proof
(i). {It follows from Proposition \ref{prop-s-rational} (ii)} that $\alpha=Y(\alpha)=\frac {p}{M}$, $x_2<\frac 1M$ and
\begin{equation}\label{struct-s-rational-21}
{\mathcal S}= \bigcup_{k=0}^{M-1}\big[x_2+\frac k{M}, \frac {k+1}{M}\big).
\end{equation}
  If $\mathcal{E}\ne \emptyset$, then it follows from $\mathcal{E}\subset\mathcal{S}\cap {(1-\widetilde{\mathcal{S}})}$ that {$\mathcal{S}\cap (1-\widetilde{\mathcal{S}})\ne \emptyset$}. This together with \eqref{struct-s-rational-21} implies $x_2<\frac 1{2M}$.
\vspace{0.2cm}

Conversely, if $x_2<\frac 1{2M}$, then
\begin{equation}\nonumber
\mathcal{S}\cap {(1-\widetilde{\mathcal{S}})}=\bigcup_{k=0}^{M-1}\Big[\frac kM+x_2,\frac {k+1}M-x_2\Big).
\end{equation}
Observe that $\mathcal{M}(t)=\langle t+\frac {p}M\rangle$ when $t\in \mathcal{S}$, and $1-\widetilde{\mathcal{M}}(1-t)=1-\langle1-t+\frac pM\rangle^\ast=\langle t-\frac{p}M\rangle$ when $t\in 1-\widetilde{\mathcal{S}}$. Then
{$\mathcal{S}\cap (1-\widetilde{\mathcal{S}})$}
is an invariant set for the maps $\mathcal{M}$ and $1-\widetilde{\mathcal{M}}(1-\cdot)$. Therefore, $\mathcal{E}=\mathcal{S}\cap {(1-\widetilde{\mathcal{S}})}
 \ne \emptyset$, (i) follows.

\vspace{0.2cm}

(ii). The proof is similar as (i). 
\eproof

\vspace{0.2cm}
\vspace{0.2cm}


{We next study} the structure of the symmetric maximal invariant set $\mathcal{E}$ when $\mathcal{S}$ satisfies
\begin{equation}\label{SUV-S5}
{\mathcal S}\cap U\ne\emptyset\ \mbox{and }\ {\mathcal S}\cap V\ne\emptyset.
\end{equation}


\begin{lem}\label{lem-suv}
 Suppose ${\mathcal S}$ satisfies \eqref{SUV-S5} and $Y(\alpha)=\frac pM$ for some co-prime
$p,M\in\Bbb N$. If $\mathcal{E}\ne \emptyset$, then $N=M-1$, where $N$ is the number in Theorem \ref{thm-4-struct}.
\end{lem}
\proof

As $\emptyset\not=\mathcal{E}\subset\mathcal{S},$  we might assume the assertion (i) of Theorem \ref{thm-4-struct} holds for $\mathcal{M}$ {without loss of generalization}\footnote{{
If the assertion (ii) of Theorem \ref{thm-4-struct} holds for $\mathcal{M}$, Proposition \ref{prop-delta-tur} yields that the result of  Theorem \ref{thm-4-struct} (i) holds for ${\mathcal M}_{-\delta}$, and ${\mathcal S}\subset {\mathcal S}_{-\delta}$ and ${\mathcal M}(t)={\mathcal M}_{-\delta}(t)$ for all $t\in{\mathcal S}$; whence Lemma \ref{lem-endpoint-change} guarantees that $\widetilde{{\mathcal S}}\subset \widetilde{{\mathcal S}}_{-\delta}$ and $\widetilde{{\mathcal M}}(t)=\widetilde{{\mathcal M}}_{-\delta}(t)$ for all $t\in\widetilde{{\mathcal S}}$, from which one clearly has $\mathcal{E}\subset\mathcal{E}_{-\delta}$, where $\mathcal{E}_{-\delta}$ stands for the symmetric maximal invariant set for $\mathcal{M}_{-\delta}$. Similarly, if the assertion (iii) of Theorem \ref{thm-4-struct} holds, then the result of Theorem \ref{thm-4-struct} (i) also holds for ${\mathcal M}_{\delta}$ and $\mathcal{E}\subset\mathcal{E}_{\delta}$. At the same time, Proposition \ref{prop-delta-tur} ensures that $N$ is invariant under above disturbances of $\mathcal{M}$. So we may always assume that Theorem \ref{thm-4-struct} (i) holds for $\mathcal{M}$.}}.
By Proposition \ref{prop-delta-tur} (ii), we have ${\mathcal M}^M(0)=0$ and 
\begin{equation}\label{delta-tur-01-1}
{\mathcal S}=\bigcup_{k=0}^{M-1} J_k, 
\end{equation}
where $J_k=[{\mathcal M}^k(0),{\mathcal M}^k(0)+\frac {|\mathcal{S}|}M)$, $k=0,1,\ldots,M-1$, are {pairwise disjoint}, and
\begin{equation}\label{delta-tur-01+1}
J_{k+1}={\mathcal M}(J_k),\ \mbox{ for all } \ k=0,1,\ldots,M-1,
\end{equation}
here we denote $J_M=J_0$.
%
Set $E=\mathcal{E}\cap J_0$, and $F=(1-\mathcal{E})\cap \pi(J_0)$.
{Note that $\mathcal{E}\subset{\mathcal S}$ and $1-\mathcal{E}\subset{\widetilde{\mathcal S}}$ are invariant sets for $\mathcal{M}$ and $\widetilde{\mathcal{M}}$ respectively, it then follows from \eqref{delta-tur-01-1} and \eqref{delta-tur-01+1} that}
\begin{equation}\label{delta-tur-02-1}
\mathcal{E}=\bigcup_{k=0}^{M-1}{\mathcal M}^k(E) =\bigcup_{k=0}^{M-1} {\big({\mathcal M}^k(0)+E\big)}
\end{equation}
 and
\begin{equation}\label{delta-tur-03-1}
 1-\mathcal{E}=\bigcup_{k=0}^{M-1}{\widetilde{\mathcal M}}^k(F) =\bigcup_{k=0}^{M-1} {\big({\mathcal M}^k(0)+F\big)}.
 \end{equation}
 Let $k_0\in\{0,1,\cdots, M-1\}$ be the number such that $k_0 p\equiv M-1\ (\mbox{mod}\ M)$.  Then
 $$J_{k_0}=Y^{-1}\big([\frac {M-1}M,1)\big)=[1-\frac {|\mathcal{S}|}M,1),$$ which means
 $M^{k_0}(0)=1-\frac {|\mathcal{S}|}M$, and hence \eqref{delta-tur-03-1} yields that
\begin{equation}\label{delta-tur-04-1}
E=\Big(1-\bigcup_{k=0}^{M-1} {\big({\mathcal M}^k(0)+F\big)}  \Big)\cap J_0=
1-{\mathcal M}^{k_0}(0)-F.
\end{equation}

\vspace{0.2cm}

If ${\mathcal M}^{M-1}(0)\in U$, then ${\mathcal M}^{M-1}(0)=1-\alpha-x_2+x_1$, by \eqref{delta-tur-03-1},
\begin{eqnarray} \nonumber
{\mathcal M}^{M-1}(0)+F
& = & \big(1-\mathcal{E}\big)\cap \pi(J_{M-1}) \nonumber\\
& = & \big(1-\mathcal{E}\big)\cap (1-\alpha-x_2+x_1, 1-\alpha-x_2+x_1+\frac {|\mathcal{S}|}M].\nonumber
\end{eqnarray}
This together with \eqref{delta-tur-04-1} implies
\begin{equation}\label{delta-tur-0-1121}
\mathcal{E}\cap [\alpha+x_2-x_1-\frac {|\mathcal{S}|}M, \alpha+x_2-x_1)=1-{\mathcal M}^{M-1}(0)-F=\alpha+x_2-x_1-\frac {|\mathcal{S}|}M+E,
\end{equation}
which means $\frac {|\mathcal{S}|}M > x_2-x_1$ since $E\ne\emptyset$ by \eqref{delta-tur-02-1}. 
Observe that $M^{k_0}(0)=1-\frac {|\mathcal{S}|}M\in V$. Then $$\alpha-\frac {|\mathcal{S}|}M+E=\mathcal{M}^{k_0+1}(0)+E=\mathcal{E}\cap[\alpha-\frac {|\mathcal{S}|}M,\alpha),$$
and therefore the left hand side of \eqref{delta-tur-0-1121} becomes
\begin{equation}\label{delta-tur-0-1121-1}
\mathcal{E}\cap [\alpha+x_2-x_1-\frac {|\mathcal{S}|}M, \alpha)=\big(\alpha-\frac {|\mathcal{S}|}M+E\big)\cap [\alpha+x_2-x_1-\frac {|\mathcal{S}|}M, \alpha).
\end{equation}
Comparing the right hand {side} of \eqref{delta-tur-0-1121} and \eqref{delta-tur-0-1121-1}, one has $E+x_2-x_1=E\cap[x_2-x_1,\frac {|\mathcal{S}|}M)$, which obviously leads to $E=\emptyset$, and thus $\mathcal{E=\emptyset}$ by \eqref{delta-tur-02-1}. This is a contradiction. Therefore we have ${\mathcal M}^{M-1}(0)\in V$, and then ${\mathcal M}^{M-1}(0)=1-\alpha$.

\vspace{0.2cm}

Now we prove $N=M-1$ by {the proof of} contradiction. Assume that $1\le N<M-1$. Recall that the gaps ${\mathcal M}^{k}([\alpha,\alpha+x_2-x_1))=[{\mathcal M}^{k+1}(0)-x_2+x_1,{\mathcal M}^{k+1}(0))$, $k=0,1,\ldots,N-1$, are located at the left hand {side} of $J_{k+1}$, $k=0,1,\ldots,N-1$, respectively.
Then the gaps are separated from ${\mathcal M}^{M-1}(0)$, and therefore, 
$J_{k_0-1}=\big[{\mathcal M}^{M-1}(0)-\frac {|\mathcal{S}|}M,{\mathcal M}^{M-1}(0)\big)=[1-\alpha-\frac{|\mathcal{S}|}M, 1-\alpha)$. 
Thus,
${\mathcal M}^{k_0-1}(0)+F=\big(1-\mathcal{E}\big)\cap (1-\alpha-\frac {|\mathcal{S}|}M, 1-\alpha]$ by \eqref{delta-tur-03-1}. This together with \eqref{delta-tur-04-1} implies
\begin{equation}\label{delta-tur-0-1122}
\mathcal{E}\cap [\alpha, \alpha+\frac {|\mathcal{S}|}M)=1-{\mathcal M}^{k_0-1}(0)-F=\alpha+E.
\end{equation}
This means $\frac {|\mathcal{S}|}M > x_2-x_1$ since the right hand side is not an empty set. And then the left hand side of \eqref{delta-tur-0-1122} becomes
$$
\mathcal{E}\cap [\alpha+x_2-x_1, \alpha+\frac {|\mathcal{S}|}M)=(\alpha+x_2-x_1+E)\cap [\alpha+x_2-x_1, \alpha+\frac {|\mathcal{S}|}M).
$$
Comparing the right hand {side} of the above equation and \eqref{delta-tur-0-1122} leads to $E=\emptyset$. This contradicts to $\mathcal{E\ne\emptyset}$. Hence, $N=M-1$. {The lemma is proved.}
\eproof

\vspace{0.2cm}



Now, we present the structure of the symmetric maximal invariant set $\mathcal{E}$ when $Y(\alpha)\in\Bbb Q$ and $\mathcal{S}$ satisfies \eqref{SUV-S5} through disturbances of the map $\mathcal{M}=\mathcal{M}_{\alpha,x_1,x_2}$. {As in} Proposition \ref{prop-delta-tur}, we let $\mathcal{M}_{\delta}=\mathcal{M}_{\alpha,x_1+\delta,x_2+\delta}$ be the $\delta -$disturbance of $\mathcal{M}$ for some $\delta$ under the premise that it is well defined. And we let $\mathcal{S}_\delta$ ($\mathcal{E}_\delta$) be the (symmetric) maximal invariant set of $\mathcal{M}_{\delta}$. 



\begin{prop}\label{prop-tur-E}
Let $\mathcal{M}=\mathcal{M}_{\alpha, x_1,x_2}$ be the map defined in \eqref{map.def} with the maximal invariant set $\mathcal{S}$ satisfying \eqref{SUV-S5} and $Y(\alpha)=\frac pM$ for some co-prime integers $1\le p<M$.
Suppose that
Theorem \ref{thm-4-struct} {\rm (i)} is valid for $\mathcal{M}$ with  $N=M-1$. Let $\Delta=x_2-x_1$,
then:
\begin{enumerate}
   \item [{\rm (i)}] $\Delta<\frac 1{M-1}$ and $|\mathcal{S}|=1-(M-1)\Delta$.
   \item [{\rm (ii)}]  $x_1=1-\frac {p}M(1+\Delta)$, $x_2=\frac {M-p}M(1+\Delta)$ and $\alpha=\frac pM-\frac {M-p}M\Delta$.
   \item [{\rm (iii)}]
For each $|\delta|<\frac{|\mathcal{S}|}M$, $\mathcal{M}_{\delta}=\mathcal{M}_{\alpha,x_1+\delta,x_2+\delta}$ is well defined, and $\mathcal{E}_\delta\ne \emptyset$ if and only if $ |\delta|<\frac{|\mathcal{S}|}{2M}$. In this case,
$${\mathcal E}_\delta=\mathcal{S}_\delta \cap \big(1-\widetilde{\mathcal{S}_\delta}\big)
=\bigcup_{k=0}^{M-1} \Big(\frac kM(1+\Delta)+[|\delta|,\frac {|\mathcal{S}|}M-|\delta|)\Big).
$$
\end{enumerate}
\end{prop}



\proof
(i). By Theorem \ref{thm-4-struct} (i), the map $\mathcal{M}$ has $M-1$ gaps with length $\Delta$, which yields that $|\mathcal{S}|=1-(M-1)\Delta$, and therefore $\Delta<\frac 1{M-1}$. (i) follows.
\vspace{0.2cm}

(ii). By Proposition \ref{prop-delta-tur} (ii) (for the case $\delta'=0$), we have ${\mathcal S}=\cup_{k=0}^{M-1} J_k$, where
$$J_k=[{\mathcal M}^k(0),{\mathcal M}^k(0)+\frac {|\mathcal{S}|}M),\ \ 0\le k\le M-1,$$
which together with the $M-1$ gaps {tile} $[0,1)$, alternatively. Therefore,
\begin{equation}\label{delta-map-06-1}
\mathcal{M}^k(0)
=\langle \frac {kp}M\rangle (1+\Delta),\ 0\le k\le M-1.
\end{equation}
{Thus, $$\alpha=\mathcal{M}(0)-(x_2-x_1)=\frac pM-\frac {M-p}M\Delta,$$ and
$x_2=\mathcal{M}^{M-1}(0)=\frac {M-p}M (1+\Delta)$ and $x_1=x_2-\Delta=1-\frac {p}M (1+\Delta)$, (ii) follows.}

\vspace{0.2cm}

(iii).  
We only prove (iii) for the case that $\delta\ge 0$ since the proof for $\delta<0$ is similar. By Proposition \ref{prop-delta-tur} (i) and (ii),
$\mathcal{M}_\delta$ is well defined and
\begin{equation}\nonumber
\mathcal{S}_\delta=\bigcup_{k=0}^{M-1}\Big(\mathcal{M}^k(0)+[\delta,\frac {|\mathcal{S}|}M)\Big). 
\end{equation}
{By \eqref{delta-map-06-1},}
\begin{eqnarray}\nonumber
1-\widetilde{\mathcal{S}_\delta}
& = & \bigcup_{k=0}^{M-1}\Big[1-\mathcal{M}^k(0)-\frac {|\mathcal{S}|}M,  1-\mathcal{M}^k(0)  -\delta  \Big)\\ \nonumber
& = & \bigcup_{k=0}^{M-1}\Big[\mathcal{M}^{M-1-k}(0),  \mathcal{M}^{M-1-k}(0) +\frac {|\mathcal{S}|}M -\delta  \Big). 
\end{eqnarray}
 Thus
\begin{equation}\label{eq-map-sset-21171}
\mathcal{E}_\delta \subset \mathcal{S}_\delta \cap \big(1-\widetilde{\mathcal{S}_\delta} \big)=\emptyset\  {\mbox{ whenever }\ \delta\ge \frac {|\mathcal{S}|}{2M}};
\end{equation}
{and, for $\delta< \frac {|\mathcal{S}|}{2M}$,}
\begin{equation}\label{delta-tur-E-4}
\mathcal{S}_\delta \cap \big(1-\widetilde{\mathcal{S}_\delta}\big)
 = \bigcup_{k=0}^{M-1}\Big[\mathcal{M}^k(0)+\delta,\mathcal{M}^k(0)+\frac {|\mathcal{S}|}M-\delta\Big).
\end{equation}
 We obtain that
$\mathcal{M}_\delta \Big(\mathcal{S}_\delta \cap \big(1-\widetilde{\mathcal{S}_\delta}\big) \Big)=\mathcal{S}_\delta \cap \big(1-\widetilde{\mathcal{S}_\delta}\big)$, since $\mathcal{M}_\delta(t)=\mathcal{M}(t)$ holds for all $t\in \mathcal{S}_\delta$ by Proposition \ref{prop-delta-tur} (iii). By Lemma \ref{lem-endpoint-change},
 $\widetilde{\mathcal{M}_\delta} \Big(\widetilde{\mathcal{S}_\delta} \cap \big(1-\mathcal{S}_\delta\big) \Big)=\widetilde{\mathcal{S}_\delta} \cap \big(1-\mathcal{S}_\delta\big)$. Therefore,
\begin{equation}\label{delta-tur-E-61}
\mathcal{E}_\delta=\mathcal{S}_\delta \cap \big(1-\widetilde{\mathcal{S}_\delta} )=\bigcup_{j=0}^{M-1} \Big(\frac jM(1+\Delta)+[\delta,\frac {|\mathcal{S}|}M-\delta)\Big),
\end{equation}
where the last equality follows from \eqref{delta-tur-E-4} and \eqref{delta-map-06-1}. Combining \eqref{eq-map-sset-21171} and \eqref{delta-tur-E-61} proves (iii).
\eproof

\vspace{0.2cm}

\subsection{Proofs of Proposition \ref{thm-exist-e++} and Corollary \ref{thm-exist-e-2}}\label{subse-5-2} $\empty$


\vspace{0.2cm}

\noindent{\bf Proof of Proposition \ref{thm-exist-e++}.} Apparently, $Y(\alpha)\in\Bbb Q$ is a consequence of Theorem \ref{thm-exist-e-1}.
 Notice that  the statement (i) (\emph{resp.} (ii)) of  Proposition \ref{thm-exist-e++} is equivalent to $\mathcal{E}\ne\emptyset$ when $\mathcal{S}\subset V$ (\emph{resp.} $\mathcal{S}\subset U$) by Propositions  \ref{prop-s-rational} and \ref{thm-exist-e-3-1}.
Therefore, we only need to prove the part related to (iii) in  Proposition \ref{thm-exist-e++}.

\vspace{0.2cm}

Assume $\mathcal{S}\cap U\ne\emptyset$, $\mathcal{S}\cap V\ne\emptyset$ and $\mathcal{E}\ne\emptyset$.
Obviously, we have $Y(\alpha)\ne 0$, and therefore $M>1$. 
 Set $\Delta=x_2-x_1$. Using Theorem \ref{thm-4-struct} and Proposition \ref{prop-delta-tur}, it can be seen that $\Delta<\frac 1{M-1}$, and there exists $|\delta|<\frac{1-(M-1)\Delta}{2M}$ such that  Theorem \ref{thm-4-struct} (i) is valid for $\mathcal{M}_{\delta}=\mathcal{M}_{\alpha,x_1+\delta,x_2+\delta}$.
Observe that $|\mathcal{S}_\delta|={1-(M-1)\Delta}$. Then (iii) can be obtained by Proposition \ref{prop-tur-E}.

\vspace{0.2cm}

{Conversely, assume} that (iii) is valid. Set
$$
I_k=\frac kM(1+\Delta)+[|\delta|,\frac {1-(M-1)\Delta}M-|\delta|),\ k=0,1,\ldots, M-1,
$$
and $E=\cup_{k=0}^{M-1} I_k$. Then we may verify that $I_k\subset U$ when $0\le k< M-N$, and $I_k\subset V$ when $M-N\le k< M$. Thus
 \begin{equation}\label{E-structu-1011}
 \mathcal{M}(I_k)=I_{k+N}\ \mbox{ for all } k=0,1,\ldots, M-1.
 \end{equation}
Here we denote $I_j=I_k$ if $j\equiv k$ (mod $M$). This means $\mathcal{M}(E)=E$, $E\cap \textbf{H}=\emptyset$ and $\widetilde{E}=1-E$, and then $E\cap (1-\widetilde{\textbf{H}})=\emptyset$ and $\widetilde{\mathcal{M}}(1-E)=1-E$ by Lemma \ref{lem-endpoint-change}. Thus,
$E\subset\mathcal{E}$. Therefore by  Lemma \ref{lem-suv} and Proposition \ref{prop-tur-E}, we have $E=\mathcal{E}$. So it remains to prove $Y(\alpha)=\frac NM$. And this is followed by \eqref{E-structu-1011} and Theorem \ref{thm-mod-1}. We complete the proof of Proposition  \ref{thm-exist-e++}.
\eproof


\vspace{0.2cm}
\vspace{0.2cm}

\noindent{\bf Proof of Corollary \ref{thm-exist-e-2}.}
{If ${\mathcal E}\ne\emptyset$, then only one of assertions (i), (ii) and (iii) in Proposition \ref{thm-exist-e++} holds. If  (i) or (ii) holds, it follows from Propositions \ref{prop-s-rational} and \ref{thm-exist-e-3-1} that  \eqref{eq-sym-e-set} holds; and if the assertion (iii)  in Proposition \ref{thm-exist-e++} holds,  then Propositions \ref{prop-delta-tur} and \ref{prop-tur-E} implies that \eqref{eq-sym-e-set} holds.}


\vspace{0.2cm}

 Recall that
\begin{equation}\nonumber
\mathcal{E}=\cup_{k=0}^{M-1} I_k,
\end{equation}
with $I_k=\frac kM+[x_2,\frac 1M-x_2)$, $I_k=\frac kM+[1-x_1,x_1-\frac {M-1}M)$ or $I_k=\frac kM(1+\Delta)+[|\delta|,\frac {1-(M-1)\Delta}M-|\delta|)$, $k=0,1,\ldots, M-1$, corresponding to the case that the requirement (i), (ii) or (iii) of Proposition \ref{thm-exist-e++} is valid, respectively. It is easy to check that
 \begin{equation}\label{E-structu-1011-11}
 \mathcal{M}(I_k)=I_{k+N}\ \mbox{ for all } k=0,1,\ldots, M-1,
 \end{equation}
where we denote $I_j=I_k$ if $j\equiv k$ (mod $M$). Observe that $1-I_k=\pi(I_{M-k-1}), k=0,1,\ldots, M-1$. This together with \eqref{E-structu-1011-11} leads to \eqref{eq-map-invers} immediately.
\eproof

\section{Sufficient and necessary conditions for Gabor frames $\mathcal{G}(H_c;a)$} \label{section-5}
{The main purpose of this section is to prove Theorem \ref{thm-gabor-eq-e-set}. For this, we first prove Lemma \ref{lem-Y-map}, and then } present two propositions ({Propositions} \ref{prop-gabor-suff-e-set} and \ref{prop-lin-e-set}) on the connections between $\mathcal{G}(H_c;a)$ and $\mathcal{M}=\mathcal{M}_{\alpha, x_1,x_2}$.


\vspace{0.2cm}

Let $a,c$ be as in \eqref{a-c-condition-1.3}, and let $\mathcal{M}:=\mathcal{M}_{\alpha, x_1,x_2}$ be the map defined in \eqref{map.def} with $x_1,x_2$ and $\alpha$ in \eqref{alpha-xx} and \eqref{alpha-xx-0}.
Then
\begin{equation}\label{alpha-x-x-1+2022}
\langle\alpha+x_2-x_1\rangle=\big\langle\frac {\lfloor c\rfloor+1}{a}\big\rangle\ \mbox{ if }\  U\ne\emptyset,
\end{equation}
 and therefore we have
\begin{equation}\label{alpha-x-x-1}
\mathcal{M}(t)=\Big\langle t+ \frac{\lfloor c\rfloor+1} a \Big\rangle
 \mbox{ if }
 t\in U \mbox{  and } \mathcal{M}(t)=\Big\langle t+ \frac{\lfloor c\rfloor} a \Big\rangle
 \mbox{ if }
 t\in V.
\end{equation}


\vspace{0.2cm}

\noindent{\bf Proof of Lemma \ref{lem-Y-map}.}
If $\mathcal{S}\subset V$, then $Y(\alpha)=\alpha \in \Bbb Q$ by
Proposition \ref{prop-s-rational}, and then
$a\in\Bbb Q$ by \eqref{alpha-xx-0}.
If $\mathcal{S}\subset U$, then {Proposition \ref{prop-s-rational} yields that} $Y(\alpha)=\langle\alpha+x_2-x_1 \rangle\in \Bbb Q$, and then
$a\in\Bbb Q$ by \eqref{alpha-x-x-1+2022}.
Therefore, {it suffices to deal with $\mathcal{S}$ for which}
\begin{equation}\nonumber
{\mathcal S}\cap U\ne\emptyset\ \mbox{and }\ {\mathcal S}\cap V\ne\emptyset.
\end{equation}
In this case, by \eqref{alpha-xx} and \eqref{alpha-xx-0},
we have $x_1=\frac {a+\langle c\rangle-1}{a}$, $x_2=\frac {\langle c\rangle}{a}$ and $\alpha=\big\langle\frac {\lfloor c\rfloor}{a}\big\rangle$.

\vspace{0.2cm}

\emph{Sufficiency}. Assume $a \in\Bbb Q$. Then $\alpha,\alpha+x_2-x_1\in\Bbb Q$. 
This means for each $t\in\mathcal{S}$, $\{\mathcal{M}^n(t):n\in\Bbb N\}$ is a finite set, and so is $\{\langle Y(t)+nY(\alpha)\rangle:n\in\Bbb N\}$ by Theorem \ref{thm-mod-1}. Therefore $Y(\alpha)\in\Bbb Q$. 

\vspace{0.2cm}

\emph{Necessity}. Assume $a \not\in\Bbb Q$ and $Y(\alpha)=\frac pM$ for some co-prime integers $1\le p<M$.  
Recall that $\alpha$, $x_2-x_1$ and $Y(\alpha)$ are invariant under the disturbance of the map $\mathcal{M}$ for some small $\delta$ by Proposition \ref{prop-delta-tur}. {There is no loss of generality in assuming that Theorem \ref{thm-4-struct} (i) holds. By} Proposition \ref{prop-delta-tur} (ii) ({taking} $\delta'=0$),
\begin{equation}\nonumber
\alpha=\mathcal{M}(0)-(x_2-x_1)=p\frac {1-N(x_2-x_1)}M+s(x_2-x_1),
\end{equation}
where $N$ is the number in Theorem \ref{thm-4-struct} (i), and $s$ is the cardinality of $\{k: \langle\frac {kp}M\rangle<\frac pM, 1\le k\le N\}$.
Substituting $\alpha=\langle \frac {\lfloor c\rfloor}a\rangle$ and $x_2-x_1=\frac{1-a}a$ into above equation, we obtain
$$
\frac {\lfloor c\rfloor}a-\Big\lfloor  \frac {\lfloor c\rfloor}a  \Big\rfloor=
\frac {p-sM+pN}M +\frac {sM-pN}{Ma}.
$$
This means ${\lfloor c\rfloor}=\frac {sM -pN}{M}$, and then $pN$ can be divided exactly by $M$, which contradicts {to the fact} that $1\le N<M $ are co-prime integers. Therefore, $Y(\alpha)\not\in\Bbb Q$. We complete the proof of the lemma.
\eproof

\vspace{0.2cm}
\vspace{0.2cm}

To prove Theorem \ref{thm-gabor-eq-e-set}, we need some preparations. Let 
$\widetilde{\mathcal{M}}=\widetilde{\mathcal{M}}_{\alpha, x_1,x_2}$ be the coherent map of $ {\mathcal{M}}$  defined in \eqref{map-conv.def}. 
For each $t\in \Bbb T_{[0,1)}$, set
\begin{equation}\label{rabcnewplus.def-2}
\mathcal{L}^{\{i\}}(t)=\mathcal{L}^{\{i\}}_{\alpha, x_1,x_2}(t):=\left\{\begin{array} {lll}
\mathcal{M}(t)  &
{\rm if} \ i=1, \\
1-\widetilde{\mathcal{M}}(1-t) &
{\rm if} \ i=-1.
\end{array}\right.
\end{equation}
Let $\Sigma:=\{1,-1\}$,
$\Sigma^n:=\overbrace{\Sigma\times\cdots\times \Sigma}^n$,
$1\le n<\infty$,
be the $n$ copies of $\Sigma$, and
$\Sigma^{\ast} := \bigcup_{1\le n<\infty}\Sigma^n$.
Given $\textbf{i}=i_1i_2\cdots i_n\in\Sigma^n$, denote $\textbf{i}|_k=i_1i_2\cdots i_k$ for each $k\le n$,
and $\sigma(\textbf{i})=\sum_{j=1}^n i_j$. Moreover, we define $\mathcal{L}^{\{\textbf{i}\}}(t):=\mathcal{L}^{\{i_n\}}(\mathcal{L}^{\{\textbf{i}|_{n-1}\}}(t))$ by recursion.

\vspace{0.2cm}

\begin{prop}\label{prop-gabor-suff-e-set}
Let $a,c$ and $\mathcal{M}$ be as in Theorem \ref{thm-gabor-eq-e-set}.
If $\mathcal{G}(H_c;a)$ is not a Gabor frame, then for each $N\in\Bbb N$, there exists $t\in \Bbb T_{[0,1)}$ such that
$$
\mathcal{L}^{\{\textbf{i}\}}(t)\in \Bbb T_{[0,1)}\setminus \big(\textbf{H}\cup (1-\widetilde{\textbf{H}})\big)  \mbox{ for all }\  \textbf{i}\in \Sigma^N.
$$
\end{prop}
\proof
By Corollary \ref{bound-sequence-cor}, there exist {a} $t_0\in [0,a)$ and a nonzero bounded sequence $\{q_j\}_{j\in\Bbb Z}$  
such that
\begin{equation}\label{eq-linear-cor-19}
\sum_{j\in\Bbb Z} q_j H_c(t_0+j-na)=0\  \mbox{ for all }\ n\in\Bbb Z.
\end{equation}
Without loss of generality, we may assume  $\|\{q_j\}_{j\in\Bbb Z}\|_{\ell^\infty}=\sup\{q_j: j\in\Bbb Z\}=1$.

\vspace{0.2cm}

For each $N\in\Bbb N$, let $\varepsilon=2^{-N-1}$
and $k\in\Bbb Z$ be a number such that $q_k>1-\varepsilon$. Then {it follows from \eqref{eq-linear-cor-19} that}
\begin{equation}\label{eq-linear-cor-102}
\sum_{j\in\Bbb Z} q_{k+j} H_c((t_0+k)+j-na)=0\  \mbox{ for all }\ n\in\Bbb Z.
\end{equation}
Set $x=\langle \frac{t_0+k}a\rangle$, and then take $n=\big\lfloor \frac{t_0+k}a\big\rfloor$ and $\big\lfloor \frac{t_0+k}a\big\rfloor+1$ respectively in \eqref{eq-linear-cor-102}.
We obtain
\begin{equation}\label{eq-linear-cor-103}
\sum_{j\in[-a x,-a x+c) } q_{k+j} - \sum_{j\in[-ax-c,-a x) } q_{k+j}=0
\end{equation}
and
\begin{equation}\label{eq-linear-cor-104}
\sum_{j\in[-a x+a,-a x+a+c) } q_{k+j} - \sum_{j\in[-a x+a-c,-a x+a) } q_{k+j}=0.
\end{equation}
Comparing \eqref{eq-linear-cor-103} with \eqref{eq-linear-cor-104},
\begin{equation}\nonumber
2q_{k}-\sum_{j\in[-a x+c, -a x+a+c)}q_{k+j} -\sum_{j\in[-a x-c, -a x+a-c)}q_{k+j}=0.
\end{equation}
This, together with the assumption of $q_k$, implies that there exist unique integers
\begin{equation}\label{eq-linear-cor-1004}
j_{\{1\}}\in[-a x+c, -a x+a+c)\ \mbox{ and }\ j_{\{-1\}}\in[-a x-c, -a x+a-c)
\end{equation}
 such that
\begin{equation}\label{eq-linear-cor-6}
q_{k+j_{\{1\}}},q_{k+j_{\{-1\}}} >1-2\varepsilon.
\end{equation}
Obviously, \eqref{eq-linear-cor-1004} implies $x\in \Bbb T_{[0,1)}\setminus \big(\textbf{H}\cup (1-\widetilde{\textbf{H}})\big)$, and
\begin{equation}\nonumber
j_{\{1\}}=\left\{\begin{array} {lll}
\lfloor c\rfloor &
{\rm if} \ x\in [x_2, 1) \\
\lfloor c\rfloor+1 &
{\rm if} \ x\in [0,x_1)
\end{array}\right. \
\mbox{ and } \
j_{\{-1\}}=\left\{\begin{array} {lll}
-\lfloor c\rfloor &
{\rm if} \ 1-x\in (x_2,1] \\
-\lfloor c\rfloor-1 &
{\rm if} \ 1-x\in (0,x_1]
\end{array}\right..
\end{equation}
Set $k^{\{i\}}=k+j_{\{i\}}$ and $x^{\{i\}}=\langle \frac{t_0+k^{\{i\}}}a\rangle$, $i\in\Sigma$.
Then $x^{\{i\}}=\mathcal{L}^{\{i\}}(x)$.

Recall {that} \eqref{eq-linear-cor-6}. Then replacing $k$ by $k^{\{i\}}$ in \eqref {eq-linear-cor-102}, and
repeating the procedure, we have $x^{\{i\}}\in \Bbb T_{[0,1)}\setminus \big(\textbf{H}\cup (1-\widetilde{\textbf{H}})\big)$, $i\in \Sigma$. Consequently, by repeating the procedure $N$ times, we obtain
$$x^{\{\textbf{i}\}}=\mathcal{L}^{\{\textbf{i}\}}(x)\in \Bbb T_{[0,1)}\setminus \big(\textbf{H}\cup (1-\widetilde{\textbf{H}})\big) \mbox{ for all }\ \textbf{i}\in \Sigma^N.$$
 This proves the proposition.
\eproof

\vspace{0.2cm}

Proposition \ref{prop-gabor-suff-e-set} provides the ingredient for the sufficiency of Theorem \ref{thm-gabor-eq-e-set}.  Now, we consider the converse part and focus our attention on the situation that $\mathcal{E}\ne\emptyset$.

\vspace{0.2cm}

For each $t\in\mathcal{E}$, by Corollary \ref{thm-exist-e-2},
\begin{equation}\nonumber
\mathcal{L}^{\{-1\}}(\mathcal{L}^{\{1\}}(t))=\mathcal{L}^{\{1\}}(\mathcal{L}^{\{-1\}}(t))=t.
\end{equation}
Observe that $\mathcal{L}^{\{1\}}$ and $\mathcal{L}^{\{-1\}}$ are the bijections from $\mathcal{E}$ onto itself. We may denote $\mathcal{L}(t)=\mathcal{L}^{\{1\}}(t)$, and then $\mathcal{L}^{\{-1\}}(t)=\mathcal{L}^{-1}(t)$.    Consequently,
\begin{equation}\nonumber
\mathcal{L}^{\{\textbf{i}\}}(t)=\mathcal{L}^{\sigma(\textbf{i})}(t),\ \ t\in\mathcal{E},
\end{equation}
for each word $\textbf{i}\in\Sigma^\ast$.
%
%
 Set
\begin{equation}\label{eq-lambda.def}
\lambda(t)=\left\{\begin{array} {lll}
\lfloor c\rfloor &
{\rm if} \ t\in [x_2, 1)\cap \mathcal{E}, \\
\lfloor c\rfloor+1 &
{\rm if} \ t\in [0,x_1)\cap \mathcal{E},
\end{array}\right.
\end{equation}
and define the sequence $\{\lambda_n(t)\}_{n\in\Bbb Z}$ by $\lambda_0(t)=0$,
and
\begin{equation}\label{eq-lambda-sequence.def}
\lambda_k(t)=\sum_{j=0}^{k-1}\lambda(\mathcal{L}^j(t))\ \mbox{ and }\  \lambda_{-k}(t)=-\sum_{j=-k}^{-1}\lambda(\mathcal{L}^j(t)), \ k\in\Bbb N.
\end{equation}
\vspace{0.2cm}
{Thus,}
\begin{equation}\label{eq-lambda-seque-1}
\lambda_{n+1}(t)=\lambda_{n}(t)+\lambda(\mathcal{L}^n(t)),\ n\in\Bbb Z,
\end{equation}
and
\begin{equation}\label{eq-lambda-seque-2}
\mathcal{L}^{\{\textbf{i}\}}(t)=\mathcal{L}^{\sigma(\textbf{i})}(t)=\Big\langle t+\frac{\lambda_{\sigma(\textbf{i})}(t)}a\Big\rangle,\ \textbf{i}\in\Sigma^\ast.
\end{equation}

\vspace{0.2cm}

By virtue of Theorem \ref{thm-exist-e-1}, {put} $Y(\alpha)=\frac pM$ be the simple fraction for some integers $0\le p<M$. {By} Proposition \ref{thm-exist-e++}, $\mathcal{E}=\cup_{k=0}^{M-1} I_k$ and $\mathcal{L}$ is an isometric map from $I_k$ to $I_{k+p}$, where we denote $I_m=I_n$ if $m\equiv n\ (\mbox{mod}\ M)$. Thus, $\lambda_n(t_1)=\lambda_n(t_2)$, $n\in\Bbb Z$, when $t_1,t_2\in I_k$ for some $k\in\Bbb Z$.


\begin{lem}\label{lem-lambad-sym-e}
Suppose $\mathcal{E}\ne\emptyset$, then for each $t\in \mathcal{E}\cap \widetilde{\mathcal{E}}$,
$\lambda_n(t)=-\lambda_{-n}(1-t)$ for all $ n\in\Bbb Z$.
\end{lem}
\proof
For the cases $\mathcal{S}\subset V$ and $\mathcal{S}\subset U$, we have $\lambda(t)=\lfloor c\rfloor$ and $\lfloor c\rfloor+1$ respectively for all $t\in \mathcal{E}$, and hence we get the desired result by \eqref{eq-lambda-sequence.def}.
Therefore, we may assume $\mathcal{S}\cap U\ne\emptyset $ and $\mathcal{S}\cap V\ne\emptyset$. {By} Proposition \ref{thm-exist-e++} (precisely, \eqref{alpha-E-U-V} and \eqref{E-structu}), 
 it is easy to see that
\begin{equation}\label{lem-lambad-sym-e-2401}
U\cap\mathcal{E}=\cup_{k=0}^{M-p-1}I_k,\ \ V\cap\mathcal{E}=\cup_{k=M-p}^{M-1}I_k
\end{equation} 
 and
\begin{equation}\label{lem-lambad-sym-e-24021}
I_k=1-\widetilde{I}_{M-k-1}\ \mbox{ for all }\ 0\le k<M.
\end{equation}

\vspace{0.2cm}

%
%

Now assume $t\in\mathcal{E}\cap \widetilde{\mathcal{E}}$. If $t\in U$, 
then by \eqref{lem-lambad-sym-e-2401}, there exists $k\in \{0,1,\ldots, M-p-1\}$ such that $t\in I_k\cap\widetilde{I_k}$, and then by \eqref{lem-lambad-sym-e-24021}, $1-t\in I_{M-1-k}$, which implies $\mathcal{L}^{-1}(1-t)\in I_{M-1-k-p}\subset U$.
Similarly, if $t\in V$, 
then there exists $k\in \{M-p,\ldots,M-1\}$ such that $t\in I_k\cap\widetilde{I_k}$, and then 
$\mathcal{L}^{-1}(1-t)\in I_{2M-1-k-p}\subset V$.  Consequently,
\begin{equation}\label{eq-lambda-sym-1}
\lambda(t)=\lambda\big(\mathcal{L}^{-1}(1-t)\big),\ \  t\in \mathcal{E}\cap \widetilde{\mathcal{E}}.
\end{equation}

 Observe that for each $t\in \mathcal{E}\cap \widetilde{\mathcal{E}}$,
$$
\mathcal{L}^{-j-1}(1-t)=\mathcal{L}^{-j}(1-\mathcal{L}(t))=\cdots=\mathcal{L}^{-1}(1-\mathcal{L}^j(t)),\  j\ge 0.
$$
{By} \eqref{eq-lambda-sym-1},
$$
\lambda(\mathcal{L}^{-j-1}(1-t))=\lambda(\mathcal{L}^{-1}(1-\mathcal{L}^j(t)))=\lambda( \mathcal{L}^j(t)) ,\  j\ge 0.
$$
This together with \eqref{eq-lambda-sequence.def} implies $\lambda_n(t)=-\lambda_{-n}(1-t)$ for all $n\ge0$. {For} the case $n<0$, the statement is obvious through replacing $t$ by $1-t$.
\eproof

\vspace{0.2cm}

\begin{prop}\label{prop-lin-e-set}
Let $a,c$ and $\mathcal{M}$ be as in Theorem \ref{thm-gabor-eq-e-set}.
If $\mathcal{E}\ne\emptyset$, then for each $t\in\mathcal{E}$,
\begin{equation}\label{eq-lin-e-set-1}
\sum_{j\in\Bbb Z} \chi_{[0,c)}(ta+\lambda_j(t)-na)=1
\end{equation}
and
\begin{equation}\label{eq-lin-e-set-2}
\sum_{j\in\Bbb Z} \chi_{[-c,0)}(ta+\lambda_j(t)-na)=1
\end{equation}
for all $ n\in\Bbb Z$.
\end{prop}

\proof
First, we prove \eqref{eq-lin-e-set-1} is valid for all $t\in\mathcal{E}\cap \widetilde{\mathcal{E}}$.
For each $n\in\Bbb Z$, let $\ell\in\Bbb Z$ be the smallest number such that $ta+\lambda_\ell(t)\ge na$. Then
\begin{equation}\label{eq-lam-uni-0001}
ta+\lambda_{\ell-1}(t)< na.
\end{equation}

If $\mathcal{L}^{\ell-1}(t)\in [x_2,1)$, {by} \eqref{eq-lambda-seque-1}, \eqref{eq-lambda.def} and \eqref{eq-lam-uni-0001},
\begin{equation}\label{eq-lam-uni}
ta+\lambda_\ell(t)=ta+\lambda_{\ell-1}(t)+\lambda(\mathcal{L}^{\ell-1}(t))<na+\lfloor c\rfloor\le na+c.
\end{equation}
If $\mathcal{L}^{\ell-1}(t)\in [0,x_1)$, then $x_1=\frac {a+\langle c\rangle-1}{a}$ by \eqref{alpha-xx}, and then {it follows from}  \eqref{eq-lambda-seque-2} and \eqref{eq-lam-uni-0001} {that}
 $$
 ta+\lambda_{\ell-1}(t)=\big\lfloor t+\frac{\lambda_{\ell-1}(t)}a\big\rfloor a+\mathcal{L}^{\ell-1}(t) a< (n-1)a+x_1a=na-1+\langle c\rangle,
 $$
and therefore {\eqref{eq-lambda.def} implies that }
\begin{equation}\label{eq-lam-uni-2}
ta+\lambda_\ell(t)=ta+\lambda_{\ell-1}(t)+\lambda(\mathcal{L}^{\ell-1}(t))<na-1+\langle c\rangle+\lfloor c+1\rfloor=na+c.
\end{equation}
Combining \eqref{eq-lam-uni} with \eqref{eq-lam-uni-2}, {by}  the assumption of $\ell$, we have
\begin{equation}\label{eq-lam-uni-0}
ta+\lambda_\ell(t)\in[na,na+c).
\end{equation}

Similarly, if $\mathcal{L}^{\ell}(t)\in [0,x_1)$, {by}  \eqref{eq-lambda-seque-1}, \eqref{eq-lam-uni-0} and \eqref{eq-lambda.def},
\begin{equation}\label{eq-lam-uni-3}
ta+\lambda_{\ell+1}(t)=ta+\lambda_{\ell}(t)+\lambda(\mathcal{L}^{\ell}(t))\ge na+\lfloor c\rfloor+1>na+c.
\end{equation}
If $\mathcal{L}^{\ell}(t)\in [x_2,1)$, in view of $x_2=\frac {\langle c\rangle}{a}$ by \eqref{alpha-xx}, {by}  \eqref{eq-lambda-seque-2} and \eqref{eq-lam-uni-0},
$$
ta+\lambda_{\ell}(t)=\big\lfloor\frac {ta+\lambda_{\ell}(t)}a\big\rfloor a + \mathcal{L}^{\ell}(t) a \ge na+x_2a=na+\langle c\rangle,
$$
and therefore by \eqref{eq-lambda.def},
\begin{equation}\label{eq-lam-uni-4}
ta+\lambda_{\ell+1}(t)=ta+\lambda_{\ell}(t)+\lambda(\mathcal{L}^{\ell}(t))\ge na+\langle c\rangle+\lfloor c\rfloor=na+c.
\end{equation}
Combining \eqref{eq-lam-uni-3} with \eqref{eq-lam-uni-4}, we have
\begin{equation}\nonumber
ta+\lambda_{\ell+1}(t)\not\in[na,na+c).
\end{equation}
This, together with \eqref{eq-lam-uni-0001} and \eqref {eq-lam-uni-0}, implies $\ell\in\Bbb Z$ is the only number such that $ta+\lambda_\ell(t)-na\in [0,c)$. Therefore, \eqref{eq-lin-e-set-1} is valid for all $t\in\mathcal{E}\cap \widetilde{\mathcal{E}}$.

\vspace{0.2cm}
\vspace{0.2cm}

Next, we prove \eqref{eq-lin-e-set-2} is valid for all $t\in\mathcal{E}\cap \widetilde{\mathcal{E}}$. Observe that $1-t \in\mathcal{E}\cap \widetilde{\mathcal{E}}$. Then replacing $t$ by $1-t$ in {the} above procedure, we have  {that,} for each $n\in\Bbb Z$, there exists  unique $\ell'\in\Bbb Z$ such that
$$
a(1-t)+\lambda_{\ell'}(1-t)-na\in [0,c).
$$
Therefore, by Lemma \ref{lem-lambad-sym-e}, $\ell'\in\Bbb Z$ is the only number such that
\begin{equation}\nonumber
at+\lambda_{-\ell'}(t)+(n-1)a\in (-c,0].
\end{equation}
By \eqref{eq-lambda-seque-2}, we have $\langle t+\frac{\lambda_{-\ell'}(t)}a \rangle =\mathcal{L}^{-\ell'}(t)\in \widetilde{\mathcal{E}}\cap\mathcal{E}$.
Recall that $0\not\in\widetilde{\mathcal{E}}$, and $\langle\frac ca\rangle=\mathcal{M}(x_2)\not\in\widetilde{\mathcal{E}}$ when $x_2\ne 1$, and $\langle\frac ca\rangle=\alpha\not\in\mathcal{E}$ when $x_2= 1$. Then $\langle t+\frac{\lambda_{-\ell'}(t)}a \rangle\not\in \{0,\langle\frac ca\rangle, 1-\langle\frac ca\rangle\}$.
%
Hence, $\ell'\in\Bbb Z$ is the unique number such that
$at+\lambda_{-\ell'}(t)+(n-1)a\in [-c,0)$.  Therefore, \eqref{eq-lin-e-set-2} is valid for $t\in\mathcal{E}\cap \widetilde{\mathcal{E}}$.

\vspace{0.2cm}
\vspace{0.2cm}


Now we prove \eqref{eq-lin-e-set-1} and \eqref{eq-lin-e-set-2} are valid for  $t\in \mathcal{E}$.  Fix $0\le k\le M-1$. For each $j\in\Bbb Z$, $\lambda_j(t)$ {is} the same for all
$t\in I_k$ by \eqref{eq-lambda.def} and \eqref{eq-lambda-sequence.def}. Then
the left hand sides of \eqref{eq-lin-e-set-1} and \eqref{eq-lin-e-set-2} are right continuous functions for $t\in I_k$. Therefore, \eqref{eq-lin-e-set-1} and \eqref{eq-lin-e-set-2} are valid for $t\in I_k$ since they are valid for $t\in I_k\cap \widetilde{I_k}\subset \mathcal{E}\cap \widetilde{\mathcal{E}}$. We conclude the proof {by} using the fact $\mathcal{E}=\cup_{k=0}^{M-1}I_k$.
\eproof

\vspace{0.2cm}

Now, we start to prove Theorem \ref{thm-gabor-eq-e-set}.

\vspace{0.2cm}

\noindent{\bf Proof of Theorem \ref{thm-gabor-eq-e-set}.} Recall that $H_c(x)=\chi_{[0,c)}(x)-\chi_{[-c,0)}(x)$.
The necessity is a consequence of Proposition \ref{prop-lin-e-set} and Corollary \ref{bound-sequence-cor}. 
Now we prove the sufficiency.

\vspace{0.2cm}

Assume that $\mathcal{G}(H_c;a)$ is not a Gabor frame. {By} Proposition \ref{prop-gabor-suff-e-set} and Theorem \ref{thm-s-1}, we have $\mathcal{S}\ne\emptyset$.
Furthermore, it can be seen from Proposition \ref{prop-gabor-suff-e-set} and Theorem \ref{thm-exist-e-1} that $Y(\alpha)\in \Bbb Q$.
 This together with Lemma \ref{lem-Y-map}  implies $a\in \Bbb Q$.
Denote $a=p/q$ for some co-prime positive integers $1\le p< q$.
{By} \eqref{alpha-x-x-1} and \eqref{rabcnewplus.def-2},
$p(t-\mathcal{L}^{\{\textbf{i}\}}(t))\in \Bbb Z$ for each $t\in \Bbb T_{[0,1)}$ and $\textbf{i}\in\Sigma^\ast$. Therefore,
$K_t=\{\mathcal{L}^{\{\textbf{i}\}}(t): \textbf{i}\in\Sigma^\ast \}$ is a subset of $\Bbb T_{[0,1)}$ with cardinality $\#(K_t)\le p$.
Set $K_{t,n}=\{\mathcal{L}^{\{\textbf{i}\}}(t): \textbf{i}\in\Sigma^j, j\le n \}$. Then $K_{t,n}\subset K_{t,n+1}\subset K_{t}$ for all $n\in\Bbb N$.

\vspace{0.2cm}

By Proposition \ref{prop-gabor-suff-e-set}, there exists $t_0\in\Bbb T_{[0,1)} $ such that
\begin{equation}\label{eq-Kset-cardinality}
K_{t_0,p+1}\subset \Bbb T_{[0,1)}\setminus \big(\textbf{H}\cup (1-\widetilde{\textbf{H}})\big).
\end{equation}
{By} the pigeonhole principle, there exists $k\le p$ such that $K_{t_0,k}= K_{t_0,k+1}$. Thus, we have
$$
\mathcal {L}^{\{i\}}(K_{t_0,k})\subset K_{t_0,k}, \ \ i\in\Sigma.
$$
For each $i\in\Sigma$, $\mathcal {L}^{\{i\}}(t)$ is an injection for $t\in \Bbb T_{[0,1)}\setminus \big(\textbf{H}\cup (1-\widetilde{\textbf{H}})\big)$. {By} \eqref{eq-Kset-cardinality}, $\#(\mathcal {L}^{\{i\}}(K_{t_0,k}))=\#(K_{t_0,k})$, and therefore, $\mathcal {L}^{\{i\}}(K_{t_0,k})= K_{t_0,k}$.
This implies {that} $K_{t_0,k}\subset \mathcal{E}$, and consequently $\mathcal{E}\ne\emptyset$, the sufficiency is proved,  and thus the proof of Theorem \ref{thm-gabor-eq-e-set} is completed.
\eproof

\vspace{0.2cm}

\section{Appendix:  Proof of Theorem \ref{main0} }\label{section-7}

For each $x\in\Bbb R$ and $q\in \Bbb N$, denote $\lfloor x\rfloor_q= {\lfloor qx\rfloor}/q$ and $\langle x\rangle_q=x-\lfloor x\rfloor_q$. {For  $u,v >0$, we define}
\begin{equation}\nonumber%
H_{u,v}:=-\chi_{[-u,0)}+\chi_{[0,v)}.
\end{equation}

Assume that $a= p/q$ for some co-prime integers $1\le p\le q$. {For} each $x\in \Bbb R$ and $\{q_n\}_{n\in\Bbb Z}\in\ell^\infty$, the sequence 
$\{\sum_{n\in\Bbb Z}q_nf(x+n-aj)\}_{j\in\Bbb Z}$
 is determined only by the values of $f$ at $x+\Bbb Z/q$.
 From this observation, we have the following lemma.
\vspace{0.2cm}

\begin{lem}\label{H-U-V}
Let $a= p/q$ with co-prime integers $1\le p\le q$, and $u,v \in \Bbb N/q$. Then $\mathcal{G}(H_{u,v};a)$ is not a Gabor frame if and only if there exist  $x\in \{ j/q: j=0,1,\dots,p-1\}$ and a nonzero bounded sequence $\{q_j\}_{j\in\Bbb Z}\in\ell^\infty$ 
such that
\begin{equation}\label{H-U-V-1}
\sum_{j\in\Bbb Z} q_j H_{u,v}(x+j-na)=0\  \mbox{ for all }\ n\in\Bbb Z.
\end{equation}
\end{lem}
\proof
Apparently, the sufficiency is a consequence of Theorem \ref{bound-sequence-0} and we only need to prove the necessity. Assume that $\mathcal{G}(H_{u,v};a)$ is not a Gabor frame, it follows from Theorem \ref{bound-sequence-0} that there  exist  $t_0\in\Bbb R$ and a nonzero bounded sequence $\{q_j\}_{j\in\Bbb Z}\in \ell^\infty$ such that
\begin{equation}\nonumber
\sum_{j\in\Bbb Z} q_j g(t_0+j-na)=0\  \mbox{ for all }\ n\in\Bbb Z,
\end{equation}
for either $g(x)=H_{u,v}^+(x)$ or $g(x)=H_{u,v}^-(x)$. Observe that $H_{u,v}^+(x)=H_{u,v}(\lfloor x\rfloor_q)$ for all $x\in\Bbb R$, and $H_{u,v}^-(x)=H_{u,v}(\lfloor x\rfloor_q)$ for all $x\not\in \Bbb Z/q$. Then \eqref{H-U-V-1} is valid for $x=\lfloor t_0\rfloor_q$ except the only case that
\begin{equation}\nonumber
g(x)=H_{u,v}^-(x)\ \mbox{ and }\ t_0\in\Bbb Z/q.
\end{equation}
In this case, observing that $H_{u,v}^-(t_0+j-na)=H_{u,v}(t_0+j-na- 1/{q})$ for all $n,j\in\Bbb Z$,  we just need to choose $x=t_0-1/{q}$ and then \eqref{H-U-V-1} holds. As \eqref{H-U-V-1} is $a-$shift invariant we may take $x\in \{ j/q:j=0,1,\dots,p-1\}$. This completes the proof of the necessity.
\eproof

\vspace{0.2cm}

\vspace{0.2cm}

\begin{prop} \label{rational-case-2}
Let $a= p/q$ for some co-prime integers $1\le p\le q$, and $c>0$. Then:
\begin{enumerate}
   \item [{\rm(1)}] {If $\langle c\rangle_q\in(0,\frac 1{2q})$,} then $\mathcal{G}(H_c;a)$ is a Gabor frame if and only if both $\mathcal{G}(H_{\lfloor c\rfloor_q};a)$ and $\mathcal{G}(H_{\lfloor c\rfloor_q,\lfloor c\rfloor_q+\frac 1q};a)$
       are Gabor frames.
   \item [{\rm(2)}] {If $\langle c\rangle_q = \frac 1{2q}$,} then $\mathcal{G}(H_c;a)$ is a Gabor frame if and only if  $\mathcal{G}(H_{\lfloor c\rfloor_q,\lfloor c\rfloor_q+\frac 1q};a)$
       is a Gabor frame.
   \item [{\rm(3)}] {If $\langle c\rangle_q\in(\frac 1{2q},\frac1q)$}, then $\mathcal{G}(H_c;a)$ is a Gabor frame if and only if both
   $\mathcal{G}(H_{\lfloor c\rfloor_q+\frac1q};a)$ and
   $\mathcal{G}(H_{\lfloor c\rfloor_q,\lfloor c\rfloor_q+\frac 1q};a)$
    are Gabor frames.
\end{enumerate}
\end{prop}

\proof
It is easy to check the following statements:
\begin{enumerate}
   \item [{(i)}] If $\langle c\rangle_q\le\langle x\rangle_q $, then $(-x+[0,c))\cap\Bbb Z/q=(-\lfloor x\rfloor_q+[0,\lfloor c\rfloor_q))\cap\Bbb Z/q$;
   \item [{(ii)}] If $\langle c\rangle_q >\langle x\rangle_q $, then $(-x+[0,c))\cap\Bbb Z/q=(-\lfloor x\rfloor_q+[0,\lfloor c\rfloor_q+\frac 1q))\cap\Bbb Z/q$;
   \item [{(iii)}] If $\langle c\rangle_q+\langle x\rangle_q <\frac1q$, then $(-x+[-c,0))\cap\Bbb Z/q=(-\lfloor x\rfloor_q+[-\lfloor c\rfloor_q,0))\cap\Bbb Z/q$;
   \item [{(iv)}] If $\langle c\rangle_q+\langle x\rangle_q \ge \frac1q$, then $(-x+[-c,0))\cap\Bbb Z/q=(-\lfloor x\rfloor_q+[-\lfloor c\rfloor_q-\frac 1q,0))\cap\Bbb Z/q$.
  \end{enumerate}


\vspace{0.2cm}
Combining with the corresponding statements above, we have the following assertions:

\vspace{0.2cm}
\begin{enumerate}
\item  [{(A)}] If $\langle c\rangle_q\le\langle x\rangle_q $ and $\langle c\rangle_q+\langle x\rangle_q <\frac1q$, then for any $n\in\Bbb Z$,
$$\sum_{j\in\Bbb Z}q_jH_c(x+j-an)=0\mbox{ if and only if } \sum_{j\in\Bbb Z}q_jH_{\lfloor c\rfloor_q}(\lfloor x\rfloor_q +j-an)=0;$$
\item  [{(B)}] If $\langle c\rangle_q\le\langle x\rangle_q $ and $\langle c\rangle_q+\langle x\rangle_q \ge \frac1q$, then for any $n\in\Bbb Z$,
$$\sum_{j\in\Bbb Z}q_jH_c(x+j-an)=0\mbox{ if and only if } \sum_{j\in\Bbb Z}q_jH_{\lfloor c\rfloor_q+\frac 1q,\lfloor c\rfloor_q}(\lfloor x\rfloor_q +j-an)=0;$$
\item  [{(C)}] If $\langle c\rangle_q >\langle x\rangle_q $ and $\langle c\rangle_q+\langle x\rangle_q <\frac1q$, then for any $n\in\Bbb Z$,
$$\sum_{j\in\Bbb Z}q_jH_c(x+j-an)=0\mbox{ if and only if } \sum_{j\in\Bbb Z}q_jH_{\lfloor c\rfloor_q,\lfloor c\rfloor_q+\frac1q}(\lfloor x\rfloor_q +j-an)=0;$$
\item  [{(D)}] If $\langle c\rangle_q >\langle x\rangle_q $ and $\langle c\rangle_q+\langle x\rangle_q \ge \frac1q$, then for any $n\in\Bbb Z$,
$$\sum_{j\in\Bbb Z}q_jH_c(x+j-an)=0\mbox{ if and only if } \sum_{j\in\Bbb Z}q_jH_{\lfloor c\rfloor_q+\frac1q}(\lfloor x\rfloor_q +j-an)=0.$$
  \end{enumerate}

Now let's prove the proposition item by item.

\vspace{0.2cm}

(1). \emph{Sufficiency}. Suppose that $\mathcal{G}(H_c;a)$ is not a Gabor frame. Using Corollary \ref{bound-sequence-cor}, there exist $x\in \Bbb R$ and a nonzero bounded sequence $\{q_j\}_{j\in\Bbb Z}\in\ell^\infty$ such that
$$
\sum_{j\in\Bbb Z} q_j H_c(x+j-na)=0\  \mbox{ for all }\ n\in\Bbb Z.
$$
{By assertions} (A), (B) and (C) and Lemma \ref{H-U-V}, at least one of
$\mathcal{G}(H_{\lfloor c\rfloor_q};a)$, $\mathcal{G}(H_{\lfloor c\rfloor_q,\lfloor c\rfloor_q+\frac 1q};a)$
and $\mathcal{G}(H_{\lfloor c\rfloor_q+\frac 1q,\lfloor c\rfloor_q};a)$ is not a Gabor frame. Observe that $\mathcal{G}(H_{\lfloor c\rfloor_q,\lfloor c\rfloor_q+\frac 1q};a)$
being a Gabor frame is equivalent to that for $\mathcal{G}(H_{\lfloor c\rfloor_q+\frac 1q,\lfloor c\rfloor_q};a)$. Thus, $\mathcal{G}(H_{\lfloor c\rfloor_q};a)$ and $\mathcal{G}(H_{\lfloor c\rfloor_q,\lfloor c\rfloor_q+\frac 1q};a)$ are not both the  Gabor frames.

\vspace{0.2cm}

\emph{Necessity}. If $\mathcal{G}(H_{\lfloor c\rfloor_q};a)$ is not a Gabor frame, it follows  immediately from Lemma \ref{H-U-V} that there exist $x\in \Bbb Z/q$ and a nonzero bounded sequence $\{q_j\}_{j\in\Bbb Z}\in\ell^\infty$ such that
\begin{equation}\label{prop-proof-1}
\sum_{j\in\Bbb Z} q_j H_{\lfloor c\rfloor_q}( x+j-na)=0\  \mbox{ for all }\ n\in\Bbb Z.
\end{equation}
Set $t= x+\frac1{2q}$. Observe that $\langle c\rangle_q\le\langle t\rangle_q $ and $\langle c\rangle_q+\langle t\rangle_q <\frac1q$. {By} \eqref{prop-proof-1} and {assertion} (A), one can get $\sum_{j\in\Bbb Z}q_jH_c(t+j-an)=0$ for all $n\in\Bbb Z$. Thus, $\mathcal{G}(H_c;a)$ is not a Gabor frame by Corollary \ref{bound-sequence-cor}.

\vspace{0.2cm}

On the other hand, if $\mathcal{G}(H_{\lfloor c\rfloor_q,\lfloor c\rfloor_q+\frac 1q};a)$ is not a Gabor frame,
{by}  Lemma \ref{H-U-V}, there exist $x\in\Bbb Z/q$ and a nonzero bounded sequence $\{q_j\}_{j\in\Bbb Z}\in\ell^\infty$ such that
$$
\sum_{j\in\Bbb Z} q_j H_{\lfloor c\rfloor_q,\lfloor c\rfloor_q+\frac 1q}( x+j-na)=0\  \mbox{ for all }\ n\in\Bbb Z,
$$
which together with {assertion} (C) implies
$ \sum_{j\in\Bbb Z} q_j H_c( x+j-na)=0\  \mbox{ for all }\ n\in\Bbb Z$.
Therefore, by Corollary \ref{bound-sequence-cor}, $\mathcal{G}(H_c;a)$ is not a Gabor frame. Thus, (1) is valid.

\vspace{0.2cm}

(2). According to {assertions} (B) and (C),  $\mathcal{G}(H_c;a)$ is a Gabor frame if and only if both $\mathcal{G}(H_{\lfloor c\rfloor_q,\lfloor c\rfloor_q+\frac 1q};a)$ and $\mathcal{G}(H_{\lfloor c\rfloor_q+\frac 1q,\lfloor c\rfloor_q};a)$ are Gabor frames by and Corollary \ref{bound-sequence-cor} and Lemma \ref{H-U-V}. Observe that $\mathcal{G}(H_{\lfloor c\rfloor_q,\lfloor c\rfloor_q+\frac 1q};a)$
being a Gabor frame is equivalent to that for $\mathcal{G}(H_{\lfloor c\rfloor_q+\frac 1q,\lfloor c\rfloor_q};a)$. Then (2) follows.

\vspace{0.2cm}

(3). The proof is similar as (1) but replacing {assertion} (A) by (D). We complete the proof of Proposition \ref{rational-case-2}.
\eproof

\vspace{0.2cm}
\vspace{0.2cm}

Now we start to prove Theorem \ref{main0}.

\vspace{0.2cm}

\noindent\textbf{Proof of Theorem \ref{main0}.}
(i). The conclusion follows from Corollary \ref{bound-sequence-cor} by taking $t_0=0$ and $q_j=1$ for all  $j\in\Bbb Z$ in \eqref{eq-linear-cor}.

\vspace{0.2cm}

(ii). For the case $ \langle c\rangle \ne \frac 12$, let $[c]$ be the integer such that $|c-[c]|<\frac 12$.
Recall that $\mathcal{G}(H_{[c]};a)$ is not a Gabor frame by (i).
 Then $\mathcal{G}(H_c;a)$ is not a Gabor frame by using Proposition \ref{rational-case-2} (1) and (3) with respect to the cases $c>[c]$ and $c<[c]$. 

\vspace{0.2cm}

Now we prove $\mathcal{G}(H_{c};a)$ is a Gabor frame when $ \langle c\rangle = \frac 12$. Suppose on the contrary that $\mathcal{G}(H_{c};a)$ is not a Gabor frame. {By} Proposition \ref{rational-case-2} (2), $\mathcal{G}(H_{\lfloor c\rfloor,\lfloor c\rfloor+1};a)$ is not a Gabor frame, which together with Lemma \ref{H-U-V} implies that there exists a nonzero bounded sequence
$\{q_j\}_{j\in\Bbb Z}\in\ell^\infty$ such that
$$
 \sum_{j\in\Bbb Z}q_jH_{\lfloor c\rfloor,\lfloor c\rfloor+1}(j-n)=0 \ \mbox{ for all } \ n\in\Bbb Z.
$$
This means
\begin{equation}\label{a=1.1}
\sum_{j=1}^{\lfloor c\rfloor}q_{n-j}=\sum_{j=0}^{\lfloor c\rfloor} q_{n+j}\ \mbox{ for all }  \ n\in\Bbb Z.
\end{equation}
Comparing \eqref{a=1.1} by taking $n=k$ and $k+1$, we have $2q_k=q_{k-\lfloor c\rfloor}+q_{k+\lfloor c\rfloor+1}$. Iterating this equation once more, we obtain
\begin{equation}\label{a=1.2}
4q_k=q_{k-2\lfloor c\rfloor}+2q_{k+1}+q_{k+2\lfloor c\rfloor+2} \mbox{ for all } k\in\Bbb Z.
\end{equation}

\vspace{0.2cm}

Let $M=\|\{q_j\}_{j\in\Bbb Z}\|_\infty$. Without loss of generality, we assume $M=\sup_{j\in\Bbb Z} q_j$ since otherwise, we may replace $q_j$ by $-q_j$, $j\in\Bbb Z$.
Set $\varepsilon=2^{-\lfloor c\rfloor-2}M$. Then there exists  $n_0\in\Bbb Z$ such that $q_{n_0}>M-\varepsilon$. This together with \eqref{a=1.2} implies $q_{n_0+1}>M-2\varepsilon$. Replacing $\varepsilon$ and $n_0$ by $2\varepsilon$ and $n_0+1$ respectively,  we can repeat this procedure $j$ times to obtain $q_{n_0+j}>M-2^j\varepsilon$ for all $j\in\Bbb N$. Therefore,
$$
\sum_{j=1}^{\lfloor c\rfloor}q_{n_0-j}\le M {\lfloor c\rfloor} \le M( {\lfloor c\rfloor}+1)-2^{\lfloor c\rfloor+1}\varepsilon < \sum_{j=0}^{\lfloor c\rfloor} q_{n_0+j}.
$$
This contradicts to \eqref{a=1.1}. Hence, $\mathcal{G}(H_{c};a)$ is a Gabor frame when $ \langle c\rangle = \frac 12$. We {have} proved (ii).

\vspace{0.2cm}

(iii). 
Assume that
there exist $x\in\Bbb R$ and $\{q_j\}_{j\in\Bbb Z}\in \ell^\infty
$ such that
\begin{equation}\label{proof-th1.1-1}
\sum_{j\in\Bbb Z}q_jH_{c}(x+j-na)=0 \mbox{ for all } n\in\Bbb Z.
\end{equation}


For each $j\in\Bbb Z$, let $n_1,n_2\in\Bbb Z$ be the smallest/largest integer such that $n_1>\frac{x+j+1-c}a$
and $n_2\le \frac{x+j-1+c}a$ respectively.
{By} \eqref{proof-th1.1-1}, $q_{j+1}H_{c}(x+j+1-n_1a)+q_{j}H_{c}(x+j-n_1a)=0$ and $q_{j-1}H_{c}(x+j-1-n_2a)+q_{j}H_{c}(x+j-n_2a)=0$. This  means that $q_j=0$ will lead to $q_{j-1}=q_{j+1}=0$, and therefore $q_k=0$ for all $k\in\Bbb Z$ by iteration.

\vspace{0.2cm}

Note that $a\not\in \Bbb Q$. There exists $n_0\in \Bbb Z$ such that $\langle x- n_0a\rangle < \min(c,1-c)$. Then $j_0=-\lfloor x -n_0a \rfloor$ is the only integer such that $x+j_0-n_0a\in [-c,c)$. Thus $q_{j_0}=0$  by \eqref {proof-th1.1-1}, and therefore, $q_j=0$ for all $j\in\Bbb Z$ by the argument above. Hence $\mathcal{G}(H_c;a)$ is a Gabor frame by Corollary \ref{bound-sequence-cor}.
We {have} proved (iii).

\vspace{0.2cm}

(iv). If $1-\frac 1{2q}<c< 1$, then $\mathcal{G}(H_c;a)$ is not a Gabor frame by Proposition \ref{rational-case-2} (3) and the observation that $\mathcal{G}(H;a)$ is not a Gabor frame by (i).



\vspace{0.2cm}

Now assume $c\le 1-\frac 1{2q}$. By the same argument in (iii), we only need to find an integer $j_0\in\Bbb Z$ such that $q_{j_0}=0$ when \eqref{proof-th1.1-1} is valid for some $x\in\Bbb R$ and $\{q_j\}_{j\in\Bbb Z}\in \ell^\infty$.
%
%
Let $n_0\in\Bbb Z$ be an integer such that $x- n_0a\in [-\frac 1{2q},\frac1{2q})+\Bbb Z$, and  $j_0\in\Bbb Z$ be the number such that $x+j_0-n_0a\in [-\frac 1{2q},\frac1{2q})$. Then $j_0$ is the only integer satisfying $x+j_0-n_0a\in[-c,c)$, and $q_{j_0}=0$ by \eqref {proof-th1.1-1}.
This proves (iv), and the proof of Theorem \ref{main0} is completed.
\eproof

\end{document}